\documentclass[11pt,a4paper, reqno]{amsart}

\usepackage{a4wide}

\usepackage{amsfonts,amsthm,amssymb,amsmath,amscd,amsopn,mathabx}
\usepackage[pdftex]{graphicx}
\usepackage{graphics}
\usepackage[matrix,arrow]{xy}
\usepackage[active]{srcltx}
\usepackage[draft]{todonotes}

\usepackage[colorlinks=true]{hyperref}
\hypersetup{
colorlinks   = true,
citecolor    = blue
}

\usepackage{enumitem, hyperref}\hypersetup{colorlinks}

\makeatletter
\def\reflb#1#2{\begingroup
    #2%
    \def\@currentlabel{#2}%
    \phantomsection\label{#1}\endgroup
}
\makeatother

\usepackage{color}

\definecolor{darkred}{rgb}{1,0,0} 
\definecolor{darkgreen}{rgb}{0,0.8,0}
\definecolor{darkblue}{rgb}{0,0,1}

\hypersetup{colorlinks,
linkcolor=darkblue,
filecolor=darkgreen,
urlcolor=darkred,
citecolor=darkblue}

\newtheorem{thm}{Theorem}
\numberwithin{thm}{section}
\numberwithin{equation}{section}
\newtheorem{theorem}[thm]{Theorem}
\newtheorem*{theorem*}{Theorem}

\newtheorem{corollary}[thm]{Corollary}
\newtheorem*{corollary*}{Corollary}
\newtheorem{lemma}[thm]{Lemma}

\newtheorem{proposition}[thm]{Proposition}

\newtheorem*{conjecture*}{Conjecture}

\newtheorem*{question*}{Question}

\newtheorem*{definition*}{Definition}

\newtheorem*{definitions*}{Definitions}

\newtheorem*{rem*}{Remark}
\theoremstyle{remark}
\newtheorem{remark}[thm]{Remark}

\newtheorem*{remark*}{Remark}

\newtheorem*{remarks*}{Remarks}
\newtheorem*{example*}{Example}
\newtheorem{example}[thm]{Example}
\newtheorem*{examples*}{Examples}

\newcommand{\R}{\mathbb{R}}
\newcommand{\Z}{\mathbb{Z}}
\newcommand{\Q}{\mathbb{Q}}
\newcommand{\C}{\mathbb{C}}
\newcommand{\N}{\mathbb{N}}

\def\CP{{\mathbb C}P}
\def\RP{{\mathbb R}P}

\newcommand{\id}{\mathit{id}}
\newcommand{\Id}{\mathit{Id}}

\newcommand{\tspan}{\text{span}}
\newcommand{\Hess}{\text{Hess}\,}

\newcommand{\ep}{\epsilon}
\newcommand{\ga}{\gamma}
\newcommand{\Ga}{\Gamma}
\newcommand{\al}{\alpha}
\newcommand{\vr}{\varphi}
\newcommand{\om}{\omega}

\newcommand{\Cont}{Cont}

\newcommand{\Sp}{\mathrm{Sp}}

\def\maslov{{\mu_\text{Maslov}}}
\def\cz{{\mu}}
\def\P{{\mathcal P}}

\def\mi{{\widehat\mu}}

\def\HC{{\mathrm{HC}}}
\def\SH{{\mathrm{SH}}}
\def\CC{{\mathrm{CC}}}
\def\HF{{\mathrm{HF}}}
\def\H{{\mathrm{H}}}

\def\Bott{{\mathcal B}}
\def\bga{{\bar\gamma}}

\def\S{{\mathcal S}^{\epsilon'}}

\def\cz{{\mu}}

\def\bG{{\bar G}}

\def\S{{\Sigma}}

\def\lg{\langle}
\def\rg{\rangle}
\def\hga{{\hat\gamma}}
\def\L{L^{2n+1}_p(\ell_0,\dots,\ell_n)}
\def\Lone{L^{2n+1}_p(1,\dots,1)}
\def\Lpone{L^{2n+1}_p(1,\dots,1;\varphi)}
\def\xis{\xi_{\text{std}}}
\def\opn{{\oplus_1^N}}
\def\db{{(\Lambda^n\xi)^{\otimes N}}}
\def\modp{{\,(\text{mod}\ p)}}

\def\bell{{\bar\ell}}
\def\hell{{\hat\ell}}
\def\tmu{{\tilde\mu}}
\def\tnu{{\tilde\nu}}
\def\th{{\tilde h}}
\def\bG{{\bar G}}
\def\bT{{\bar T}}
\def\balpha{{\bar\alpha}}
\def\oN{{\oplus N}}
\def\sec{{\mathfrak s}}

\def\maslov{{\mu_\text{Maslov}}}
\def\A{{\mathcal{A}}}
\def\morse{{\mu_\text{Morse}}}
\def\Nn{{\mathcal{N}}}

\begin{document}

\title[Dynamical implications of convexity beyond dynamical convexity]{Dynamical implications of convexity\\ beyond dynamical convexity}

\author[Miguel Abreu]{Miguel Abreu}
\author[Leonardo Macarini]{Leonardo Macarini}

\address{Center for Mathematical Analysis, Geometry and Dynamical Systems,
Instituto Superior T\'ecnico, Universidade de Lisboa, 
Av. Rovisco Pais, 1049-001 Lisboa, Portugal}
\email{mabreu@math.tecnico.ulisboa.pt}

\address{Center for Mathematical Analysis, Geometry and Dynamical Systems,
Instituto Superior T\'ecnico, Universidade de Lisboa, 
Av. Rovisco Pais, 1049-001 Lisboa, Portugal}
\email{macarini@math.tecnico.ulisboa.pt}

\subjclass[2010]{53D40, 37J10, 37J55} \keywords{Closed orbits,
Conley-Zehnder index, Reeb flows, dynamical convexity, equivariant symplectic homology}


\thanks{The authors were partially supported by FCT/Portugal through projects UID/MAT/04459/2020 and PTDC/MAT-PUR/29447/2017.}

\begin{abstract} 
We establish sharp dynamical implications of convexity on symmetric spheres that do not follow from dynamical convexity. It allows us to show the existence of elliptic and non-hyperbolic periodic orbits and to furnish new examples of dynamically convex contact forms, in any dimension, that are not equivalent to convex ones via contactomorphisms that preserve the symmetry. Moreover, these examples are $C^1$-stable in the sense that they are actually not equivalent to convex ones via contactomorphisms that are $C^1$-close to those preserving the symmetry. We also show the multiplicity of symmetric non-hyperbolic and symmetric (not necessarily non-hyperbolic) closed Reeb orbits under suitable pinching conditions.
\end{abstract}

\maketitle

\tableofcontents

\section{Introduction and main results}

\subsection{Introduction}

In this work we focus on the dynamics of Reeb flows on convex hypersurfaces with symmetry in the standard symplectic vector space. In this setting, we explore dynamical implications of convexity, the existence of elliptic and non-hyperbolic closed orbits, the relation between convexity and dynamical convexity and the multiplicity problem for symmetric non-hyperbolic and symmetric (not necessarily non-hyperbolic) closed Reeb orbits.

To elaborate, consider Reeb flows on the standard contact sphere $(S^{2n+1},\xis)$. Two of the most fundamental problems in Hamiltonian dynamics are the multiplicity question for simple (i.e., non-iterated) closed Reeb orbits and the existence of elliptic closed Reeb orbits. Two long standing conjectures establish that there are at least $n+1$ simple closed orbits and at least one elliptic closed orbit for \emph{any} contact form $\beta$ on $(S^{2n+1},\xis)$. (These conjectures appeared originally in the literature for strictly convex contact forms (see the definition below), see, for instance, \cite{Eke} and \cite{Lon02}. However, we think that the conjectures should be stated for any contact form; c.f. \cite{GM}).  The first conjecture was proved for $n=1$ by Cristofaro-Gardiner and Hutchings \cite{CGH} (in the more general setting of Reeb flows in dimension three) and independently by Ginzburg, Hein, Hryniewicz and Macarini \cite{GHHM}; see also \cite{LL} where an alternate proof was given using a result from \cite{GHHM}. In higher dimensions, the question is widely open without additional assumptions on $\beta$, such as convexity or certain index requirements or non-degeneracy of closed Reeb orbits. The second conjecture is completely open in any dimension (bigger than one) without further hypotheses.

In order to explain the convexity assumption, note that there is a natural bijection between contact forms $\beta$ on $(S^{2n+1},\xis)$ and starshaped hypersurfaces $\S_\beta$ in $\R^{2n+2}$ in the following way. Consider the restriction of the Liouville form $\lambda$ to the unit sphere $S^{2n+1} \subset \R^{2n+2}$. A contact form $\beta$ supporting the (cooriented) standard contact structure $\xis$ is a 1-form given by $f\lambda|_{S^{2n+1}}$, where $f: S^{2n+1} \to \R$ is a positive function. The bijection is given by
\[
\beta=f\lambda|_{S^{2n+1}} \longleftrightarrow\Sigma_\beta=\{\sqrt{f(x)}x;\, x\in S^{2n+1}\}
\]
and it satisfies the property that if $\Sigma_\beta$ is the energy level of a homogeneous of degree two Hamiltonian $H_\beta: \R^{2n+2} \to \R$ then the Hamiltonian flow of $H_\beta$ on $\Sigma_\beta$ is equivalent to the Reeb flow of $\beta$. We say that $\beta$ is \emph{convex} (resp. \emph{strictly convex}) if $\Sigma_{\beta}$ bounds a convex (resp. strictly convex) domain.

Let us denote by $\P$ (resp. $\P_e$) the set of simple (resp. simple elliptic) closed Reeb orbits for $\beta$. When $\beta$ is strictly convex, a remarkable result due to Long and Zhu \cite{LZ} asserts that $\#\P\geq \lfloor (n+1)/2 \rfloor +1$. This result was improved when $n$ is even by Wang \cite{Wa}, furnishing the lower bound $\#\P\geq \lceil (n+1)/2 \rceil +1$. No lower bound for $\#\P_e$ is known under only the hypothesis that $\beta$ is strictly convex.

Under the assumption that $\beta$ is strictly convex \emph{and invariant by the antipodal map}, Liu, Long and Zhu \cite{LLZ} showed that $\#\P\geq n+1$ and Dell'Antonio, D'Onofrio and Ekeland \cite{DDE} proved that $\#\P_e\geq 1$. Thus, the aforementioned conjectures are true for this important class of contact forms.

The convexity requirement is used in several ways in these results. However, it is not natural from the point of view of contact topology since it is not a condition invariant under contactomorphisms. An alternative notion, introduced by Hofer, Wysocki and Zehnder \cite{HWZ}, is \emph{dynamical convexity}.  A contact form $\beta$ on $S^{2n+1}$ is called dynamically convex if every closed Reeb orbit $\ga$ of $\beta$ has Conley-Zehnder index $\cz(\ga)$ greater than or equal to $n+2$. Clearly dynamical convexity is invariant under contactomorphisms and it is not hard to see that convexity implies dynamical convexity. When $\beta$ is dynamically convex, Ginzburg and G\"urel \cite{GG} and, independently, Duan and Liu \cite{DL}, proved that $\#\P\geq \lceil (n+1)/2 \rceil +1$, showing that the lower bound established by Long, Zhu and Wang in \cite{LZ,Wa}  in the convex case holds for dynamically convex hypersurfaces.

When $\beta$ is dynamically convex and invariant by the antipodal map, Abreu and Macarini proved in \cite{AM2} that $\P_e\geq 1$, extending the aforementioned result of Dell'Antonio, D'Onofrio and Ekeland \cite{DDE} to dynamically convex contact forms. The extension of the results of Liu, Long and Zhu \cite{LLZ} to the dynamically convex scenario is more subtle and carried out in the following way. In \cite{GM}, Ginzburg and Macarini introduced the notion of \emph{strong dynamical convexity}: a symmetric contact form $\beta$ on $S^{2n+1}$ is strongly dynamically convex if it is dynamically convex and its degenerate symmetric periodic orbits satisfy a technical additional assumption involving the normal forms of the eigenvalue one; see \cite{GM} for details. They proved that if $\beta$ is strongly dynamically convex then $\#\P\geq n+1$. Moreover, they showed that if $\beta$ is convex and invariant by the antipodal map then it is strongly dynamically convex. Using this, they were able to construct the first examples of antipodally symmetric dynamically convex contact forms (in dimension at least five) that are not equivalent to convex ones via contactomorphisms that commute with the antipodal map. The main point in the construction of these examples was to realize that, under the presence of symmetries, convexity implies more than dynamical convexity.

In this paper, continuing the study in \cite{GM}, we explore further implications of convexity for symmetric contact forms on the sphere. More precisely, we consider contact forms $\alpha$ on lens spaces $\L$ endowed with the induced contact structure $\xi$. Given a homotopy class $a \in \pi_1(\L)$ we will associate to it three rational numbers $k_a$, $h_a$ and $\th_a$ related to the positive equivariant symplectic homology of $\L$ and that can be easily computed from the weights $\ell_0,\dots,\ell_n$. (Note here that since, in general, $c_1(\xi)\neq 0$, the grading of the positive equivariant symplectic homology is fractional.) We say that $\alpha$ is convex (resp. strictly convex) if so is its lift $\beta=\pi^*\alpha$, where $\pi: S^{2n+1} \to \L$ is the quotient projection. Given a convex (resp. strictly convex) contact form on $\L$ and a closed orbit $\ga$ with non-trivial homotopy class $a$, we show in Theorem \ref{thm:main} that its index satisfies $\cz(\ga) \geq k_a$ and the following holds: $\ga$ must be non-hyperbolic if $\cz(\ga)<h_a$ (resp. $\cz(\ga)<\th_a$) and $\ga$ has to be elliptic if $\cz(\ga)=k_a$ and $a$ satisfies a suitable positivity assumption.

In Theorem \ref{thm:sharp} we show that this result is sharp: we must have a closed orbit $\ga$ with homotopy class $a$ such that $\cz(\ga)=k_a$, there are examples of strictly convex contact forms carrying a hyperbolic closed orbit $\ga$ with homotopy class $a$ such that $\cz(\ga)=h_a=\th_a$ and examples of strictly convex contact forms with a non-elliptic closed orbit $\ga$ with positive homotopy class $a$ such that $\cz(\ga)=k_a+1$. Using these results, Corollaries \ref{cor:elliptic}, \ref{cor:non-hyperbolic} and \ref{cor:positive homotopy} show that convex contact forms on $\L$ must carry elliptic and non-hyperbolic closed orbits in certain homotopy classes and furnish sufficient conditions to ensure the existence of such homotopy classes.

We say that $\alpha$ is dynamically convex if so is its lift $\beta$ to $S^{2n+1}$; in other words, $\alpha$ is dynamically convex if every \emph{contractible} closed orbit $\ga$ of $\alpha$ satisfies $\cz(\ga)\geq n+2$. We show in Theorem \ref{thm:dc} that the dynamical implications of Theorem \ref{thm:main} do not follow from dynamical convexity at all. It allows us to furnish new examples of symmetric dynamically convex contact forms on $S^{2n+1}$ (in any dimension) that are not equivalent to convex contact forms via contactomorphisms that preserve the symmetry. Actually, these examples are not equivalent to convex contact forms via contactomorphisms $C^1$-close to those that preserve the symmetry (Theorem \ref{thm:dc open}).

Finally, let us briefly describe our results on the multiplicity of closed orbits. Given a symmetric contact form on the sphere, a refinement of the multiplicity question on periodic orbits is to give a lower bound for the number of simple \emph{symmetric} closed Reeb orbits. It is equivalent to the multiplicity question of closed orbits of contact forms on lens spaces whose homotopy classes are generators of the fundamental group. Using Theorem \ref{thm:main} and Floer homology techniques, we obtain results on the multiplicity of \emph{symmetric non-hyperbolic} closed Reeb orbits for suitable contact forms. More precisely, we show in Theorem \ref{thm:multiplicity non-hyp} the existence of $\lfloor (n+1)/2 \rfloor$ simple symmetric non-hyperbolic closed Reeb orbits for contact forms on certain lens spaces satisfying suitable pinching conditions. At last, using a lower bound for the period of closed orbits of strictly convex contact forms due to Croke and Weinstein \cite{CW} (which is also used in the proof of Theorem \ref{thm:multiplicity non-hyp}) and our Floer homology tools, we establish in Theorem \ref{thm:multiplicity} the existence of $n+1$ symmetric (not necessarily non-hyperbolic) closed orbits for convex contact forms satisfying weaker pinching conditions, generalizing a result due to Ekeland and Lasry \cite{EL}.

\subsection{Main results}

Given an integer $p>0$, consider the $\Z_p$-action on $S^{2n+1}$, regarded as a subset of $\C^{n+1}\setminus \{0\}$,  generated by the map
\begin{equation}
\label{eq:action}
\psi(z_0, \dots, z_n) = \left(e^{\frac{2\pi\sqrt{-1}\ell_0}{p}}z_0, e^{\frac{2\pi\sqrt{-1}\ell_1}{p}}z_1, \dots, e^{\frac{2\pi\sqrt{-1}\ell_n}{p}}z_n \right),
\end{equation}
where $\ell_0, \ldots, \ell_n$ are integers called the \emph{weights} of the action. Such an action is free when the weights  are coprime with $p$ and in that case we have a lens space obtained as the quotient of $S^{2n+1}$ by the action of $\Z_p$. We denote such a lens space by $L_p^{2n+1}(\ell_0, \ell_1, \ldots, \ell_n)$.

Note that different sequences of weights might produce diffeomorphic lens spaces, for instance by permuting the weights, adding a multiple of $p$ to some weights, multiplying every weight by some $k$ coprime with $p$ and changing the sign of some weights. Moreover these are the only possibilities leading to diffeomorphic lens spaces. Clearly $\psi$ preserves $\xis$ and therefore $\L$ carries the induced contact structure $\xi$. Permuting the weights, adding a multiple of $p$ to some weights, multiplying every weight by some $k$ coprime with $p$ and changing the sign of \emph{all} the weights lead to contactomorphic lens spaces. (As showed in \cite{AMM}, changing the sign of some weights can lead to lens spaces that are diffeomorphic but not contactomorphic.) \emph{Throughout this work, we will choose the weights $\ell_0,\dots,\ell_n$ such that $\ell_0=1$ and $-p/2 < \ell_i \leq p/2$ for every $i$.} These conditions determine the weights uniquely up to permutation in the last $n$ weights.

Given a homotopy class $a \in \pi_1(\L)$, we will associate to it three key rational numbers in this work. Firstly, we have the lowest degree with non-trivial positive equivariant symplectic homology associated to $a$, that is,
\[
k_a := \min\{k \in \Q;\, \HC_k^a(\L)\neq 0\},
\]
see Section \ref{sec:ESH&IT} for the definition of equivariant symplectic homology and, in particular, its (fractional) grading. This number can be computed in terms of the weights in the following way. Let $j_a \in \{1,\dots,p\}$ be such that the deck transformation corresponding to $a$ is $\psi^{j_a}$. Let $\ell^a_0, \ell^a_1, \ldots, \ell^a_n$ be the corresponding \emph{homotopy weights}, that is the unique integers satisfying $-p/2 < \ell^a_i \leq p/2$ for every $i$ and
\[
\psi^{j_a}(z_0, \dots, z_n) = \left(e^{\frac{2\pi \sqrt{-1} \ell^a_0}{p}}z_0, e^{\frac{2\pi \sqrt{-1} \ell^a_1}{p}}z_1, \dots, e^{\frac{2\pi \sqrt{-1} \ell^a_n}{p}}z_n \right).
\]
Consider the number of positive/negative homotopy weights counted with multiplicity:
\[
w_+^a=\#\{i;\, \ell^a_i>0\}\quad\text{and}\quad w_-^a=\#\{i;\, \ell^a_i<0\}.
\]
Then one can show that
\begin{equation}
\label{eq:k_a}
k_a=w_-^a - w_+^a + \frac{2\sum_i \ell^a_i}{p}+1,
\end{equation}
see Proposition \ref{prop:k_a}.

\begin{example}
\label{ex:k_a}
Consider the following illustrative examples:
\begin{enumerate}
\item Let $a$ be a non-trivial homotopy class of $L^{2n+1}_p(1,\dots,1)$. If $j_a\leq p/2$ then $\ell^a_i=j_a>0$ for every $i$ and therefore
\[
k_a = -(n+1) + \frac{2j_a(n+1)}{p} + 1 = \frac{2j_a(n+1)}{p} - n.
\]
If $j_a>p/2$ then $\ell^a_i=j_a-p<0$ for every $i$ and therefore
\[
k_a = n+1 + \frac{2(j_a-p)(n+1)}{p} + 1 = \frac{2j_a(n+1)}{p} - n.
\]
In particular, we have that $k_a\neq k_b$ whenever $a\neq b$.
\item Let $a$ be a non-trivial homotopy class of $L^{2n+1}_p(1,-1,\dots,1,-1)$ with $p>2$ and $n$ odd. (Note that $L^{2n+1}_p(1,-1,\dots,1,-1)=L^{2n+1}_p(1,\dots,1)$ when $p=2$.) If $a^2\neq 0$ we have that $w^a_-=w^a_+=(n+1)/2$ and $\sum_i \ell^a_i=0$ and therefore $k_a=1$. If $a^2=0$ then $w^a_+=n+1$, $w^a_-=0$ and $\frac{2\sum_i \ell^a_i}{p}=n+1$ implying that $k_a=1$ as well. Hence, $k_a=1$ for every $a$.
\end{enumerate}
\end{example}

Now, we will consider two rational numbers related to $k_a$ and the \emph{multiplicity} of the weights. More precisely, let $\bell^a_1,\dots,\bell^a_k$ be the absolute values of the homotopy weights $\ell^a_0,\dots,\ell^a_n$. Order $\bell^a_i$ such that $\bell^a_1<\bell^a_2<\dots<\bell^a_k$. Given $i \in \{1,\dots,k\}$ we define
\[
\mu^a_i=\#\{j;\, \ell^a_j=\bell^a_i\ \text{and}\ \ell^a_j\neq p/2\}
\]
and
\[
\quad\nu^a_i=\#\{j;\, \ell^a_j=-\bell^a_i\ \text{or we have that}\ \bell^a_i=p/2\ \text{and}\ \ell^a_j=\bell^a_i=p/2\}.
\]
(Note that if $\bell^a_i=p/2$ then $i=k$.) We also define
\[
\tmu^a_i=\#\{j;\, \ell^a_j=\bell^a_i\}\quad\text{and}\quad\tnu^a_i=\#\{j;\, \ell^a_j=-\bell^a_i\}.
\]
Set $\mu^a_0=\nu^a_0=\tnu^a_0=0$. Then one considers the numbers
\[
h_a=\max\bigg\{k_a-1+\sum_{i=0}^j \mu^a_i - \sum_{i=0}^{j} \nu^a_i;\, j \in \{0,\dots,k\}\bigg\}
\]
and
\[
\th_a=\max\bigg\{k_a-1+\sum_{i=1}^j \tmu^a_i - \sum_{i=0}^{j-1} \tnu^a_i;\, j \in \{1,\dots,k\}\bigg\}.
\]
Note that $h_a\leq \th_a$ since $\mu^a_i \leq \tmu^a_i$ and $\nu^a_i \geq \tnu^a_i$ for every $i$. Moreover, $h_a=\th_a$ whenever $p$ is odd.

\begin{example}
\label{ex:h_a}
Let us compute these numbers in the previously considered examples:
\begin{enumerate}
\item Let $a$ be a non-trivial homotopy class of $L^{2n+1}_p(1,\dots,1)$. If $j_a<p/2$ then $\ell^a_i=j_a>0$ for every $i$ and therefore
\[
h_a = \th_a = k_a -1 + n+1 = \frac{2j_a(n+1)}{p}.
\]
If $j_a=p/2$ then $\ell^a_i=p/2$ for every $i$ and consequently
\[
h_a=k_a-1=0\quad\text{and}\quad \th_a=k_a+n=n+1.
\]
If $j_a>p/2$ then $\ell^a_i=j_a-p<0$ for every $i$ and therefore
\[
h_a = \th_a = k_a -1 = \frac{2j_a(n+1)}{p} - (n+1).
\]
\item Let $a$ be a non-trivial homotopy class of $L^{2n+1}_p(1,-1,\dots,1,-1)$ with $p>2$ and $n$ odd. As we saw in Example \ref{ex:k_a}, $k_a=1$ for every $a$. Note that we have only one absolute value $\bell^a_1$, i.e., $k=1$. Consider first the case that $a^2\neq 0$. Then we have that $h_a=k_a-1=0$ because $\mu^a_1=\nu^a_1=(n+1)/2$. Since $k=1$ and $\tmu^a_1=(n+1)/2$, we have that $\th_a = (n+1)/2$. Now, consider the case where $a^2=0$. Then $\mu^a_1=0$ and $\nu^a_1=n+1$ and therefore $h_a=k_a-1=0$. On the other hand, $\tmu^a_1=n+1$ and $\tnu^a_1=0$ and consequently $\th_a=n+1$.
\end{enumerate}
\end{example}

\begin{remark}
Note that the identity $\th_a=k_a+n$ (resp $h_a=k_a+n$) actually holds for every non-trivial homotopy class $a$ in $\L$ such that $\ell^a_i>0$ (resp. $\ell^a_i>0$ and $\ell^a_i\neq p/2$) for every $i$.
\end{remark}

Recall that a periodic orbit is called \emph{hyperbolic} if every eigenvalue of its linearized Poincar\'e map has modulus different from one. On the other hand, it is called \emph{elliptic}  if every eigenvalue of its linearized Poincar\'e map has modulus one. Our first result establishes new dynamical implications of the (not necessarily strict) convexity of contact forms on the standard contact sphere under the presence of symmetries.

\begin{theorem}
\label{thm:main}
Let $\alpha$ be a convex (resp. strictly convex) contact form on a lens space $\L$ and $\gamma$ a closed Reeb orbit of $\alpha$ with non-trivial homotopy class $a$. Then the following assertions hold:
\begin{enumerate}
\item $\cz(\gamma) \geq k_a$;
\item if $\cz(\gamma) < h_a$ (resp. $\cz(\gamma) < \th_a$) then $\ga$ is non-hyperbolic;
\item if $\ell^a_i>0$ and $\ell^a_i\neq p/2$ (resp.  $\ell^a_i>0$) for every $i$ and $\cz(\gamma) = k_a$ then $\ga$ is elliptic.
\end{enumerate}
\end{theorem}

\begin{remark}
We have that $k_0=n+2$. Therefore, when $a=0$ the inequality $\cz(\ga) \geq k_a$ means precisely dynamical convexity. Thus, the first assertion is a generalization of dynamical convexity for periodic orbits with non-trivial homotopy class. In fact, a notion of dynamical convexity for general contact manifolds was introduced in \cite{AM2} and the first assertion states that $\alpha$ is $a$-dynamically convex in the terminology of \cite{AM2}. (Note that to achieve the inequality $\cz(\ga) \geq k_0$ for contractible orbits we need strict convexity while the lower bound $\cz(\ga) \geq k_a$ for non-contractible orbits holds without assuming \emph{strict} convexity.) The second and third assertions have no counterparts for contractible orbits.
\end{remark}

\begin{theorem}
\label{thm:sharp}
Theorem \ref{thm:main} is sharp. More precisely, we have the following:
\begin{enumerate}
\item Given any convex contact form $\alpha$ on $L^{2n+1}_p(\ell_1,\dots,\ell_n)$ and a homotopy class $a$ we must have at least one periodic orbit $\ga$ with homotopy class a such that $\cz(\ga) = k_a$.
\item Given any integers $n\geq 1$ and $p\geq 2$ there exists a strictly convex contact form $\alpha$ on $L^{2n+1}_p(1,\dots,1)$ and a hyperbolic closed Reeb orbit $\ga$ of $\alpha$ with non-trivial homotopy class $a$ satisfying $\cz(\ga)=h_a=\th_a$.
\item There exists a strictly convex contact form $\alpha$ on $L^{3}_4(1,1)$ and a hyperbolic closed Reeb orbit $\ga$ of $\alpha$ with non-trivial homotopy class $a$ such that $\cz(\ga)=k_a+1$ and $a$ satisfies $\ell^a_i>0$ and $\ell^a_i\neq p/2$ for every $i$.
\end{enumerate}
\end{theorem}

\begin{remark}
By Proposition \ref{prop:N}, the contact structure $\xi$ on $L^{3}_4(1,1)$ satisfies the following. The smallest positive integer $N$ such that $Nc_1(\xi)=0$ is equal to 2. Therefore, by \eqref{eq:index difference} we have that the indexes of the periodic orbits are integers. Hence, $\cz(\ga)>k_a$ if and only if $\cz(\ga)\geq k_a+1$. 
\end{remark}

Theorems \ref{thm:main} and \ref{thm:sharp} have the following straightforward consequences on the existence of symmetric elliptic and non-hyperbolic closed orbits of contact forms on $(S^{2n+1},\xis)$ which improves previous results due to Arnaud \cite{Ar} and Liu, Wang and Zhang \cite{LWZ}. Let $\beta$ be a contact form on $S^{2n+1}$ invariant under the $\Z_p$-action generated by \eqref{eq:action}. A closed orbit $\ga$ of $\beta$ is called symmetric if $\psi(\ga(\R))=\ga(\R)$. Note that, the simple symmetric periodic orbits of $\beta$ are in bijection with the simple periodic orbits of $\alpha$ whose homotopy classes are generators of $\pi_1(\L)$, where $\alpha$ is the contact form on $\L$ whose lift to $S^{2n+1}$ is $\beta$.

\begin{corollary}
\label{cor:elliptic}
Let $\alpha$ be a convex (resp. strictly convex) contact form on $\L$ with $p\geq 2$. Assume that $\ell_i>0$ and  $\ell_i\neq p/2$ (resp. $\ell_i>0$) for every $i$. Then $\alpha$ carries at least one elliptic closed orbit whose homotopy class is a generator of $\pi_1(\L)$. 
\end{corollary}

\begin{corollary}
\label{cor:non-hyperbolic}
Let $\alpha$ be a convex (resp. strictly convex) contact form on $\L$ with $p\geq 2$. Assume that $\#\{i;\ell_i>0\ \text{and}\ \ell_i\neq p/2\}-\#\{i;\ell_i<0\ \text{or}\ \ell_i=p/2\} \geq 2$ (resp. $\#\{i;\ell_i>0\}-\#\{i;\ell_i<0\} \geq 2$) for every $i$. Then $\alpha$ carries at least one non-hyperbolic closed orbit whose homotopy class is a generator of $\pi_1(\L)$. 
\end{corollary}

Indeed, consider the generator $a$ of $\pi_1(\L)$ whose homotopy weights are equal to $\ell_0,\dots,\ell_n$. We have from the first assertion of Theorem \ref{thm:sharp} that $\alpha$ must carry a closed orbit $\ga$ with homotopy class $a$ such that $\cz(\ga)=k_a$. Then, under the hypotheses of Corollary \ref{cor:elliptic}, by the third assertion of Theorem \ref{thm:main}, we conclude that $\ga$ has to be elliptic. On the other hand, under the hypotheses of Corollary \ref{cor:non-hyperbolic}, we have that $h_a\geq k_a+1$ (resp. $\th_a\geq k_a+1$) and therefore, by the second assertion of Theorem \ref{thm:main}, we conclude that $\ga$ has to be non-hyperbolic.

Heuristically, we can look at Theorem \ref{thm:main} using the following analogy with Floer homology. Given a closed symplectic aspherical manifold $M^{2n}$ we have that $\HF_*(M) \cong \H_{*+n}(M)$. Consider a $C^2$-small non-degenerate autonomous Hamiltonian $H: M \to \R$. Then the 1-periodic orbits of $H$ correspond to the critical points of $H$. Given a critical point $p$, let $\ga_p$ be the corresponding 1-periodic orbit of $H$ and note that $\cz(\ga_p)=\morse(p)-n$, where $\morse(p)$ is the Morse index of $p$. It is easy to see that if $\cz(\ga_p)=-n$ then it is elliptic and if $\cz(\ga_p)<0$ then it must be non-hyperbolic (actually, it must be non-hyperbolic whenever $\cz(\ga_p)\neq 0$). Thus, if the index of $\ga_p$ is the lowest one then $\ga_p$ must be elliptic and if the index of $\ga_p$ is lower than the ``middle degree" then it has to be non-hyperbolic.

We say that a homotopy class $a \in \pi_1(\L)$ is \emph{positive} (resp. \emph{strictly positive}) if the corresponding homotopy weights satisfy $\ell_i^a>0$ for every $i$ (resp. $\ell_i^a>0$ and $\ell_i^a\neq p/2$ for every $i$). (Equivalently, $a$ is positive (resp. strictly positive) if the imaginary part of $e^{2\pi\sqrt{-1}\ell_i^a/p}$ is non-negative (resp. strictly positive) for every $i$.) Note that if $a$ is positive (resp. strictly positive) then $\th_a=k_a+n$ (resp. $h_a=k_a+n$) so that the previous heuristic analogy is more enlightening for these homotopy classes. A natural question is the existence of positive or strictly positive homotopy classes. Theorems \ref{thm:main} and \ref{thm:sharp} yield the following corollary, proved in Section \ref{sec:proof positive homotopy}, which addresses this question.

\begin{corollary}
\label{cor:positive homotopy}
Let $\alpha$ be a convex (resp. strictly convex) contact form on a lens space and $a$ be a strictly positive (resp. positive) homotopy class. Then $\alpha$ carries an elliptic periodic orbit with homotopy class $a$. In particular, the following assertions hold:
\begin{enumerate}
\item For a 3-dimensional lens space $L(p,q)=L^3_p(1,q)$, if $q\neq -1$ then $L(q,p)$ carries a positive homotopy class $a$. Therefore, every strictly convex contact form on such lens space has an elliptic closed orbit with homotopy class $a$.
\item For $L(p,q)$, if $q= -1$ then $L(q,p)$ carries a positive homotopy class $a$ if $p$ is even and it does not carry a positive homotopy class if $p$ is odd. Therefore, every strictly convex contact form on such lens space has an elliptic closed orbit with homotopy class $a$ whenever $p$ is even.
\item For a higher dimensional lens space $\L$, if the weights assume only two values $\ell_0=1$ and $\ell_1=q$ with $q\neq -1$ then $\L$ carries a positive homotopy class $a$. If $q= -1$ and $p$ is even then $\L$ carries a positive homotopy class $a$ as well. Therefore, every strictly convex contact form on such lens space has an elliptic closed orbit with homotopy class $a$.
\end{enumerate}
Furthermore, if $\ga$ is a closed orbit of $\alpha$ with strictly positive (resp. positive) homotopy class $a$ such that $\cz(\ga)<k_a+n$ then $\ga$ must be non-hyperbolic.
\end{corollary}

\begin{remark}
The same reasoning shows the following. Let $\alpha$ be a convex (resp. strictly convex) contact form on a lens space and $a$ be a homotopy class such that
\[
\#\{i;\ell^a_i>0\ \text{and}\ \ell^a_i\neq p/2\}-\#\{i;\ell^a_i<0\ \text{or}\ \ell^a_i=p/2\} \geq 2\ \ (\text{resp.}\ \#\{i;\ell^a_i>0\}-\#\{i;\ell^a_i<0\} \geq 2).
\]
Then $\alpha$ carries a non-hyperbolic closed orbit with homotopy class $a$. Note that, when $n=1$, this condition is equivalent to $a$ being strictly positive (resp. positive) (which is consistent with the fact that in dimension three a periodic orbit is elliptic if and only if it is non-hyperbolic).
\end{remark}

\begin{remark}
The existence of an elliptic closed orbit when $p$ is even in the corollary above also follows from the results in \cite{AM2,DDE} since in this case the lift of the contact form to $S^{2n+1}$ is invariant by the antipodal map.
\end{remark}

Another consequence of Theorem \ref{thm:main} is the following immediate corollary.

\begin{corollary}
Let $\alpha$ be a strictly convex contact form on $\RP^{2n+1}$. Then every closed Reeb orbit $\ga$ of $\alpha$ satisfying $\cz(\ga)<n+1$ is non-hyperbolic.
\end{corollary}

If $n=1$ this result readily follows from the dynamical convexity of $\alpha$ and the multiplicative property of the index for hyperbolic periodic orbits: if $\ga$ is hyperbolic then $\cz(\ga^k)=k\cz(\ga)$ for every $k$; hence, if $\cz(\ga)\leq 1$ then $\cz(\ga^2)\leq 2$, contradicting the dynamical convexity (note that the contact structure on $\RP^3$ has vanishing first Chern class and therefore the index is an integer so that $\cz(\ga)<2$ if and only if $\cz(\ga)\leq 1$). However, there is no reason why this is true in higher dimensions. Indeed, the next result shows that the assumption that $\alpha$ is convex in Theorem \ref{thm:main} cannot be relaxed to the condition that $\alpha$ is dynamically convex at all.

\begin{theorem}
\label{thm:dc}
The following assertions hold:
\begin{enumerate}
\item Consider integers $n\geq 1$ and $p\geq 3$. If $n=1$ (resp. $n=2$), assume furthermore that $p\geq 5$ (resp. $p\geq 4$). Then there exists a dynamically convex contact form $\alpha$ on $L^{2n+1}_p(1,\dots,1)$ carrying a closed Reeb orbit with non-trivial homotopy class $a$ such that $\cz(\ga)<k_a$. 

\item There exists a dynamically convex contact form $\alpha$ on $L^{5}_{11}(1,1,1)$ and a hyperbolic closed Reeb orbit $\ga$ of $\alpha$ with non-trivial homotopy class $a$ such that $\ell^a_i>0$ for every $i$ and $\cz(\ga)=k_a<h_a$.
\end{enumerate}
\end{theorem}

\begin{remark}
\label{rmk:Chern class}
Note that the second assertion shows that both the second and third assertions of Theorem \ref{thm:main} do not hold assuming that $\alpha$ is dynamically convex. The first Chern class of the contact structure on $L^{5}_{11}(1,1,1)$ does not vanish. However, we can construct examples with vanishing first Chern class for which the second assertion of Theorem \ref{thm:main} does not hold; see Remarks \ref{rmk:CC1}, \ref{rmk:CC2}, \ref{rmk:CC3}, \ref{rmk:CC4} and \ref{rmk:CC5}.
\end{remark}

\begin{remark}
The closed orbit $\ga$ in the first assertion is non-degenerate. Therefore, every contact form $\tilde\alpha$ $C^2$-close to $\alpha$ carries a closed Reeb orbit $\tilde\ga$ with non-trivial homotopy class $a$ such that $\cz(\tilde\ga)<k_a$. Clearly, a similar stability statement also holds for the contact form in the second assertion: every contact form $\tilde\alpha$ $C^2$-close to $\alpha$ carries a hyperbolic closed Reeb orbit $\tilde\ga$ with non-trivial homotopy class $a$ such that $\ell^a_i>0$ for every $i$ and $\cz(\ga)=k_a<h_a$. It shows a contrast between these examples and those obtained in \cite{GM} where the degeneracy of the periodic orbits plays a crucial role.
\end{remark}

\begin{remark}
The contact form $\alpha$ in both assertions can be chosen arbitrarily $C^1$-close to a convex contact form; see Remarks \ref{rmk:small1} and \ref{rmk:small2}.
\end{remark}

Note that, by the invariance of equivariant symplectic homology, if $\phi: \L \hookleftarrow$ is a contactomorphism then $k_a=k_{\phi_*a}$. In particular, if $k_a \neq k_b$ whenever $a\neq b$ then $\phi$ acts trivially on $\pi_1(\L)$. Hence, in this case, all the assertions in Theorem \ref{thm:main} are invariant by $\phi$. Therefore, since this property holds for $L^{2n+1}_p(1,\dots,1)$ (see Example \ref{ex:k_a}), Theorem \ref{thm:dc} furnish new examples of dynamically convex contact forms $\beta$ on spheres that are not contactomorphic to convex contact forms via contactormophisms that commute with the corresponding symmetry.

The fact that we do not assume \emph{strict} convexity implies more than this: let $S \subset \Cont(S^{2n+1})$ be the subset of contactomorphisms that commute with the corresponding $\Z_p$-action. Then there exists a $C^2$-neighborhood $U$ of $S$ such that $\beta$ is not equivalent to a convex contact form via any contactomorphism $\vr \in U$. Indeed, if there exists a sequence $\vr_i \xrightarrow{C^2} \bar\vr \in S$ such that $\vr_i^*\beta$ is convex then so is $\bar\vr^*\beta$, furnishing a contradiction. However, we actually have this property for a \emph{$C^1$-neighborhood} of $S$ as the following theorem shows.

 \begin{theorem}
\label{thm:dc open}
Let $\alpha$ be one of the contact forms furnished by Theorem \ref{thm:dc} and consider its lift $\beta$ to $S^{2n+1}$. Let $S \subset \Cont(S^{2n+1})$ be the subset of contactomorphisms that commute with the corresponding $\Z_p$-action. Then there exists a $C^1$-neighborhood $U$ of $S$ such that $\beta$ is not equivalent to a convex contact form via any contactomorphism $\vr \in U$.
\end{theorem}

\begin{remark}
Very recently, Chaidez and Edtmair \cite{CE} showed examples of dynamically convex contact forms on $S^3$ that are not equivalent to a convex contact form via any contactomorphism. However, their methods work only in dimension three.
\end{remark}

Another application of Theorem \ref{thm:main}, obtained using Floer homology techniques, is the following result on the multiplicity of symmetric non-hyperbolic periodic orbits on convex spheres. Before we state it, we need to introduce some definitions. Let $\alpha$ be a contact form on $\L$ and $\beta$ its lift to $S^{2n+1}$. Let $H_\beta: \R^{2n+2} \to \R$ be the unique homogeneous of degree two Hamiltonian such that $\S_\beta=H_\beta^{-1}(1)$. Note that $\beta$ is convex if and only if so is $H_\beta$ (that is, the Hessian of $H_\beta$ is positive semi-definite at every point). Given real numbers $0 < r \leq R$ we say that $\alpha$ is $(r,R)$-pinched if $R^{-2}\|x\|^2 \leq H_\beta(x) \leq r^{-2}\|x\|^2$ for every $x \in \Sigma_\beta$. We say that $\alpha$ is $(r,R)$-H-pinched if $R^{-2}\|v\|^2 \leq \Hess H_\beta(x)(v,v) \leq r^{-2}\|v\|^2$ for every $x \in \Sigma_\beta$ and $v \in \R^{2n+2}$. (Here ``H" stands for the Hessian.) Using the homogeneity of $H_\beta$, it is easy to see that if $\beta$ is $(r,R)$-H-pinched then it is $(\sqrt{2}r,\sqrt{2}R)$-pinched.

\begin{theorem}
\label{thm:multiplicity non-hyp}
Let $n\geq 1$ and $p\geq 2$ be integers and $0< r \leq R$ be real numbers such that $\frac{R}{r}<\sqrt{p+1}$. Let $\alpha$ be an $(r,R)$-H-pinched contact form on $L^{2n+1}_p(1,\dots,1)$. Denote by $a$ the generator of $\pi_1(L^{2n+1}_p(1,\dots,1))$ such that $\ell^a_i=1$ for every $i$. Suppose that the periodic orbits of $\alpha$ with homotopy class $a$ are isolated. Then $\alpha$ carries at least $\lfloor \frac{n+1}{2} \rfloor$ geometrically distinct non-hyperbolic closed Reeb orbits with homotopy class $a$.
\end{theorem}

Finally, we prove the following result on the multiplicity of symmetric (not necessarily non-hyperbolic) closed orbits for contact forms satisfying a weaker pinching condition. It improves results from \cite{AHG,Ke} and does not use Theorem \ref{thm:main}; the convexity of the contact form is used to obtain a lower bound for the period of closed orbits of strictly convex contact forms due to Croke and Weinstein \cite{CW} (which is also used in the proof of Theorem \ref{thm:multiplicity non-hyp}).
 
\begin{theorem}
\label{thm:multiplicity}
Let $n\geq 1$ and $p\geq 1$ be integers and $0< r \leq R$ be real numbers such that $\frac{R}{r}<\sqrt{p+1}$. Let $\alpha$ be a strictly convex $(r,R)$-pinched contact form on $L^{2n+1}_p(1,\dots,1)$. Denote by  $a$ the generator of $\pi_1(L^{2n+1}_p(1,\dots,1))$ such that $\ell^a_i=1$ for every $i$. Then $\alpha$ carries at least $n+1$ geometrically distinct closed Reeb orbits with homotopy class $a$.
 \end{theorem}
 
Note that we allow $p=1$ in the previous theorem. This case corresponds to a classical result due to Ekeland and Lasry \cite{EL}.

 \subsection{Organization of the paper}
 
The rest of the paper is organized as follows. The background on equivariant symplectic homology and Lusternik-Schnirelmann theory in Floer homology necessary for this work is presented in Sections \ref{sec:ESH} and \ref{sec:LS} respectively. Its (fractional) grading is discussed in Section \ref{sec:grading}. Section \ref{sec:computations} contains index computations relevant for this work and shows how to deal with the non-vanishing of the first Chern class in lens spaces. The tools from index theory are introduced in Section \ref{sec:Bott}. Theorem \ref{thm:main} is proved in Section \ref{sec:proof main} and its Corollary \ref{cor:positive homotopy} is proved in Section \ref{sec:proof positive homotopy}. Theorem \ref{thm:sharp}, that establishes the sharpness of Theorem \ref{thm:main}, is proved in Section \ref{sec:sharp}. Its proof appears after the proof of Theorem \ref{thm:dc}, presented in Section \ref{sec:proof dc}, since it uses several ingredients from that. Finally, Sections \ref{sec:proof multiplicity non-hyp} and \ref{sec:proof multiplicity} are devoted to the proof of our multiplicity results, namely, Theorems \ref{thm:multiplicity non-hyp} and \ref{thm:multiplicity} respectively.

\subsection{Conventions}

Throughout this work, we will use the convention that the natural numbers are given by the positive integers. Given a symplectic manifold $(M,\om)$ and a Hamiltonian $H_t: M \to \R$, we take Hamilton's equation to be $i_{X_{H_t}}\om=-dH_t$. A compatible almost complex structure $J$ is defined by the condition that $\om(\cdot,J\cdot)$ is a Riemannian metric. Throughout this work, the Conley-Zehnder index $\mu$ is normalized so that when $Q$ is a small positive definite quadratic form the path $\Ga: [0,1] \to \Sp(2n)$ generated by $Q$ and given by $\Ga(t)=\exp(tJQ)$ has $\cz(\Ga)=n$. We also take the canonical symplectic form on $\R^{2n}$ to be $\sum dq_i \wedge dp_i$. For degenerate paths, the Conley--Zehnder index $\mu$ is defined as the lower semi-continuous extension of the Conley--Zehnder index from the paths with non-degenerate endpoint. More precisely,
$$
\mu(\Gamma)=\liminf_{\tilde{\Gamma}\to\Gamma}\mu(\tilde{\Gamma}),
$$
where $\tilde{\Gamma}$ is a small perturbation of $\Gamma$ with non-degenerate endpoint. These conventions are consistent with the ones used in \cite{GM}. 

\subsection{Acknowledgements}

We are grateful to Viktor Ginzburg and Umberto Hryniewicz for useful comments on a preliminary version of this paper.

\section{Equivariant symplectic homology and index theory}
\label{sec:ESH&IT}

\subsection{Equivariant symplectic homology}
\label{sec:ESH}

Let $(M^{2n+1},\xi)$ be a contact manifold endowed with a strong symplectic filling given by a Liouville domain $W$ such that $c_1(TW)|_{H_2(W,\R)}=0$. The positive equivariant symplectic homology $\SH^{S^1,+}_*(W)$ is a symplectic invariant introduced by Viterbo \cite{Vit} and developed by Bourgeois and Oancea \cite{BO10, BO13a, BO13b, BO17}.

It turns out that the positive equivariant symplectic homology can be obtained as the homology of a chain complex $\CC_*(\alpha)$ with rational coefficients generated by the good closed Reeb orbits of a non-degenerate contact form $\alpha$ on $M$. This complex is filtered by the action and graded by the Conley-Zehnder index; see \cite[Proposition 3.3]{GG} and Section \ref{sec:grading} for a discussion concerning the grading and good orbits. The differential in the complex $\CC_*(\alpha)$, but not its homology, depends on several auxiliary choices, and the nature of the differential is not essential for our purposes. The complex $\CC_*(\alpha)$ is functorial in $\alpha$ in the sense that a symplectic cobordism equipped with a suitable extra structure gives rise to a map of complexes. For the sake of brevity and to emphasize the obvious analogy with contact homology, we denote the homology of $\CC_*(\alpha)$ by $\HC_*(M)$ rather than $\SH^{S^1,+}_*(W)$. Furthermore, once we fix a free homotopy class of loops in $W$, the part of $\CC_*(\alpha)$ generated by closed Reeb orbits in that class is a subcomplex. As a consequence, the entire complex $\CC_*(\alpha)$ breaks down into a direct sum of such subcomplexes indexed by free homotopy classes of loops in $W$.

A remarkable observation by Bourgeois and Oancea in \cite[Section 4.1.2]{BO17} is that under suitable additional assumptions on the indices of closed Reeb orbits the positive equivariant symplectic homology is defined even when $M$ does not have a symplectic filling and therefore is a contact invariant. To be more precise, we assume that $c_1(\xi)|_{H_2(M,\R)}=0$ and that $M$ admits a non-degenerate contact form $\alpha$ such that all of its closed \emph{contractible} Reeb orbits have Conley--Zehnder index strictly greater than $3-n$. Furthermore, under this assumption once again the positive equivariant symplectic homology of $M$ can be described as the homology of a complex $\CC_*(\alpha)$ generated by good closed Reeb orbits of $\alpha$, graded by the Conley-Zehnder index and filtered by the action. The complex breaks down into the direct sum of subcomplexes indexed by free homotopy classes of loops \emph{in $M$}. We will use the notation $\HC^a_*(M)$ to denote the homology of the complex generated by the orbits with free homotopy class $a$. Given a non-degenerate contact form $\alpha$ and numbers $0<T_1<T_2\leq \infty$ we denote by $\HC^{a,(T_1,T_2)}_*(\alpha)$ the equivariant symplectic homology of $\alpha$ with free homotopy class $a$ and action window $(T_1,T_2)$. When $\alpha$ is degenerate and both $T_1$ and $T_2$ are not in the action spectrum of $\alpha$ we define $\HC_*^{a,(T_1,T_2)}(\alpha)$ as $\HC_*^{a,(T_1,T_2)}(\balpha)$ for some small non-degenerate perturbation $\balpha$ of $\alpha$. (Recall that the action spectrum of $\alpha$ is given by $\A(\alpha)=\{\int_\ga\alpha;\,\ga\ \text{is a closed Reeb orbit of}\ \alpha\}$.) Given $T \in (0,\infty]$ we denote by $\HC_*^{a,T}(\alpha)$ the filtered equivariant symplectic homology $\HC_*^{a,(\ep,T)}(\alpha)$ for some $\ep>0$ sufficiently small such that $\ep<\min\{T;\,T \in \A(\alpha)\}$.

The functoriality of $\CC_*(\alpha)$ respects the action and homotopy filtrations and therefore the following holds. Given contact forms $\alpha_0$, $\alpha_1$ and $\alpha_2$, we say that $\alpha_0 < \alpha_1 < \alpha_2$ if $\alpha_1=f_1\alpha_0$ and $\alpha_2=f_2\alpha_0$ for functions $f_i: M \to \R$ ($i=1,2$) such that $f_1(x)>1$ and $f_1(x)<f_2(x)$ for every $x \in M$. Then, given  $0<T_1<T_2$ not in the action spectrum of $\alpha_i$ ($i=0,1,2$), we have continuation maps $\phi_{\alpha_i,\alpha_j}: \HC_*^{a,(T_1,T_2)}(\alpha_i) \to \HC_*^{a,(T_1,T_2)}(\alpha_j)$ (with $i,j \in \{0,1,2\}$ such that $i>j$) that fit into the commutative diagram
\begin{equation}
\label{eq:triangle}
\xymatrix{\HC^{a,(T_1,T_2)}_*(\alpha_2) \ar[rr]^{\phi_{\alpha_2,\alpha_0}}
\ar[dr]_{\phi_{\alpha_2,\alpha_1}} && \HC^{a,(T_1,T_2)}_*(\alpha_0).\\
& \HC^{a,(T_1,T_2)}_*(\alpha_1) \ar[ur]_{\phi_{\alpha_1,\alpha_0}}&}
\end{equation}

These continuation maps have the following property that will be useful in this work. Let $\alpha_t$, $t\in [0,1]$, be a smooth family of contact forms such that $\alpha_t<\alpha_{t'}$ whenever $t<t'$. Suppose that there exists $T \in \R$ such that $T \notin \cup_{t\in[0,1]} \A^a(\alpha_t)$, where
\[
\A^a(\alpha_t)=\bigg\{\int_\ga\alpha_t;\,\ga\ \text{is a closed Reeb orbit of}\ \alpha_t\ \text{with homotopy class}\ a\bigg\}.
\]
Then
\[
\phi_{\alpha_1,\alpha_0}: \HC^{a,T}_*(\alpha_1) \to \HC^{a,T}_*(\alpha_0)
\]
is an isomorphism. Hence, by the diagram \eqref{eq:triangle}, we conclude that if a contact form $\alpha$ satisfies $\alpha_0 < \alpha < \alpha_1$ and $T \notin \A^a(\alpha)$ then the corresponding continuation map
\[
\phi_{\alpha_1,\alpha}: \HC^{a,T}_*(\alpha_1) \to \HC^{a,T}_*(\alpha)
\]
is injective.

Now, suppose that $(M,\xi)$ is given by a lens space $\L$ with the contact structure $\xi$ induced by $(S^{2n+1},\xis)$. Then $c_1(\xi)|_{H_2(M,\R)}=0$, because $H^2(M,\Z)$ is torsion, and every contractible closed orbit of a dynamically convex contact form $\alpha$ on $\L$ has Conley-Zehnder index strictly greater than $3-n$ (recall that a contact form on $\L$ is dynamically convex if every contractible closed orbit has index bigger than or equal to $n+2$). The existence of dynamically convex contact forms follows from the fact that every strictly convex contact form is dynamically convex. Thus, $\HC_*^a(\L)$ is a contact invariant that can be obtained as the homology of a chain complex $\CC^a_*(\alpha)$ with rational coefficients generated by the good closed Reeb orbits of $\alpha$ with homotopy class $a$ (note that $\pi_1(\L)$ is abelian). Moreover, $\HC_*^a(\L)$ have continuation maps that satisfy the diagram \eqref{eq:triangle}.

Finally, let us briefly recall the definition of equivariant local symplectic homology. Given an isolated (possibly degenerate) closed orbit $\ga$ of $\alpha$ we have its equivariant local symplectic homology $\HC_*(\ga)$ \cite{GG,HM}. It is supported in $[\cz(\ga),\cz(\ga)+\nu(\ga)]$ (i.e. $\HC_k(\ga)=0$ for every $k \notin [\cz(\ga),\cz(\ga)+\nu(\ga)]$), where $\nu(\ga)$ is the nullity of $\ga$, i.e., the geometric multiplicity of the eigenvalue 1 of the linearized Poincar\'e map. If $\alpha$ carries finitely many simple closed orbits with free homotopy class $a$ then if $\HC^a_k(M) \neq 0$ there exists a closed orbit $\ga$ with free homotopy class $a$ such that $\HC_k(\ga)\neq 0$.

\subsection{Lusternik-Schnirelmann theory}
\label{sec:LS}

In what follows, we will explain briefly the results from Lusternik-Schnirelmann theory in Floer homology necessary for this work. We refer to \cite{GG} for details.

Let $M=\L$ and $a \in\pi_1(M)$. Let $\alpha$ be a contact form on $M$ such that every closed orbit with homotopy class $a$ is isolated and $T \in (0,\infty]\setminus\A(\alpha)$. Given a non-trivial element $w \in \HC^{a,T}_k(\alpha)$ we have a spectral invariant given by
\[
c_w(\alpha)=\inf\{T' \in (0,T)\setminus\A(\alpha);\,w\in\text{Im}(i^{a,T'})\}
\]
where $i^{a,T'}: \HC_*^{a,T'}(\alpha) \to  \HC_*^{a,T}(\alpha)$ is the map induced in the homology by the inclusion of the complexes. It turns out that there exists a periodic orbit $\ga$ with action $c_w(\alpha)$ and free homotopy class $a$ such that $\HC_k(\ga)\neq 0$; c.f. \cite[Corollary 3.9]{GG}. 

There is a shift operator $D: \HC^{a}_*(\alpha) \to  \HC^{a}_{*-2}(\alpha)$ introduced in \cite{BO13b} which respects the action filtration and satisfies the property
\begin{equation}
\label{eq:shift op}
c_w(\alpha) > c_{D(w)}(\alpha)
\end{equation}
see \cite[Theorem 1.1]{GG}. The shift operator and the spectral invariants are functorial with respect to the continuation maps in the sense that given contact forms $\alpha_0<\alpha_1$ we have that $D$ commutes with $\phi_{\alpha_1,\alpha_0}: \HC_*^{a,T}(\alpha_1) \to \HC_*^{a,T}(\alpha_0)$ and $c_{\phi_{\alpha_1,\alpha_0}(w)}(\alpha_0)\leq c_w(\alpha_1)$ for every $w\in \HC^{a,T}_*(\alpha_1)$, see \cite[Proposition 3.1]{GG}.

Suppose now that $M$ admits a contact form $\alpha_0$ whose Reeb flow generates a free circle action such that $a$ is the homotopy class of the simple orbits of $\alpha_0$. Let $T \in (0,\infty]\setminus\A(\alpha_0)$. Let $B=M/S^1$ and assume further that $\H_*(B;\Q)$ vanishes in odd degrees. Then the shift operator $D$ can be (partially) computed in the following way. Let $\Delta: \H_*(B;\Q) \to \H_{*-2}(B;\Q)$ be the shift operator given by the Gysin exact sequence associated to the $S^1$-bundle $S^1 \to M \to B$
\[
\cdots \rightarrow \H_*(M;\Q) \rightarrow \H_*(B;\Q) \xrightarrow{\Delta} \H_{*-2}(B;\Q) \rightarrow \H_{*-1}(M;\Q) \rightarrow \cdots.
\]
Suppose that there exist non-zero elements $v_i \in \H_{2i}(B;\Q)$, $i=0,\dots,n$, such that $\Delta(v_{i+1})=v_i$ for every $i \in \{0,\dots,n-1\}$.

\begin{remark}
In this work, we will use the techniques presented in this section in the particular case where $M=L^{2n+1}_p(1,\dots,1)$ is endowed with the contact form $\alpha_0$ that generates the obvious free circle action (induced by the Hopf fibration) whose orbit space $B$ is $\CP^n$. In this case, clearly $\H_*(B;\Q)$ vanishes in odd degrees, $\H_{2i}(B;\Q) \cong \Q$ for every $i \in \{0,\dots,n\}$ and $\Delta: H_{2i+2}(B;\Q) \to H_{2i}(B;\Q)$ is an isomorphism for every $i \in \{0,\dots,n-1\}$.
\end{remark}

Let $\ga$ be a simple periodic orbit of $\alpha_0$ and denote by $A(\ga)$ the action of $\ga$. A standard Morse-Bott computation shows that there exists $\ep>0$ such that
\[
\HC_*^{a,(A(\ga^{(k-1)p+1})-\ep,A(\ga^{(k-1)p+1})+\ep)}(\alpha_0) \cong \H_{*-\cz(\ga^{(k-1)p+1})}(B;\Q).
\]
for every $k \in \N$. It turns out that this isomorphism is equivariant with respect to $D$ and $((k-1)p+1)\Delta$; see \cite[Proposition 2.22]{GG}. (Note that $((k-1)p+1)\Delta$ is the map in the Gysin exact sequence of the $S^1$-bundle induced by the $((k-1)p+1)$-th iterate of the Reeb flow of $\alpha_0$.) From this (using our assumption that $\H_*(B;\Q)$ vanishes in odd degrees) we can conclude that
\begin{equation}
\label{eq:ESH prequantization action}
\HC^{a,T}_*(\alpha_0) \cong \oplus_{k\in \{j \in \N;\, A(\ga^{(j-1)p+1})<T\}}\H_{*-\cz(\ga^{(k-1)p+1})}(B;\Q)
\end{equation}
and that (using our hypothesis on $\Delta$) if $T>A(\ga)$ then there exist non-zero elements $w_i \in \HC^{a,T}_{\cz(\ga)+2i}(\alpha_0)$, $i=0,\dots,n$, such that $Dw_{i+1}=w_i$ for every $i \in \{0,\dots,n-1\}$. (Note here that $D$ respects the action filtration and therefore given non-zero elements $w_i$ in the first summand $\HC_*^{a,(A(\ga)-\ep,A(\ga)+\ep)}(\alpha_0) \cong \HC_*^{a,(0,A(\ga)+\ep)}(\alpha_0)$ such that $D^{\A(\ga)+\ep}w_{i+1}=w_i$, where $D^{\A(\ga)+\ep}: \HC_*^{a,(0,A(\ga)+\ep)}(\alpha_0) \to \HC_*^{a,(0,A(\ga)+\ep)}(\alpha_0)$ is the shift operator in the corresponding action window, then the corresponding elements $w_i$ in $\HC^{a,T}_*(\alpha_0)$ also satisfy $Dw_{i+1}=w_i$.)

Taking $T=\infty$ (so that $\HC^{a,T}_*(\alpha_0)=\HC^{a}_*(M)$ does not depend on the contact form) we conclude in particular that
\begin{equation}
\label{eq:ESH prequantization}
\HC^{a}_*(M) \cong \oplus_{k\in \N}\H_{*-\cz(\ga^{(k-1)p+1})}(B;\Q)
\end{equation}
and that there exist non-zero elements $w_i \in \HC^{a}_{\cz(\ga)+2i}(M)$, $i=0,\dots,n$, such that $Dw_{i+1}=w_i$. From this and \eqref{eq:shift op} we infer that there exists an injective map
\begin{equation}
\label{eq:carrier map}
\psi: \{0,\dots,n\} \to \P^a(\alpha),
\end{equation}
called carrier map, where $\P^a(\alpha)$ is the set of closed orbits of $\alpha$ with homotopy class $a$ (note the difference between $\alpha$ and $\alpha_0$), such that if $\ga_i=\psi(i)$ then $A(\ga_i)=c_{w_i}(\alpha)$ and $\HC_{\cz(\ga)+2i}(\ga_i)\neq 0$; see \cite{GG}.

Using the functoriality of $D$, we can refine this carrier map in the following way. Consider contact forms $\alpha<\alpha'$ and let $T \in (0,\infty]\setminus(\A(\alpha)\cup\A(\alpha'))$. Suppose that there exist non-zero elements $w_i \in \HC^{a,T}_{k_i}(\alpha')$, $i=0,\dots,n$, such that $Dw_{i+1}=w_i$ for every $i \in \{0,\dots,n-1\}$ (for some $k_i$ such that $k_i=k_{i+1}-2$) and that the continuation map $\phi_{\alpha',\alpha}: \HC^{a,T}_{*}(\alpha') \to \HC^{a,T}_{*}(\alpha)$ is injective. Then there exists an action filtered injective carrier map
\begin{equation}
\label{eq:filtered carrier map}
\psi^T: \{0,\dots,n\} \to \P^{a,T}(\alpha),
\end{equation}
where $\P^{a,T}(\alpha)$ is the set of closed orbits of $\alpha$ with homotopy class $a$ and period less than $T$, such that if $\ga_i=\psi(i)$ then $A(\ga_i)=c_{\phi_{\alpha',\alpha}(w_i)}(\alpha)$ and $\HC_{k_i}(\ga_i)\neq 0$.

\subsection{Grading}
\label{sec:grading}

The grading of the positive equivariant symplectic homology is defined in \cite{BO17} as follows. In what follows, let $M=\L$ be endowed with the contact structure $\xi$ induced from $(S^{2n+1},\xis)$. Given a homotopy class $a \in \pi_1(M)$, choose a reference loop $\psi^a$ in $M$ and a symplectic trivialization of $(\psi^a)^*\xi$. When $a=0$ we ask that both $\psi^a$ and the trivialization are constant. Given a closed orbit $\ga: S^1 \to M$ with homotopy class $a$ and a homotopy between $\ga$ and $\psi^a$ we have an induced trivialization of $\ga^*\xi$. The assumption that $c_1(\xi)|_{H_2(M,\R)}=0$ implies that the homotopy class of this trivialization does not depend on the choice of the homotopy. Then the index of $\ga$ is the Conley-Zehnder index of the symplectic path given by the linearized Reeb flow along $\ga$ (restricted to the contact structure) with respect to the trivialization of $\ga^*\xi$. Note that, for non-trivial homotopy classes, this grading depends on the choice of the reference loops $\psi^a$ and of the trivializations of $(\psi^a)^*\xi$. In general, this grading has the following issue: this trivialization might be not  \emph{closed under iterations}, that is, the trivialization induced on $\gamma^j$ might be not homotopic with the $j$-th iterate of the  trivialization over $\gamma$.  This is a problem when we try to use index theory.

In order to fix this issue, we will define a \emph{fractional} grading using sections of the determinant line bundle in the following way \cite{McL,Sei}. Note that $c_1(\xi)$ is torsion and let $N$ be the smallest positive integer such that $Nc_1(\xi)=0$ ($N$ can be easily computed from the weights $\ell_0,\dots,\ell_n$; see Proposition \ref{prop:N}). Then $(\Lambda_\C^n\xi)^{\otimes N}$ is a trivial line bundle. Choose a trivialization $\tau: (\Lambda_\C^n\xi)^{\otimes N} \to M \times \C$ which corresponds to a choice of a non-vanishing section $\sec$ of $(\Lambda_\C^n\xi)^{\otimes N}$. The choice of this trivialization furnishes a unique way to symplectically  trivialize $\oplus_1^N\xi$ along periodic orbits of $\alpha$ up to homotopy. As a matter of fact, given a periodic orbit $\gamma$, let $\Phi: {\gamma^*\oplus_1^N\xi} \to S^1 \times \C^{nN}$  be a trivialization of $\oplus_1^N\xi$ over $\gamma$ as a Hermitian vector bundle such that its highest complex  exterior power coincides with $\tau$. This condition fixes the homotopy class of $\Phi$:  given any other such trivialization $\Psi$ we have, for every $t \in S^1$, that $\Phi_t \circ \Psi_t^{-1}: \C^{nN} \to \C^{nN}$  has complex determinant equal to one and therefore the Maslov index of the symplectic path $t \mapsto \Phi_t \circ \Psi_t^{-1}$ vanishes, where $\Phi_t:=\pi_2 \circ \Phi|_{{\gamma^*\oplus_1^N\xi}(t)}$ and  $\Psi_t:=\pi_2 \circ \Psi|_{{\gamma^*\oplus_1^N\xi}(t)}$ with $\pi_2: S^1 \times \C^{nN} \to \C^{nN}$ being the projection onto the second factor; cf. \cite{McL}. Notice that this trivialization is closed under iterations, that is, the trivialization induced on $\gamma^j$ coincides, up to homotopy, with the $j$-th iterate of the  trivialization over $\gamma$.

Now, one can define the Conley-Zehnder index $\cz(\gamma;\sec)$ of a closed orbit $\gamma$ in the following way. By the previous discussion, $\sec$ induces a unique up to homotopy symplectic trivialization $\Phi: {\gamma^*\oplus_1^N\xi} \to S^1 \times \R^{2nN}$. Using this trivialization, the linearized Reeb flow gives the symplectic path
\[
\Gamma(t) = \Phi_t \circ \oplus_1^N d\phi_\alpha^t(\gamma(0))|_\xi \circ \Phi_0^{-1},
\]
where $\phi^t_\alpha$ is the Reeb flow of $\alpha$. Then the Conley-Zehnder index is defined as
\[
\cz(\gamma;\sec)=\frac{\cz(\Ga)}{N}
\]
where the Conley-Zehnder index of $\Gamma$ is defined as the lowersemicontinuous extension of the usual Conley-Zehnder index for non-degenerate paths. It turns out that, since $H^1(M;\Q)=0$, this index does not depend on the choice of $\sec$ since every two such sections are homotopic;  see \cite[Lemma 4.3]{McL}.

Note that this grading is fractional in general. Even though the idea of a fractional grading may seem unnatural at first, it can be thought of as a way of keeping track of some information about the homotopy classes of the orbits. Indeed, given two homotopic orbits we have that their index difference is an integer. As a matter of fact, fixed a homotopy class $a$, the gradings obtained using sections of $(\Lambda_\C^n\xi)^{\otimes N}$ and reference loops coincide up to a constant. To see this, choose a reference loop $\psi^a$ and a trivialization of $(\psi^a)^*\xi$. It induces an obvious (product)  trivialization $\Psi^a$ of ${(\psi^a)^*\oplus_1^N\xi}$. Given a closed orbit $\ga$ with homotopy class $a$, we have the induced trivialization of ${\ga^*\oplus_1^N\xi}$. Let $\cz(\gamma;\Psi^a)$ be the index of $\ga$ using this trivialization. Let $\Phi^a$ be a trivialization ${(\psi^a)^*\oplus_1^N\xi}$ induced by a non-vanishing section $\sec$ of $(\Lambda_\C^n\xi)^{\otimes N}$. Then we have that
\begin{equation}
\label{eq:index difference}
\cz(\gamma;\sec)-\cz(\ga;\Psi^a)=c^a:=\frac{2\maslov(\Phi^a_t \circ (\Psi^a_t)^{-1})}{N}.
\end{equation}

Let $\bga$ be the underlying simple orbit of $\ga$, that is, $\bga$ is a simple orbit and $\ga=\bga^j$ for some $j \in \N$. Denote by $\bar a$ the homotopy class of $\bga$ (so that $a=\bar a^j$). We say that $\ga$ is good if
\[
\cz(\ga;\Psi^a)-\cz(\bga;\Psi^{\bar a}) \in 2\Z.
\]
Otherwise, it is called bad. From the previous discussion we have that $\ga$ is good if and only if
\[
\cz(\gamma;\sec)-\cz(\bga;\sec) \in c^a - c^{\bar a} + 2\Z.
\]

\subsection{Index computations and non-vanishing of the first Chern class in lens spaces}
\label{sec:computations}

This section contains index computations relevant for this work and shows how to deal with the non-vanishing of the first Chern class in lens spaces. First, we have the following result regarding the computation of $N$ in terms of the weights that define the lens space.

\begin{proposition}
\label{prop:N}
There is an isomorphism $H^2 (L_p^{2n+1} (\ell_0, \ell_1, \ldots, \ell_n); \Z) \cong \Z_p$ such that the first Chern class $c_1(\xi)$ is given by
\[
c_1 (\xi) = \sum_{i=0}^n \ell_i \modp.
\]
Hence,
\[
m \cdot c_1 (\xi) = 0 \Leftrightarrow m \cdot \sum_{i=0}^n \ell_i  = 0 \modp.
\]
In particular, we have that $N=\min\{m \in \N;\,m\sum_i\ell_i=0\modp\}$.
\end{proposition}

\begin{proof}
This is a particular case of Proposition 2.16 in~\cite{AM1}  and its proof. The main points are the following:
\begin{itemize}
\item[(i)] The quotient map from $S^{2n+1}$ to $L_p^{2n+1} (\ell_0, \ell_1, \ldots, \ell_n)$ is a principal $\Z_p$-bundle and its classifying map $f: L_p^{2n+1} (\ell_0, \ell_1, \ldots, \ell_n) \to B \Z_p$ induces an isomorphism $f^\ast : H^2 (B \Z_p; \Z) \to H^2 (L_p^{2n+1} (\ell_0, \ell_1, \ldots, \ell_n); \Z)$.
\item[(ii)] The $\Z_p$-action on $\C^{n+1}$ gives rise to an associated vector bundle $S^{2n+1} \times_{\Z_p} \C^{n+1}$ over $L_p^{2n+1} (\ell_0, \ell_1, \ldots, \ell_n)$, and
\[
c_1 (\xi) = c_1 (S^{2n+1} \times_{\Z_p} \C^{n+1}) = f^\ast c_1 (E\Z_p \times_{\Z_p}  \C^{n+1}).
\]
\item[(iii)] $H^2 (B \Z_p; \Z) \cong \Z_p \cong$ character group of $\Z_p$ and $c_1 (E\Z_p \times_{\Z_p}  \C^{n+1})$ is the sum of the characters that determine the $\Z_p$-action on $\C^{n+1}$.
\end{itemize}
\end{proof}

Given $a \in \pi_1(\L)$, consider on $\R^{(2n+2)N} \simeq \C^{(n+1)N}$ the Hamiltonian
\begin{equation}
\label{eq:G_a}
G_a(z_0,\dots,z_{n+(N-1)(n+1)})=\frac{\pi}{p}\bigg(\sum_{j=0}^{N-1}\sum_{i=0}^n \ell^a_i\|z_{i+j(n+1)}\|^2 - \bigg(N\sum_{i=0}^n \ell^a_i\bigg)\|z_{n+(N-1)(n+1)}\|^2\bigg)
\end{equation}
so that its Hamiltonian flow given by
\begin{align}
\label{eq:flow G_a}
\vr^{G_a}_t(z_0,\dots,z_{n+(N-1)(n+1)})=(&e^{\frac{2\pi\sqrt{-1}\ell^a_0t}{p}}z_0,\dots,e^{\frac{2\pi\sqrt{-1}\ell^a_nt}{p}}z_n,\dots,e^{\frac{2\pi\sqrt{-1}\ell^a_0t}{p}}z_{(n+1)(N-1)},\dots, \nonumber \\
& e^{\frac{2\pi\sqrt{-1}\ell^a_{n-1}t}{p}}z_{n+(n+1)(N-1)-1},e^{\frac{2\pi\sqrt{-1}(\ell^a_n-N\sum_i \ell^a_i)t}{p}}z_{n+(N-1)(n+1)})
\end{align}
generates a loop with Maslov index zero. We have that
\begin{equation}
\label{eq:sum l^a mod p}
\sum_{i=0}^n \ell^a_i = j_a\sum_{i=0}^n \ell_i \modp
\end{equation}
and so, by Proposition \ref{prop:N}, $N \sum_{i=0}^n \ell^a_i = 0\modp$. Therefore,
\begin{align*}
\vr^{G_a}_1(z_0,\dots,z_{n+(N-1)(n+1)}) = (&e^{\frac{2\pi\sqrt{-1}\ell^a_0}{p}}z_0,\dots,e^{\frac{2\pi\sqrt{-1}\ell^a_n}{p}}z_n,\dots,e^{\frac{2\pi\sqrt{-1}\ell^a_0}{p}}z_{(n+1)(N-1)},\dots,\\
& e^{\frac{2\pi\sqrt{-1}\ell^a_n}{p}}z_{n+(N-1)(n+1)}) \\
= (&\psi^{j_a}(z_0,\dots,z_n),\dots,\psi^{j_a}(z_{(n+1)(N-1)},\dots,z_{n+(N-1)(n+1)})).
\end{align*}

\begin{remark}
Note that the key reason why we take $N$ copies of $\C^{n+1}$ is to generate a loop with \emph{zero} Maslov index whose time one map is $(\psi^{j_a},\dots,\psi^{j_a})$.
\end{remark}

Let $\alpha$ be a contact form on $\L$ and $\ga$ a closed Reeb orbit with non-trivial homotopy class $a$. Multiplying $\alpha$ by a constant if necessary, we can assume that the period of $\ga$ is $1$. Let $\beta$ be the lift of $\alpha$ to $S^{2n+1}$ and $\hga$ a segment of Reeb orbit of $\beta$ that lifts $\ga$. Let $H_\beta: \R^{2n+2} \to \R$ be a homogeneous of degree two Hamiltonian such that $H_\beta^{-1}(1)=\Sigma_\beta$ and consider $\hga$ as a segment of Hamiltonian orbit of $H_\beta$ on $\Sigma_\beta$.

Denote by $\Ga_\beta: [0,1] \to \Sp((2n+2)N)$ the symplectic path given by $N$ copies of the linearized Hamiltonian flow of $H_\beta$ along $\hga$ with respect to the constant symplectic trivialization of $T\R^{(2n+2)N}$, that is,
\begin{equation}
\label{eq:Ga_beta}
\Ga_\beta(t)=\oplus_1^N D\vr^{H_\beta}_t(\hga(0))
\end{equation}
with $t \in [0,1]$. The next result shows how to compute the index of $\ga$ in terms of $\Ga_\beta$. Its proof follows \cite[Proposition 3.1]{AMM}.

\begin{proposition}
\label{prop:index}
We have that
\[
\cz(\ga)=\frac{\cz(\vr^{G_a}_{-t}\circ\Ga_\beta)}{N}+1.
\]
\end{proposition}

\begin{proof}
Consider the symplectization $W=\L\times\R$ of $\L$ endowed with the symplectic form $\om=d(e^r\alpha)$, where $r$ is the coordinate in the $\R$-component. Let $H: W \to \R$ be the Hamiltonian given by $H(x,r)=e^r$ so that its Hamiltonian vector field is given by $X_H(x,r)=(R_\alpha(x),0)$, where $R_\alpha$ is the Reeb vector field. Let $\ga_H$ be the Hamiltonian closed orbit of $H$ at $H^{-1}(1)$ corresponding to $\ga$.

Choose a section $\sec$ of $\Lambda^n_\C\xi^{\otimes N}$ and consider a trivialization $\Phi^\xi$ of $\ga_H^*\oplus_1^N \xi$ as a Hermitian vector bundle such that its highest complex exterior power coincides with $\sec$; see Section \ref{sec:grading}. We have that $\oplus_1^N TW = \oplus_1^N (\xi \oplus \xi^\om)$, where $\xi^\om$ is the symplectic orthogonal to $\xi$ which admits a global (symplectic) frame given by $R_\alpha$ and the vertical vector field $\partial_r$. Let $\Phi^{\xi^\om}$ be the symplectic trivialization of $\ga_H^*\oplus_1^N \xi^\om$ that sends the frame $\{R_\alpha,\partial_r\}^\oN$ to a fixed symplectic basis. Consider the symplectic trivialization of $\ga_H^*\oplus_1^N TW$ given by $\Phi:=\Phi^\xi\oplus\Phi^{\xi^\om}$ and let $\Ga_H=\Ga_H^\xi\oplus\Ga_H^{\xi^\om}$ be the symplectic path given by $\Phi_t \circ \oplus_1^N D\vr^H_t \circ \Phi_0^{-1}$, where $\Ga_H^\xi=\Phi^\xi_t \circ \oplus_1^N D\vr^H_t|_\xi \circ (\Phi^\xi_0)^{-1}$ and $\Ga_H^{\xi^\om}=\Phi^{\xi^\om}_t \circ \oplus_1^N D\vr^H_t|_{\xi^\om} \circ (\Phi^{\xi^\om}_0)^{-1}$.

We have that
\[
\cz(\ga_H)=\cz(\Ga_H)/N=(\cz(\Ga_H^\xi)+\cz(\Ga_H^{\xi^\om}))/N.
\]
By our choice of $H$, we clearly have that $\cz(\ga)=\cz(\Ga_H^\xi)/N$ and, since $R_\alpha$ and $\partial_r$ are both invariant under the linearized flow of $H$, $\Ga_H^{\xi^\om}$ is the path constant equal the identity and therefore
\[
\cz(\Ga_H^{\xi^\om})=-N.
\]
Consequently, we arrive at
\[
\cz(\ga_H)=\cz(\ga)-1.
\]
Hence, in order to prove the proposition, we need to show that
\begin{equation}
\label{eq:index H}
\cz(\ga_H)=\frac{\cz(\vr^{G_a}_{-t}\circ\Ga_\beta)}{N}.
\end{equation}

Let $(u_1,\dots,u_{(2n+2)N})$ be a symplectic frame of $\ga_H^*\oplus_1^N TW$ obtained via the trivialization $\Phi$. Let $\pi: \R^{2n+2}\setminus\{0\} \to TW$ be the quotient projection with respect to the $\Z_p$-action in $\R^{2n+2}\setminus\{0\}$. Let $v_i=(d\pi)^{-1}(u_i)$ be the frame of $\hga^*T\R^{(2n+2)N}$ given by the lift of $(u_i)$. Denote by $\Psi$ the trivialization of $\hga^*T\R^{(2n+2)N}$ induced by $(v_i)$. Clearly,
\[
\cz(\ga_H)=\cz(\hga;\Psi)
\]
where the index on the left hand side is computed using the frame $(u_i)$.

Now, consider the symplectic frame $(w_i)$ of $\hga^*T\R^{(2n+2)N}$ given by $w_i(t)=D\vr^{G_a}_t(\hga(0))(v_i(0))$ and let  $\Theta$ be the trivialization of  $\hga^*T\R^{(2n+2)N}$ induced by $(w_i)$. (Although $\vr^{G_a}_t$ is a linear flow, we take its derivative because the proof does not use this linearity; see Remark \ref{rmk:linearity1}.) It is easy to see that
\[
\cz(\hga;\Theta)=\frac{\cz(D\vr^{G_a}_t(\hga(0))^{-1}\circ\Ga_\beta)}{N},
\]
where the index on the right hand side is computed using the constant trivialization of $T\R^{(2n+2)N}$. Therefore, in order to prove \eqref{eq:index H}, it is enough to show that
\begin{equation}
\label{eq:indexes}
\cz(\hga;\Psi) = \cz(\hga;\Theta).
\end{equation}
To accomplish this equality, let $m$ be the smallest positive integer such that $\ga^m$ is contractible and consider the lift $\bga$ of $\ga^m$ to $\R^{2n+2}$ which is a closed orbit of $H_\beta$. Let $g:=\psi^{j_a}$ and consider the extensions of the frames $(v_i)$ and $(w_i)$ to $\bga^*T\R^{(2n+2)N}$ given by
\[
v'_i(t+j)=\oplus_1^N Dg^j(\hga(t))(v_i(t))\ \text{and}\ w'_i(t+j)=\oplus_1^N Dg^j(\hga(t))(w_i(t))
\]
for every $t \in [0,1]$ and $j \in \{0,\dots,m-1\}$. Since $g^j\circ\pi=\pi$ we have that $(v'_i)$ is the lift of the obvious extension $(u'_i)$ of the frame $(u_i)$ to $(\ga_H^m)^*\oplus_1^N TW$. It follows from this that
\begin{equation}
\label{eq:index Psi'}
\cz(\bga;\Psi')=\cz(\ga_H^m),
\end{equation}
where $\Psi'$ is the trivialization of $\bga^*T\R^{(2n+2)N}$ induced by $(v'_i)$ and the index on the right hand side  is computed using the frame $(u'_i)$. We also conclude from $g^j\circ\pi=\pi$ that
\[
w'_i(t)=D\vr^{G_a}_t(v_i(0)).
\]
Since $\vr^{G_a}_t$ generates a loop with vanishing Maslov index,
\begin{equation}
\label{eq:index Theta'}
\cz(\bga;\Theta')=\cz(\bga)
\end{equation}
where $\Theta'$ is the trivialization of $\hga^*T\R^{(2n+2)N}$ induced by $(w'_i)$ and the index on the right hand side is computed using the constant trivialization of $T\R^{(2n+2)N}$. It is easy to see that the right hand sides of \eqref{eq:index Psi'} and \eqref{eq:index Theta'} are equal. (Indeed, since $\ga^m$ is contractible, the index of its lift $\bga$ (computed using the constant trivialization of $T\R^{(2n+2)N}$) is equal to the index of $\ga^m_H$ using a trivialization defined over a capping disk (c.f. \cite[Lemma 3.4]{AM1}) which coincides with the index of $\ga^m_H$ using the frame $(u_i')$.) Hence, we arrive at
\begin{equation}
\label{eq:index Psi'=index Theta'}
\cz(\bga;\Psi') = \cz(\bga;\Theta').
\end{equation}

Consider the map $A': [0,m] \to \Sp((2n+2)N)$ uniquely defined by the property that $A'_t v'_i(t)=w'_i(t)$ for every $i$ and $t$. Using the fact that $g^j\circ\pi=\pi$ we conclude that
\[
v'_i(j)=w'_i(j)\ \text{for every}\ i \implies A'_{j}=\id
\]
for every $j \in \{0,\dots,m\}$. We have that
\begin{equation}
\label{eq:indexes hga}
\cz(\bga;\Psi') = \cz(\bga;\Theta') + \frac{2\maslov(A')}{N}
\end{equation}
which implies, by \eqref{eq:index Psi'=index Theta'}, that $\maslov(A')=0$. Similarly,
\begin{equation}
\label{eq:indexes bga}
\cz(\hga;\Psi) = \cz(\hga;\Theta) + \frac{2\maslov(A)}{N}
\end{equation}
where $A=A'|_{[0,1]}$. We claim that
\begin{equation}
\label{eq:indexes A and A'}
\maslov(A')=m\maslov(A).
\end{equation}
which would imply \eqref{eq:indexes}, finishing the proof. To prove the claim, note that
\begin{align*}
\oplus_1^N Dg^j(\hga(t))A_tv'_i(t) & = \oplus_1^N Dg^j(\hga(t))w'_i(t) \\
& = w'_i(t+j) \\
& = A'_{t+j}v'_i(t+j) \\
& = A'_{t+j}\oplus_1^N Dg^j(\hga(t))v'_i(t) \\
\end{align*}
for every $i$ and consequently
\begin{equation}
\label{eq:loops}
A'_{t+j} = \oplus_1^N Dg^j(\hga(t))A_t\oplus_1^N Dg^j(\hga(t))^{-1}
\end{equation}
for every $t \in [0,1]$ and $j \in \{0,\dots,m-1\}$. Now, let $g_s$ be a homotopy between $g_0=g$ and $g_1=\Id$ (the existence of this homotopy follows from the fact that $\psi^{j_a}$ is the time one map of a flow). Note that the curve $\oplus_1^N Dg^j(\hga(t))A_t\oplus_1^N Dg^j(\hga(t))^{-1}$ is homotopic to the curve $A_t$ via the homotopy $\oplus_1^N Dg_s^j(\hga(t))A_t\oplus_1^N Dg_s^j(\hga(t))^{-1}$. Thus, all the loops $t \in [0,1] \mapsto A'_{t+j}$, $j \in \{0,\dots,m-1\}$, are homotopic to the loop $t \in [0,1] \mapsto A_t$. Since $A'$ is a concatenation of these loops, we conclude the equality \eqref{eq:indexes A and A'}.
\end{proof}

\begin{remark}
\label{rmk:linearity1}
The previous proof does not use the linearity of the flow $\vr^{G_a}_t$: it uses only that it generates a loop with zero Maslov index. Therefore, if $(\psi^{j_a},\dots,\psi^{j_a})$ is the time one map of a Hamiltonian flow $\vr^{G_a}_t$ that generates a loop with vanishing Maslov index then
\[
\cz(\ga)=\frac{\cz(D\vr^{G_a}_{t}(\hga(0))^{-1}\circ\Ga_\beta)}{N}+1.
\]
\end{remark}

The next proposition shows how to compute $k_a$ easily from the weights that determine the lens space.

\begin{proposition}
\label{prop:k_a}
We have that
\[
k_a=w_-^a - w_+^a + \frac{2\sum_i \ell^a_i}{p}+1.
\]
\end{proposition}

\begin{proof}
Consider the Hamiltonian $H: \R^{2n+2} \to \R$ given by
\[
H(z_0,\dots,z_n)=\frac{j_a\pi}{p}\|z_0\|^2 + \sum_{i=1}^n \pi\ep_i\|z_i\|^2
\]
where the coefficients $j_a/p,\ep_1,\dots,\ep_n$ are rationally independent and the $\ep_i$'s are positive and very small. Consider the ellipsoid $E=H^{-1}(1)$ and let $\beta$ be the contact form on $S^{2n+1}$ such that $\Sigma_\beta=E$. Let $\alpha$ be the induced contact form on $\L$ (note that $E$ is invariant under the $\Z_p$-action).

Let $\ga$ be the closed orbit of $H$ given by $\ga(t)=(e^{2\pi\sqrt{-1}tj_a/p},0,\dots,0)$. Since $\ell^a_0=j_a \modp$ (by our normalization of the homotopy weights), given $a \in \pi_1(\L)$, we have that the closed orbit $\ga_a$ of $\alpha$ given by the projection of $\ga|_{[0,1]}$ (we are tacitly identifying $E$ with the unit sphere) has homotopy class $a$.

Choosing $\ep_1,\dots,\ep_n$ sufficiently small, it is easy to see that $\ga_a$ is the closed orbit with smallest index among the closed orbits of $\alpha$ with homotopy class $a$. Thus, it is enough to show that
\[
\cz(\ga_a)=w_-^a - w_+^a + \frac{2\sum_i \ell^a_i}{p}+1.
\]
Let $\Ga_\beta: [0,1] \to \Sp((2n+2)N)$ be the path given by $\Ga_\beta(t)=\oplus_1^N D\vr^H_t(\ga(0))$. By Proposition \ref{prop:index}, we have to prove that
\begin{equation}
\label{eq:index orbit}
\frac{\cz(\vr^{G_a}_{-t}\circ\Ga_\beta)}{N}=w_-^a - w_+^a + \frac{2\sum_i \ell^a_i}{p}.
\end{equation}
But
\[
\Ga_\beta(t)=\oplus_1^N (e^{2\pi\sqrt{-1}j_at/p},e^{2\pi\sqrt{-1}\ep_1t},\dots,e^{2\pi\sqrt{-1}\ep_nt}).
\]
From this and \eqref{eq:flow G_a}, we have that
\begin{align*}
\vr^{G_a}_{-t}\circ\Ga_\beta=(&e^{\frac{-2\pi\sqrt{-1}(\ell^a_0-j_a)t}{p}},e^{\frac{-2\pi\sqrt{-1}\ell^a_1t}{p}+2\pi\sqrt{-1}\ep_1t},\dots,e^{\frac{-2\pi\sqrt{-1}\ell^a_nt}{p}+2\pi\sqrt{-1}\ep_nt},\dots, \\
&e^{\frac{-2\pi\sqrt{-1}(\ell^a_0-j_a)t}{p}},e^{\frac{-2\pi\sqrt{-1}\ell^a_1t}{p}+2\pi\sqrt{-1}\ep_1t},\dots, e^{\frac{-2\pi\sqrt{-1}\ell^a_{n-1}t}{p}+2\pi\sqrt{-1}\ep_{n-1}t}, \\
& e^{\frac{-2\pi\sqrt{-1}(\ell^a_n-N\sum_i \ell^a_i)t}{p}+2\pi\sqrt{-1}\ep_nt}).
\end{align*}
Since the index of the path $e^{\frac{-2\pi\sqrt{-1}(\ell^a_0-j_a)t}{p}}$, $t \in [0,1]$, is $-1$ if  $\ell^a_0>0$ (i.e. $\ell^a_0=j_a$) and $1$ if $\ell^a_0<0$ (i.e. $\ell^a_0=j_a-p$), $N\sum_i \ell^a_i=0 \modp$ (by Proposition \ref{prop:N} and \eqref{eq:sum l^a mod p}) and the $\ep_i$'s are very small, we can easily see that
\[
\mu(\vr^{G_a}_{-t}\circ\Ga_\beta)=N(w_-^a - w_+^a) + \frac{2N\sum_i \ell^a_i}{p}
\]
concluding \eqref{eq:index orbit}.
\end{proof}

\subsection{Bott's function}
\label{sec:Bott}

Let $\Ga: [0,T] \to \Sp(2n)$ be a symplectic path starting at the identity and $P:=\Ga(T)$ its endpoint. Following \cite{Lon99,Lon02}, one can associate to $\Ga$ its Bott's function $\Bott: S^1 \to \Z$ which will be a crucial tool throughout this work. It has the following properties:
\begin{itemize}
\item[(a)] (Bott's formula) We have that $\cz(\Ga^k) = \sum_{z^k=1} \Bott(z)$ for every $k \in \N$. In particular, the mean index $\mi(\Ga):=\lim_{k\to\infty}\frac{\cz(\Ga^k)}{k}$ satisfies
\[
\mi(\Ga) = \int_{S^1} \Bott(z)\,dz,
\]
where the total measure of the circle is normalized to be equal to one.

\item[(b)] If $\Ga = \Ga_1 \oplus \Ga_2$ then $\Bott = \Bott_1 + \Bott_2$ where $\Bott_i$ is the Bott's function associated to $\Ga_i$ for $i=1,2$.

\item[(c)] If $\Ga_1$ and $\Ga_2$ are homotopic with fixed endpoints then $\Bott_1=\Bott_2$.

\item[(d)] The discontinuity points of $\Bott$ are contained in $\sigma(P) \cap S^1$, where $\sigma(P)$ is the spectrum of $P$.

\item[(e)] $\Bott(z)=\Bott(\bar z)$ for every $z \in S^1$.

\item[(f)] The \emph{splitting numbers} $S^\pm_z(P) := \lim_{\ep\to 0^+} \Bott(e^{\pm \sqrt{-1}\epsilon}z)-\Bott(z)$ depend only on $P$ and satisfy, for every $z \in S^1$,
\begin{equation}
\label{eq:splitting symmetry}
S^\pm_z(P)=S^\mp_{\bar z}(P),
\end{equation}
\begin{equation}
\label{eq:splitting additivity}
S^\pm_z(P_1\oplus P_2)=S^\pm_z(P_1)+S^\pm_z(P_2),
\end{equation}
\begin{equation}
\label{eq:bound splitting1}
0 \leq S^\pm_z(P) \leq \nu_z(P)
\end{equation}
and
\begin{equation}
\label{eq:bound splitting2}
S^+_z(P) + S^-_z(P) \leq \eta_z(P)
\end{equation}
where $\nu_z(P)$ and $\eta_z(P)$ are the geometric and algebraic multiplicities of $z$ (viewing $P$ as a complex matrix) respectively if $z \in \sigma(P) \cap S^1$ and zero otherwise. Moreover,
\begin{equation}
\label{eq:bott via splitting}
\Bott(e^{\sqrt{-1}\theta}) = \Bott(1) + S^+_1(P) + \sum_{\phi \in (0,\theta)} (S^+_{e^{\sqrt{-1}\phi}}(P)-S^-_{e^{\sqrt{-1}\phi}}(P)) - S^-_{e^{\sqrt{-1}\theta}}(P)
\end{equation}
for every $\theta \in [0,2\pi)$. (Note that the sum above makes sense since $S^\pm_z(P)\neq 0$ only for finitely many points $z \in S^1$.)

\item[(g)] $\Bott_\Ga(z)$ is lower semicontinuous with respect to $\Ga$ in the $C^0$-topology. More precisely, let $\P([0,T],\Sp(2n))$ be the set of continuous paths in $\Sp(2n)$ starting at the identity endowed with the $C^0$-topology. Then, for a fixed $z \in S^1$, the map
\[
\P([0,T],\Sp(2n)) \to \Z
\]
that sends $\Ga$ to $\Bott_\Ga(z)$, where $\Bott_\Ga$ denotes the Bott's function associated to $\Ga$, is lower semicontinuous, that is,
\[
\Bott_\Ga(z) = \sup_U \inf_{\Ga' \in U} \Bott_{\Ga'}(z),
\]
where the supremum runs over all $C^0$-neighborhoods $U$ of $\Ga$ in $\P([0,T],\Sp(2n))$.
\end{itemize}

We refer to \cite{Lon02} for a proof of these properties. In the proof of Theorem \ref{thm:main} the following comparison result, proved in \cite{GM} using the spectral flow, will play a crucial role.

\begin{theorem}\emph{(\cite[Theorem 2.2]{GM})}
\label{thm:comparison}
Let $\Ga_i: [0,T] \to \Sp(2n)$ ($i=1,2$) be two symplectic paths starting at the identity and satisfying the differential equation
\[
\frac{d}{dt}\Ga_i(t)= JA_i(t)\Ga_i(t),
\]
where $A_i(t)$ is a path of symmetric matrices. Suppose that $A_1(t) \geq A_2(t)$ for every $t$ and let $\Bott_i$ be the Bott's function associated to $\Ga_i$. Then
\[
\Bott_1(z) \geq \Bott_2(z)
\]
for every $z \in S^1$.
\end{theorem}

\section{Proof of Theorem \ref{thm:main}}
\label{sec:proof main}

\subsection{Idea of the proof}
Let $\beta$ be the lift of $\alpha$ to $S^{2n+1}$ and $H_\beta$ be the convex (resp. strictly convex) homogeneous of degree two Hamiltonian such that $H_\beta^{-1}(1)=\Sigma_\beta$. Let $\hga$ be a segment of Hamiltonian orbit of $H_\beta$ that projects onto $\ga$. Let $\Ga_\beta: [0,1] \to \Sp((2n+2)N)$ be the linearized Hamiltonian flow along $\hga$ (or, more precisely, $N$ copies of it as defined in \eqref{eq:Ga_beta}). It is not true, in general, that the index of $\ga$ equals the index of $\Ga_\beta$: as showed in Proposition \ref{prop:index}, we need to ``correct" the path $\Ga_\beta$ taking $\Ga:=\vr^{G_a}_{-t}\circ\Ga_\beta$ so that the ``correction term" $\vr^{G_a}_{-t}$ plays a key role.

The convexity (resp. strict convexity) of $\alpha$ implies, by our comparison result given by Theorem \ref{thm:comparison}, that
\[
\Bott_\Ga(z)\geq \Bott_{G_a}(z)\ (\text{resp.}\, \Bott_\Ga(z)\geq \Bott_{G^\ep_a}(z))
\]
where $G^\ep_a$ is a small ``negative" perturbation of $G_a$ given by \eqref{eq:G^ep_a} and $\Bott_\Ga$, $\Bott_{G_a}$ and $\Bott_{G^\ep_a}$ are the Bott's functions associated to $\Ga$, $\vr^{G_a}_{-t}$ and $\vr^{G^\ep_a}_{-t}$ respectively.

Then, a careful analysis of $\Bott_{G_a}$ and $\Bott_{G^\ep_a}$ allows us to show that the Bott's function $\Bott_\ga$ associated to the linearized Reeb flow along $\ga$ satisfies
\begin{itemize}
\item[(A)] $\Bott_\ga(1) \geq k_a$.
\item[(B)] There exists $z \in S^1\setminus\{1\}$ such that $\Bott_\ga(z)\geq h_a$ (resp. $\Bott_\ga(z)\geq \th_a$).
\item[(C)] Under the assumptions of Assertion 3, there exists $z \in S^1\setminus\{1\}$ such that $\Bott_\ga(z)\geq k_a+n$ (indeed, under these assumptions $h_a=k_a+n$ (resp. $\th_a=k_a+n$)).
\end{itemize}

Assertion A implies, by Bott's formula, that $\cz(\ga)\geq k_a$, proving Assertion 1. If $\ga$ is hyperbolic then $\Bott_\ga$ must be constant (because all the splitting numbers vanish). Hence, by Assertion B, if $\Bott_\ga(1)=\cz(\ga)<h_a$ (resp. $\cz(\ga)<\th_a$) then $\ga$ is non-hyperbolic, proving Assertion 2. Finally, we can show that if $\Bott_\ga(1)=\cz(\ga)=k_a$ and $\Bott_\ga(z)\geq k_a+n$ for some $z \in S^1$ (so that there is a big enough jump of Bott's function) then $\ga$ must be elliptic (Proposition \ref{prop:elliptic}). Thus, Assertion 3 follows from Assertion C. It was pointed out to us by a referee that some similar ideas appear in \cite{LWZ}.

\subsection{Proof of the theorem}

Assume, without loss of generality, that $\ga$ has period $1$. As in the previous section, let $\beta$ be the lift of $\alpha$ to $S^{2n+1}$ and $H_\beta$ be the convex (resp. strictly convex) homogeneous of degree two Hamiltonian such that $H_\beta^{-1}(1)=\Sigma_\beta$. Let $\hga$ be a segment of Hamiltonian orbit of $H_\beta$ that projects onto $\ga$ (here, as before, we are tacitly identifying  $S^{2n+1}$ and $\Sigma_\beta$). Let $\Ga_\beta: [0,1] \to \Sp((2n+2)N)$ be the path given by \eqref{eq:Ga_beta} and $\Ga=\vr^{G_a}_{-t}\circ\Ga_\beta$. The sympletic path $\Ga$ satisfies the differential equation
\[
\frac{d}{dt}\Ga(t)=JA(t)\Ga(t)
\]
where $J$ is the canonical complex structure in $\R^{(2n+2)N}$ and $A(t)$ is a path of symmetric matrices. It follows from \cite[Lemma 7.5]{GM} that
\[
A(t)=-\Hess G_a+(\vr^{G_a}_t)^*\big(\oplus_1^N \Hess H_\beta(\hga(t))\big)\vr^{G_a}_t
\]
where $(\vr^{G_a}_t)^*$ denotes the transpose of $\vr^{G_a}_t$ and $\oplus_1^N \Hess H_\beta(\hga(t))$ is the quadratic form on $\R^{(2n+2)N}$ given by $N$ copies of $\Hess H_\beta(\hga(t))$.

\begin{remark}
\label{rmk:linearity2}
If $\vr^{G_a}_t$ is not linear and $\Ga:=D\vr^{G_a}_{t}(\hga(0))^{-1}\circ\Ga_\beta$ (see Remark \ref{rmk:linearity1}) then $A(t)=-\Hess G_a+D\vr^{G_a}_t(\hga(0))^*\big(\oplus_1^N\Hess H_\beta(\hga(t))\big)D\vr^{G_a}_t(\hga(0))$.
\end{remark}

Given $\ep>0$, define the Hamiltonian
\begin{equation}
\label{eq:G^ep_a}
G^\ep_a(z_0,\dots,z_{n+(N-1)(n+1)})=\frac {\pi}{p}\bigg(\sum_{j=0}^{N-1}\sum_{i=0}^n (\ell^a_i-p\ep)\|z_{i+j(n+1)}\|^2 - \bigg(N\sum_{i=0}^n \ell^a_i\bigg)\|z_{n+(N-1)(n+1)}\|^2\bigg)
\end{equation}
whose flow is clearly linear. Since $H_\beta$ is convex (resp. strictly convex),
\[
\lg (\vr^{G_a}_t)^*\big(\oplus_1^N\Hess H_\beta(\hga(t))\big)\vr^{G_a}_tv,v \rg = \lg \big(\oplus_1^N\Hess H_\beta(\hga(t))\big)\vr^{G_a}_tv,\vr^{G_a}_tv \rg \geq 0\ (\text{resp.} > 0)
\]
for every $t$ and $v \in \R^{(2n+2)N}$. Hence, $A(t) \geq -\Hess G_a$ (resp. $A(t) \geq -\Hess G^\ep_a = 2\ep\Id -\Hess G_a$ for some $\ep>0$ sufficiently small). In what follows, we will take $\ep$ sufficiently small such that $\vr^{G^\ep_a}_{-1}$ has no eigenvalue $-1$.

Thus, Theorem \ref{thm:comparison} yields the inequality
\begin{equation}
\label{eq:comparison}
\Bott_\Ga(z)\geq \Bott_{G_a}(z)\ (\text{resp.}\, \Bott_\Ga(z)\geq \Bott_{G^\ep_a}(z))
\end{equation}
for every $z \in S^1$, where $\Bott_\Ga$ and $\Bott_{G_a}$ (resp. $\Bott_{G^\ep_a}$) are the Bott's functions associated to $\Ga$ and $\vr^{G_a}_{-t}$ (resp. $\vr^{G^\ep_a}_{-t}$) respectively. Using the arguments in the proofs of Proposition \ref{prop:index} and \cite[Proposition 4.1]{GM} we easily conclude that
\begin{equation}
\label{eq:Bott_ga}
\Bott_\ga(z) =
\begin{cases}
\Bott_\Ga(1)/N+1\ \text{if}\ z=1 \\
\Bott_\Ga(z)/N\ \text{otherwise},
\end{cases}
\end{equation}
for every $z \in S^1$.

Let us study $\Bott_{G_a}$. First, we easily derive from \eqref{eq:flow G_a} that
\begin{equation}
\label{eq:index G_a}
\Bott_{G_a}(1)=N(w_-^a - w_+^a) + \frac{2N\sum_i \ell^a_i}{p} = N(k_a-1),
\end{equation}
where the last equality follows from Proposition \ref{prop:k_a}. Thus,
\[
\cz(\ga) = \Bott_\Ga(1)/N+1 \geq \Bott_{G_a}(1)/N+1 = k_a,
\]
proving the first assertion of the theorem.

The proof of the remaining assertions relies on an analysis of $\Bott_{G_a}(z)$ and $\Bott_{G^\ep_a}(z)$ for $z \neq 1$. Once we know the Bott's function at $1$, it is completely determined by its splitting numbers; see \eqref{eq:bott via splitting}. The eigenvalues of $\vr^{G_a}_{-1}$ are the following:
\begin{itemize}
\item Those with negative imaginary part, given by $e^{-2\pi\sqrt{-1}\ell^a_i/p}$ with $\ell^a_i>0$ and $\ell^a_i\neq p/2$, having multiplicity $\mu^a_i$, and their complex conjugates.
\item Those with non-negative imaginary part, given by $e^{-2\pi\sqrt{-1}\ell^a_i/p}$ with $\ell^a_i<0$ or $\ell^a_i= p/2$, having multiplicity $\nu^a_i$, and their complex conjugates.
\end{itemize}
Since the splitting numbers satisfy \eqref{eq:splitting symmetry}, it is enough to determine them in the eigenvalues with non-negative imaginary part given by $e^{2\pi\sqrt{-1}\bell^a_i/p}$ for $i\in\{1,\dots,k\}$. (Recall that $\bell^a_1 < \dots < \bell^a_k$ are the absolute values of the homotopy weights.) From the multiplicity of the eigenvalues discussed above, the additivity of the splitting numbers \eqref{eq:splitting additivity} and \cite[List 9.1.12, page 198]{Lon02}, we conclude the following:
\begin{itemize}
\item $S^+(e^{2\pi\sqrt{-1}\bell^a_i/p})=\mu^a_i$ and $S^-(e^{2\pi\sqrt{-1}\bell^a_i/p})=\nu^a_i$ if $\ell^a_i\neq p/2$;
\item $S^+(e^{2\pi\sqrt{-1}\bell^a_i/p})=S^-(e^{2\pi\sqrt{-1}\bell^a_i/p})=\nu^a_i$ if $\ell^a_i=p/2$ (in this case $i=k$).
\end{itemize}
From this and \eqref{eq:index G_a} we can conclude that $\Bott_{G_a}$ is given by the following. Since $\Bott_{G_a}(z)=\Bott_{G_a}(\bar z)$ for every $z \in S^1$, it is enough to establish the values of $\Bott_{G_a}$ in the upper (closed) half-circle.
We have that if $\bell^a_k\neq p/2$ (resp. $\bell^a_k=p/2$) then
\begin{equation}
\label{eq:Bott_Ga}
\frac{\Bott_{G_a}(e^{\sqrt{-1}\theta})}{N} =
\begin{cases}
k_a-1\ \text{if}\ \theta \in [0,2\pi\bell^a_1/p) \\
k_a-1+\sum_{i=1}^{j-1}\mu^a_i - \sum_{i=1}^j\nu^a_i\ \text{if}\ \theta=2\pi\bell^a_j/p,\,j=1,\dots,k\ (\text{resp.}\,j=1,\dots,k-1) \\
k_a-1+\sum_{i=1}^{j}\mu^a_i - \sum_{i=1}^j\nu^a_i\ \text{if}\ \theta \in (2\pi\bell^a_j/p,2\pi\bell^a_{j+1}/p),\,j=1,\dots,k-1 \\
k_a-1+\sum_{i=1}^{k}\mu^a_i - \sum_{i=1}^k\nu^a_i\ \text{if}\ \theta \in (2\pi\bell^a_k/p,\pi]\ (\text{resp.}\,\theta=2\pi\bell^a_k/p=\pi),
\end{cases}
\end{equation}
for every $\theta \in [0,\pi]$; see Figures \ref{fig:Bott-G_a-1} and \ref{fig:Bott-G_a-2}. Note that the function is lower semicontinuous since the splitting numbers are non-negative.

\begin{figure}[ht]
\includegraphics[width=4in, height=2.1in]{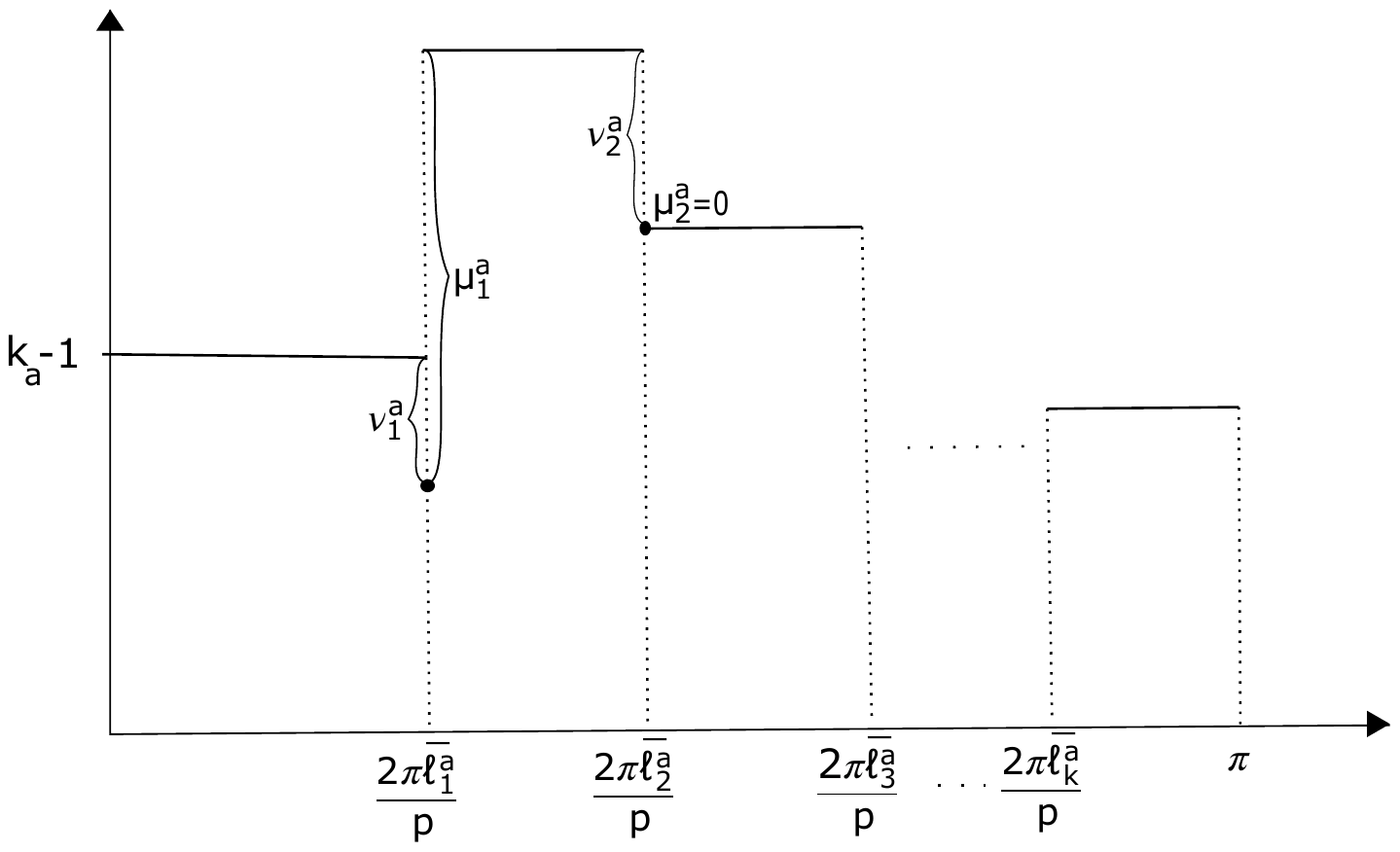}
\centering
\caption{The function $\Bott_{G_a}/N$ with the jumps $\mu^a_i$ and $\nu^a_i$ when $\bell^a_k\neq p/2$. The first jump goes down and the second goes up (when we move counterclockwise in the circle or, equivalently, move to the right in the figure).} 
\label{fig:Bott-G_a-1}
\end{figure}

\begin{figure}[ht]
\includegraphics[width=4in, height=2.1in]{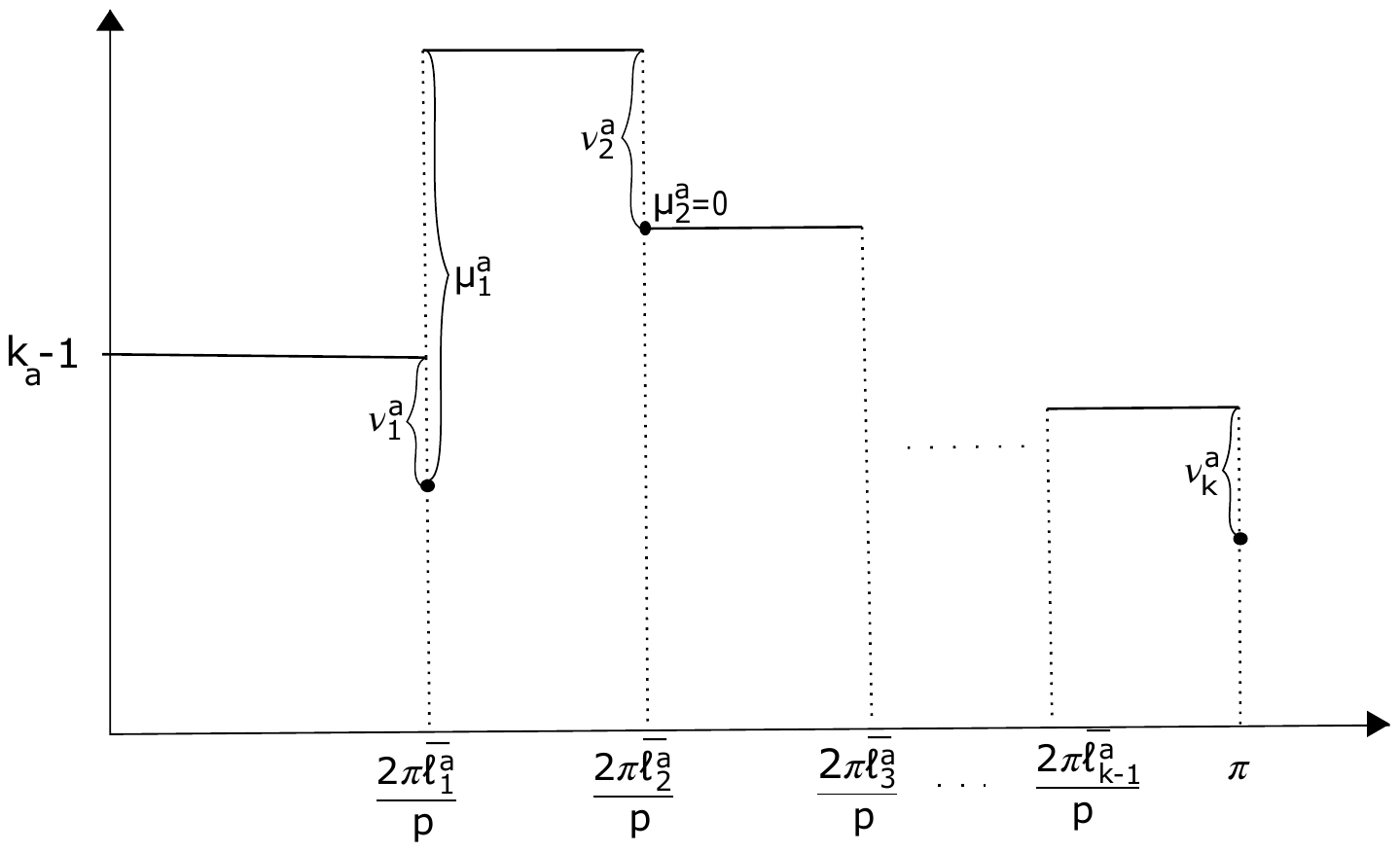}
\centering
\caption{The function $\Bott_{G_a}/N$ with the jumps $\mu^a_i$ and $\nu^a_i$ when $\bell^a_k=p/2$. The jump at $\pi=2\pi\bell^a_k/p$ goes down.}
\label{fig:Bott-G_a-2}
\end{figure}

From this, we conclude that
\[
\max_{z \in S^1} \Bott_{G_a}=\max_{z \in S^1\setminus\{1\}} \Bott_{G_a}=Nh_a.
\]
Using this, \eqref{eq:comparison} and \eqref{eq:Bott_ga} we infer that there exists $z\in S^1\setminus\{1\}$ such that $\Bott_\ga(z)\geq h_a$. Hence, if $\cz(\ga)<h_a$ then $\Bott_\ga(1)<\Bott_\ga(z)$ which implies that $\Bott_\ga$ is not constant and therefore $\ga$ is not hyperbolic, proving the second assertion of the theorem when $\alpha$ is not necessarily strictly convex.

When $\al$ is strictly convex, we have to study $\Bott_{G^\ep_a}$. Taking $\ep>0$ sufficiently small, we conclude that the eigenvalues of $\vr^{G^\ep_a}_{-1}$ are the following:
\begin{itemize}
\item Those with negative imaginary part, given by $e^{-2\pi\sqrt{-1}(\ell^a_i/p-\ep)}$ with $\ell^a_i>0$, having multiplicity $\tmu^a_i$, and their complex conjugates.
\item Those with positive imaginary part, given by $e^{-2\pi\sqrt{-1}(\ell^a_i/p-\ep)}$ with $\ell^a_i<0$, having multiplicity $\tnu^a_i$, and their complex conjugates.
\end{itemize}
The possibly non-zero splitting numbers of $\Bott_{G^\ep_a}$ in the upper-half circle are the following (see \cite[List 9.1.12, page 198]{Lon02}):
\begin{itemize}
\item $S^+(e^{2\pi\sqrt{-1}(\bell^a_i/p-\ep)})=\tmu^a_i$ and $S^-(e^{2\pi\sqrt{-1}(\bell^a_i/p-\ep)})=0$;
\item $S^+(e^{2\pi\sqrt{-1}(\bell^a_i/p+\ep)})=0$ and $S^-(e^{2\pi\sqrt{-1}(\bell^a_i/p+\ep)})=\tnu^a_i$ if $\bell^a_i\neq p/2$.
\end{itemize}
(Recall here that, by our choice of $\ep$, $\vr^{G^\ep_a}_{-1}$ has no eigenvalue $-1$.) Now, note that an argument analogous to the one to derive \eqref{eq:index G_a} shows that $\Bott_{\Ga^a_\ep}=N(k_a-1)$. Using this and our knowledge about the splitting numbers we can deduce that if $\bell^a_k\neq p/2$ (resp. $\bell^a_k=p/2$) then
\begin{equation}
\label{eq:Bott_Ga_ep}
\frac{\Bott_{G^\ep_a}(e^{\sqrt{-1}\theta})}{N} =
\begin{cases}
k_a-1\ \text{if}\ \theta \in [0,2\pi(\bell^a_1/p-\ep)] \\
k_a-1+\sum_{i=1}^{j}\tmu^a_i - \sum_{i=0}^{j-1}\tnu^a_i\ \text{if}\ \theta \in (2\pi(\bell^a_j/p-\ep),2\pi(\bell^a_j/p+\ep)) \\
k_a-1+\sum_{i=1}^{j}\tmu^a_i - \sum_{i=1}^j\tnu^a_i\ \text{if}\ \theta \in [2\pi(\bell^a_j/p+\ep),2\pi(\bell^a_{j+1}/p-\ep)] \\
k_a-1+\sum_{i=1}^{k}\tmu^a_i - \sum_{i=0}^{k-1}\tnu^a_i\ \text{if}\ \theta \in (2\pi(\bell^a_k/p-\ep),2\pi(\bell^a_k/p+\ep))\ (\text{resp.}\,\theta \in (\pi(1-2\ep),\pi]) \\
k_a-1+\sum_{i=1}^{k}\tmu^a_i - \sum_{i=1}^{k}\tnu^a_i\ \text{if}\ \theta \in [2\pi(\bell^a_k/p+\ep),\pi]\ \text{and}\ \bell^a_k\neq p/2
\end{cases}
\end{equation}
for every $\theta \in [0,\pi]$ and $j \in \{1,\dots,k-1\}$; see Figures \ref{fig:Bott-G_a_epsilon-1} and \ref{fig:Bott-G_a_epsilon-2}. Note that when $\bell^a_k=p/2$, $\Bott_{G^\ep_a}$ has a jump that goes up at $\theta=\pi(1-2\ep)$ while $\Bott_{G_a}$ has a jump that goes down at $\theta=\pi$.

\begin{figure}[ht]
\includegraphics[width=4in, height=2.1in]{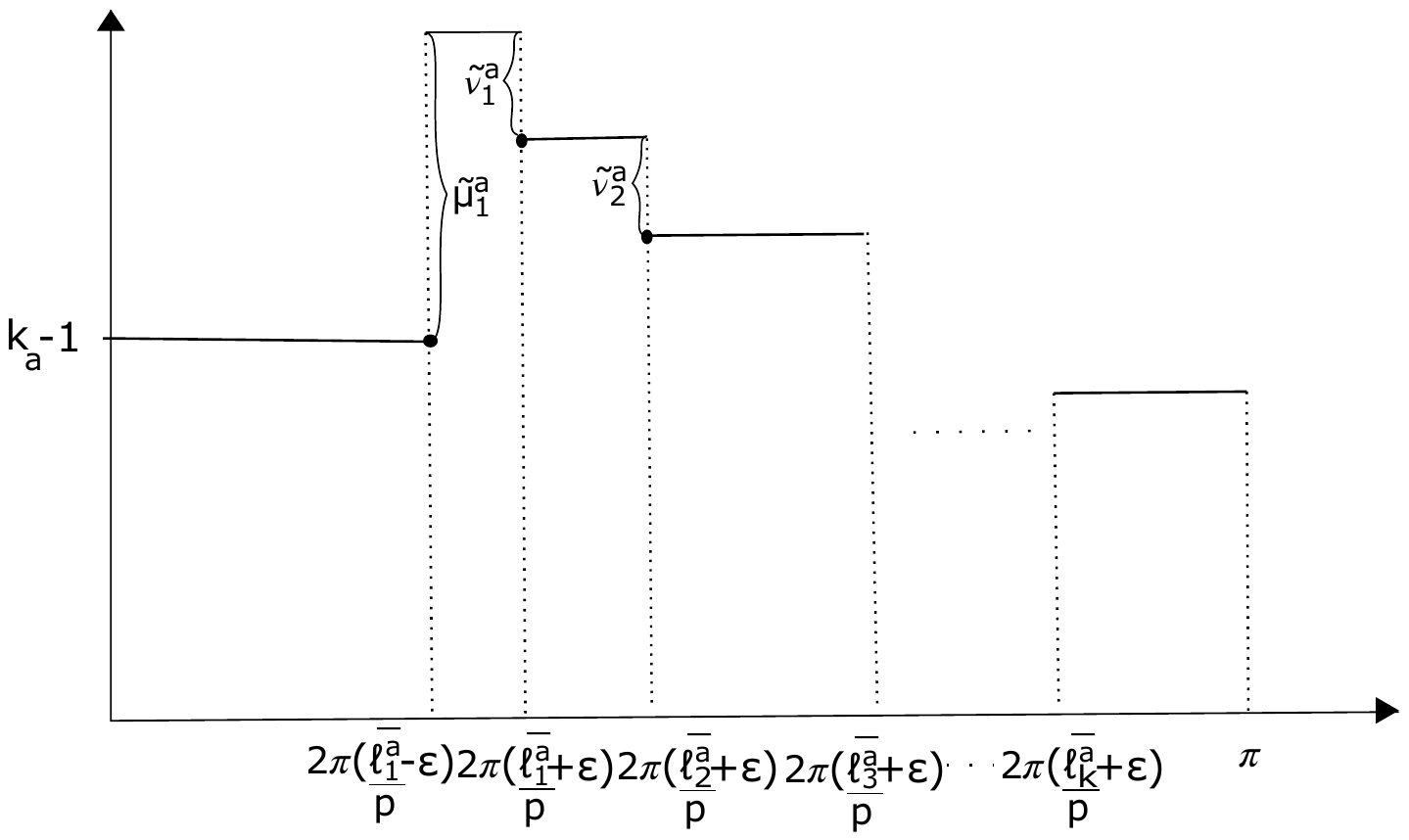}
\centering
\caption{The function $\Bott_{G^\ep_a}/N$ with the jumps $\tmu^a_i$ and $\tnu^a_i$ when $\bell^a_k\neq p/2$. The first jump goes up and happens when $\theta=2\pi\big(\frac{\bell^a_1}{p}-\ep\big)$. The second one goes down and happens when $\theta=2\pi\big(\frac{\bell^a_1}{p}+\ep\big)$.}
\label{fig:Bott-G_a_epsilon-1}
\end{figure}

\begin{figure}[ht]
\includegraphics[width=4in, height=2.1in]{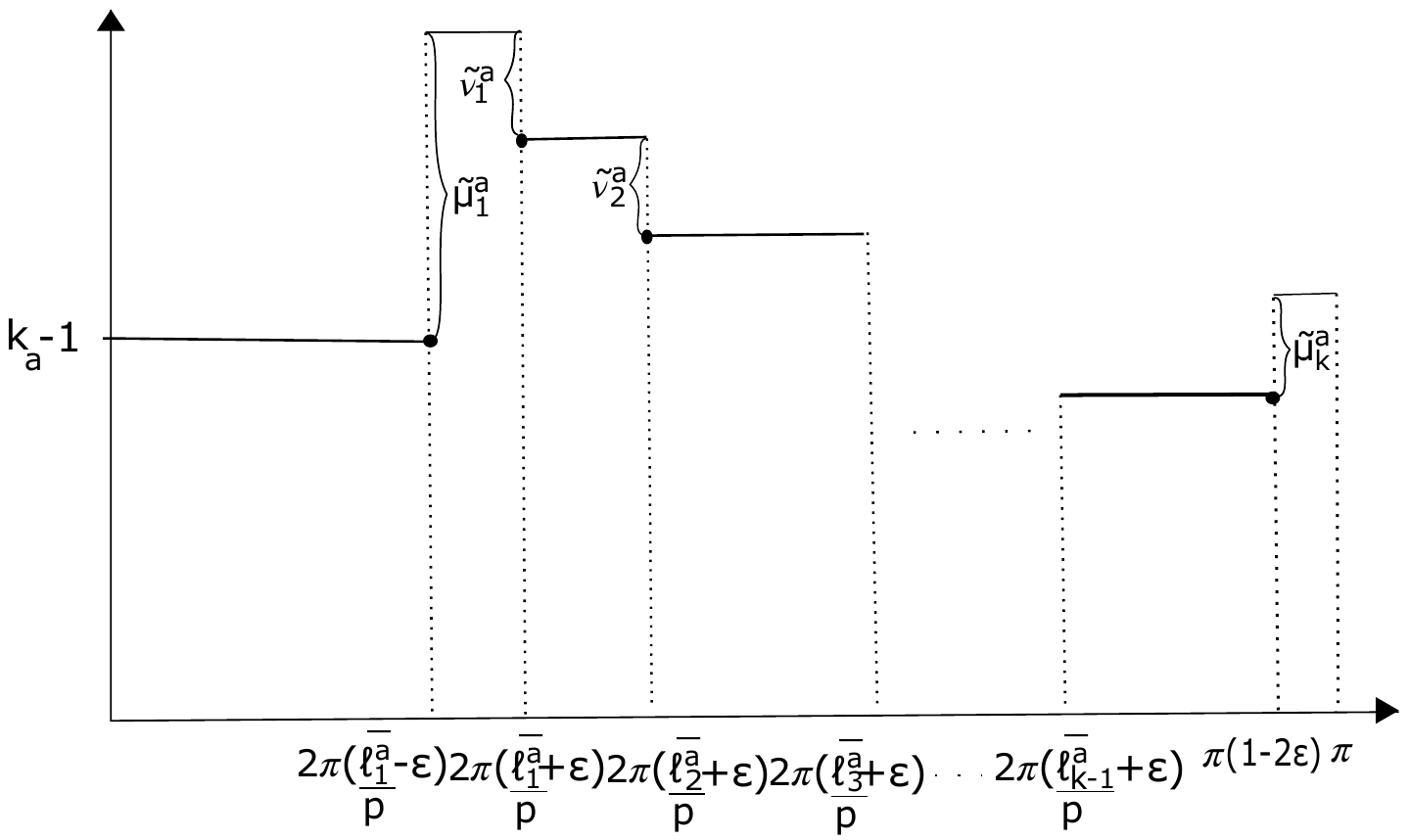}
\centering
\caption{The function $\Bott_{G^\ep_a}/N$ with the jumps $\tmu^a_i$ and $\tnu^a_i$ when $\bell^a_k=p/2$. The jump at $\pi(1-2\ep)=2\pi\big(\frac{\bell^a_k}{p}-\ep\big)$ goes up.}
\label{fig:Bott-G_a_epsilon-2}
\end{figure}

From this, we conclude that
\[
\max_{z \in S^1} \Bott_{G^\ep_a}=\max_{z \in S^1\setminus\{1\}} \Bott_{G^\ep_a}=N\th_a.
\]
Using this, \eqref{eq:comparison} and \eqref{eq:Bott_ga} we conclude that there exists $z\in S^1\setminus\{1\}$ such that $\Bott_\ga(z)\geq \th_a$. Thus, if $\cz(\ga)<\th_a$ then $\Bott_\ga$ is not constant and therefore $\ga$ is not hyperbolic, proving the second assertion of the theorem when $\alpha$ is strictly convex.

To prove the third assertion, note that from our assumptions that $\ell^a_i>0$ and $\ell^a_i\neq p/2$ (resp. $\ell^a_i>0$) for every $i$ we have that $h_a=k_a+n$ (resp. $\th_a=k_a+n$) and therefore $\max_{z \in S^1\setminus\{1\}} \Bott_{G_a}=N(k_a+n)$ (resp. $\max_{z \in S^1\setminus\{1\}} \Bott_{G^\ep_a}=N(k_a+n)$). Using this, \eqref{eq:comparison} and \eqref{eq:Bott_ga} we conclude that there exists $z\in S^1\setminus\{1\}$ such that $\Bott_\ga(z)\geq k_a+n$. Now, the result readily follows from the following proposition. Recall that a symplectic path $\Ga: [0,T] \to \Sp(2n)$ is elliptic if every eigenvalue of $\Ga(T)$ has modulus one.

\begin{proposition}
\label{prop:elliptic}
Let $\Ga: [0,T] \to \Sp(2n)$ be a symplectic path starting at the identity such that $\Bott_\Ga(z)-\Bott_\Ga(1)\geq n$ for some $z\in S^1$. Then $\Ga$ is elliptic.
\end{proposition}

\begin{proof}
Let $\theta$ be the smallest positive number such that $z=e^{\sqrt{-1}\theta}$. Let $P=\Ga(T)$ and $\{e^{\sqrt{-1}\theta_1},\dots,e^{\sqrt{-1}\theta_l}\}$ the eigenvalues of $P$ with modulus one and argument $0 < \theta_k < \theta$,  such that $\theta_k < \theta_{k+1}$ for every $1\leq k \leq l-1$. By \eqref{eq:bott via splitting},
\[
\Bott_\Gamma (z) - \Bott_\Gamma(1) = (S^+_1(P)- S^-_z(P)) + \sum_{k=1}^l (S^+_{e^{\sqrt{-1}\theta_k}}(P) - S^-_{e^{\sqrt{-1}\theta_k}}(P)).
\]
Since the splitting numbers are non-negative (see \eqref{eq:bound splitting1}), it follows from our assumption that
\[
S^+_1(P) + \sum_{k=1}^l S^+_{e^{\sqrt{-1}\theta_k}}(P) \geq n.
\]
By \eqref{eq:splitting symmetry}, this implies that
\[
S^-_1(P) + \sum_{k=1}^l S^-_{e^{-\sqrt{-1}\theta_k}}(P) \geq n.
\]
Putting these two inequalities together we arrive at
\begin{align*}
2n \geq \sum_{z\in S^1} \eta_z(P) & \geq \sum_{z \in S^1} S^+_z(P)+S^-_z(P) \\
& \geq S^+_1(P) + S^-_1(P) + \sum_{k=1}^l (S^+_{e^{\sqrt{-1}\theta_k}}(P) + S^-_{e^{-\sqrt{-1}\theta_k}}(P)) \geq 2n,
\end{align*}
where the first inequality is trivial by a dimensional reason, the second inequality holds by \eqref{eq:bound splitting2} and the third inequality follows again from the fact that the splitting numbers are non-negative. This ensures that all of the above inequalities are in fact equalities. Hence, $\Ga$ is elliptic.
\end{proof}

\section{Proof of Corollary \ref{cor:positive homotopy}}
\label{sec:proof positive homotopy}

We have only to prove the existence/non-existence of positive homotopy classes. In all the assertions, we have that the weights assume only two values 1 and $q$. Let $a$ be the homotopy class such that $j_a=1$. Firstly, note that when $1\leq q \leq p/2$ we have that $a$ is obviously positive.

Suppose now that $-p/2 < q \leq -1$. The integers $mq$, with $m\in\N$, form a negative arithmetic sequence with first term bigger that $-p/2$ and with distance less that $p/2$ between consecutive terms. Hence, there exists a smallest positive $k\in\N$ such that
\[
-p < kq \leq -p/2.
\]
Moreover, since $-p/2 < q \leq -1$, we also have that
\[
2 \leq k \leq \lceil p/2\rceil 
\]
and
\[
k>p/2 \ \ \text{iff}\ \ \text{$p$ is odd and $q=-1$.}
\]
Hence, when $q=-1$ and $p$ is odd, the homotopy weights of $a^j$ are $j$ and $-j$ for all $1\leq j \leq p/2$, and there are no positive homotopy classes in the fundamental group of $\L$. On the other hand, when $-p/2 < q < -1$ or when $q=-1$ and $p$ is even, the homotopy weights of $a^k$ are $k > 0$, $p+kq >0$, and we have that the homotopy class $a^k$ is positive. Note that when $q=-1$ and $p$ is even we have that $k=p/2$ and hence $a^{p/2}$ is the unique positive homotopy class.

\section{Proof of Theorem \ref{thm:dc}}
\label{sec:proof dc}

Let $M=L^{2n+1}_p(1,\dots,1)$ be the lens space endowed with the induced contact structure $\xi$ and $N$ be the minimal positive integer such that $Nc_1(\xi)=0$. In the proof we will take advantage of the following useful fact. A contact form $\alpha$ on a lens space $\L$ is called toric if its lift $\beta$ to $S^{2n+1}$ satisfies the property that $\Sigma_\beta$ is a linear ellipsoid. (Although this is not the standard definition, the standard definition is equivalent to this one \cite{Le}.)

\begin{proposition}
\label{prop:circle action}
Let $a \in \pi_1(\L)$ and $m$ be a positive integer such that $j_a$ and $m$ are coprime. Given $i \in \{0,\dots,n\}$ let $\hell_i=\ell_i \modp$ such that $1\leq \hell_i \leq p-1$. Then there exists a toric contact form on $L$ whose Reeb flow generates a circle action whose generic (simple) orbit has homotopy class $a$ and Conley-Zehnder index $\frac{2(\sum_i \hell_i)}{p}j_a+2m-n$.
\end{proposition}

\begin{proof}
We will use the toric geometry description of $L_p^{2n+1} (\ell_0 =1, \ell_1, \ldots, \ell_n)$ and its toric Reeb vectors presented in \cite[Section 2]{AMM} (see in particular Subsection 2.3.3).

The moment cone $C\subset\R^{n+1}$ of $L_p^{2n+1} (\ell_0 =1, \ell_1, \ldots, \ell_n)$ can be chosen to have primitive integral exterior normals $\nu_0, \ldots, \nu_{n} \in \Z^{n+1}$ given by 
\[
\nu_0=\begin{bmatrix}
-\hat{\ell}_1\\ 
\vdots\\ 
-\hat{\ell}_{n-1}\\ 
p\\ 
-\sum_{i=1}^n \hat{\ell}_i
\end{bmatrix},\,\,
 \nu_j=\begin{bmatrix}
\ \\ 
e_j\\
\ \\
0\\
1
\end{bmatrix}\textup{ for }j=1, \ldots, n-1,\,\, \nu_n=\begin{bmatrix}
0\\
\vdots\\ 
0\\ 
0\\ 
1
\end{bmatrix},
\]
where $e_1, \ldots, e_{n-1} \in \Z^{n-1}$ are the canonical coordinate basis vectors of $\Z^{n-1}$. The fundamental group $\pi_1 (L_p^{2n+1} (\ell_0 =1, \ell_1, \ldots, \ell_n))$ is naturally isomorphic to $\Z^{n+1}/\Nn \cong \Z_p$, where $\Nn$ is the $\Z$-span of $\{\nu_0, \ldots, \nu_{n} \}$. More precisely, 
\[
[(a_0, \ldots, a_n)] = a_{n-1} \modp.
\]

Any toric Reeb vector on $L_p^{2n+1} (\ell_0 =1, \ell_1, \ldots, \ell_n)$ can be written as a positive linear combination of the normals $\nu_0, \ldots, \nu_{n} \in \Z^{n+1}$. Consider the following toric Reeb vector:
\[
R_a = \left(\sum_{i=0}^n \frac{\hat{\ell}_i j_a}{p} \nu_i \right) + m \nu_n = 
\begin{bmatrix}
0\\ 
\vdots\\ 
0\\ 
j_a\\ 
m
\end{bmatrix}.
\]
Its generic (simple) orbit $\gamma$ has homotopy class $a$, corresponding to $j_a$, and period $1$. We will now use the set-up of Section \ref{sec:computations} to compute its Conley-Zehnder index.

The lift of $\gamma$ to $\R^{2(n+1)} = \C^{n+1}$ corresponds to the linear symplectic path
\[
\Gamma_t (z_0, z_1, \ldots, z_n) = \left(e^{2\pi\sqrt{-1} \frac{ \hat{\ell}_0 j_a}{p} t} z_0,  \ldots, e^{ 2\pi\sqrt{-1}  \frac{ \hat{\ell}_{n-1} j_a}{p} t} z_{n-1}, 
e^{2\pi\sqrt{-1}  \left( \frac{\hat{\ell}_n j_a}{p} + m \right) t} z_n \right).
\]
Hence, we have that
\begin{align*}
\cz \left(\varphi_{-t}^{G_a} \circ (\oplus_1^N \Gamma_t)\right)  & =  2N \left(\sum_{i=0}^n \frac{\hat{\ell}_i j_a -\ell_i^a}{p} -
\frac{n+1}{2} + m + \sum_{i=0}^n \frac{\ell_i^a}{p}   \right) \\
& = 2N \left( \left(\sum_{i=0}^n \hat{\ell}_i\right) \frac{j_a}{p} + m - \frac{n+1}{2}   \right),
\end{align*}
which implies that
\[
\cz (\gamma) = \frac{\cz \left(\varphi_{-t}^{G_a} \circ (\oplus_1^N \Gamma_t)\right) }{N} + 1 =
2 \left( \left(\sum_{i=0}^n \hat{\ell}_i\right) \frac{j_a}{p} + m - \frac{n}{2}   \right),
\]
as desired.
\end{proof}

\subsection{Proof of Assertion 1}
\label{sec:proof dc1}

Let $a \in \pi_1(M)$ be such that $j_a=p-1$ so that $k_a=\frac{2n+2}{p}(p-1)-n$. Let $m=1$ and $\balpha$ be the contact form given by Proposition \ref{prop:circle action} corresponding to $a$ and $m$. The Conley-Zehnder index of the generic orbit of $\balpha$ is given by
\[
\frac{2n+2}{p}(p-1)+2-n=2n+4- \frac{2n+2}{p} - n.
\]
We will need the following elementary lemma.

\begin{lemma}
Suppose that $p\geq 3$. If $n=1$ (resp. $n=2$), assume further that $p\geq 5$ (resp. $p\geq 4$). Then there exists an even integer $\Delta$ such that
\begin{equation}
\label{eq:Ct}
\frac{2n+4}{p} \leq \Delta - \frac{2n+2}{p} < 2n+2 - \frac{2n+2}{p}.
\end{equation}
\end{lemma}

\begin{proof}
Note that \eqref{eq:Ct} is equivalent to
\[
\frac{4n+6}{p} \leq \Delta < 2n+2.
\]
Hence, we have to show that the interval $[(4n+6)/p,2n+2)$ contains at least one even integer. It is clear if $n=1$ and $p\geq 5$ or $n=2$ and $p\geq 4$. To prove it for $n\geq 3$, note that, since $p\geq 3$, it is enough to show that the interval $[(4n+6)/3,2n+2)$ contains at least two integer numbers. But it follows from the fact that the Lebesgue measure of $[(4n+6)/3,2n+2)$ is bigger than or equal to two whenever $n\geq 3$.
\end{proof}

Now, let $F: \R^2 \to \R$ be defined as $F(q,p)=\pi(q^2+p^2)$. Given $\ep>0$ consider the Hamiltonian $G: \R^{2n} \to \R$ given by
\[
G(q_1,p_1,\dots,q_n,p_n) = - \big(\frac{2n+4-\Delta+\ep}{2}F(q_1,p_1) + \sum_{i=2}^n \ep F(q_i,p_i) \big).
\]
Let $f: \R \to \R$  be a smooth function such that
\begin{itemize}
\item $f(0)=1$, $f'(0)=1$;
\item $f'(r) \in (0,1)$ and $f''(r)>0$ for every $r\in (-r_0,0)$ and some $r_0>0$;
\item $f(r)=C$ for every $r\leq -r_0$, where $C$ is a constant bigger than $1/2$.
\end{itemize}
Define the Hamiltonian
\[
H=f \circ G,
\]
where the constants $\ep$ and $r_0$ will be properly chosen. 

\begin{remark}
\label{rmk:small1}
Note that, choosing $\ep<1$, $r_0$ very small and $C$ close enough to 1, we can make $H$ arbitrarily uniformly $C^1$-close to the constant function equal to one.
\end{remark}

Consider the lens space $M$ as the prequantization circle bundle of an orbifold $B$ given by the quotient of $M$ by the circle action generated by the Reeb flow of $\balpha$. Let $\pi: M \to B$ be the quotient projection. Take a point $x_0$ in the smooth part of $B$ and a neighborhood $U$ of $x_0$ with Darboux coordinates $(q_1,p_1,\dots,q_n,p_n)$ identifying $x_0$ with the origin. Taking $r_0$ sufficiently small and viewing $H$ as an Hamiltonian on $U$, extend $H$ to $B$ setting $H|_{B\setminus U}\equiv C$. Define the contact form
\[
\alpha=\balpha/\hat H,
\]
where $\hat H=H\circ\pi$. The Reeb vector field of $\alpha$ is given by
\[
R_\alpha=\hat HR_{\balpha}+ \hat X_{H},
\]
where $R_{\balpha}$ is the Reeb vector field of $\balpha$ (assume that the generic orbit of $\balpha$ has minimal period one) and $\hat X_{H}$ is the horizontal lift of the Hamiltonian vector field of $H$. By the construction of $H$, clearly $W:=\pi^{-1}(U)$ is invariant under the Reeb flow of $\alpha$ and outside $W$ the Reeb vector field of $\alpha$ is a constant multiple of $R_{\balpha}$.

\begin{lemma}
\label{lemma:dynconvex1}
The contact form $\alpha$ is dynamically convex.
\end{lemma}

Let $\ga_0$ be the simple closed orbit of $R_\alpha$ over $x_0$. Recall that, by construction, the homotopy class of $\ga_0$ is $a$.

\begin{lemma}
\label{lemma:not a-dynconvex}
We have that $\cz(\ga_0) < k_a$.
\end{lemma}

The proof of first assertion of Theorem \ref{thm:dc} follows immediately from Lemmas \ref{lemma:dynconvex1} and \ref{lemma:not a-dynconvex}. These two lemmas are proved in the next section.

\subsection{Proof of Lemmas \ref{lemma:dynconvex1} and \ref{lemma:not a-dynconvex}}

\subsubsection{Proof of Lemma \ref{lemma:dynconvex1}}
\label{sec:dynconvex}

Consider on $W \simeq U \times S^1$ the coordinates $(x,\theta)$, where $x=(q_1,p_1,\dots,q_n,p_n)$ and $\theta$ is the coordinate along the fiber such that, with respect to these coordinates, $x_0$ is the origin and
\[
\balpha|_W=\lambda+d\theta,
\]
where $\lambda=\frac 12\sum_{i=1}^n (q_idp_i-p_idq_i)$ is the Liouville form. In what follows, for a shorter notation, let $TU^N = \opn TU$ and $
\xi^N = \opn \xi$. Let $\ga$ be a contractible periodic orbit of $\alpha$. Clearly, if $\ga$ lies outside $W$ then $\mu(\ga)\geq n+2$ because $\balpha$ is toric (and therefore dynamically convex). So suppose that the image of $\ga$ is contained in $W$. The Darboux coordinates induce an obvious (constant) trivialization $D: TU^N \to U \times \R^{2nN}$. From this we get a trivialization of $\xi^N|_W$ given by
\[
\Phi(v_1,\dots,v_N)=\pi_2(D(\pi_*v_1,\dots,\pi_*v_N)),
\]
where $\pi_2: U \times \R^{2nN} \to \R^{2nN}$ is the projection onto the second factor. It is clear that
\begin{equation}
\label{eq:index Phi}
\mu(\ga,\Phi)=\mu(\ga_H),
\end{equation}
where $\mu(\ga,\Phi)$ stands for the index of $\ga$ with respect to the trivialization $\Phi$ and $\ga_H=\pi\circ\ga$ is the corresponding orbit of $H$ with the index computed using the trivialization $D$. Let $\mathfrak f$ be the generator of $\pi_1(W)\simeq \Z$ given by the homotopy class of a simple orbit of $R_{\balpha}$ contained in $W$. Let $q\in \Z$ be such that $[\ga]=q\mathfrak f$, where $[\ga]$ is the homotopy class of $\ga$ in $W$. It turns out that $q$ is given by the Hamiltonian action of $\ga_H$. Indeed,
\begin{equation}
\label{eq:Ham action}
q=\int_\ga d\theta = \int_0^T \balpha(R_\alpha(\ga(t))) - \lambda(X_{H}(\ga_H(t)))\,dt = \int_0^T H(\ga_H(t)) - \lambda(X_{H}(\ga_H(t)))\,dt,
\end{equation}
where $T$ is the period of $\ga_H$. Therefore, the action $A_{H}(\ga_H)$ is an integer number. Note that, since $\ga$ is contractible and $a$ is a generator of $\pi_1(M)$, $q$ is a positive multiple of $p$.

Consider a trivialization $\Psi$ of $\ga^*\xi^N$ induced by our choice of a section of $\db$. The relation between the trivializations $\Phi$ and $\Psi$ are given by the following lemma.

\begin{lemma}
\label{lemma:triv}
We have that
\[
\mu(\ga,\Psi)=\mu(\ga,\Phi)+q(2n+4-\frac{2n+2}{p}).
\]
\end{lemma}

\begin{proof}
Choose a simple orbit $\phi(t)$ of $R_{\balpha}$ contained in $W$ and let $Q: [0,1] \times S^1 \to W$ be a homotopy between $\ga$ and $\phi^q$. We can extend the trivialization $\Psi$ to $Q^*\xi^N$ inducing a trivialization of $(\phi^q)^*\xi^N$. We have that
\[
\mu(\ga,\Psi)-\mu(\ga,\Phi)=\mu(\phi^q,\Psi)-\mu(\phi^q,\Phi).
\]
But an easy computation shows that
\[
\mu(\phi^q,\Psi)-\mu(\phi^q,\Phi)=\mu(\phi^q,\Psi)+n=q(2n+4-\frac{2n+2}{p}).
\]
Indeed, since $\phi$ is a generic orbit, clearly $\mu(\phi^q,\Phi)=-n$ for every $q$ (using \eqref{eq:index Phi} and the fact that the corresponding linearized Hamiltonian flow of the projected orbit on $U$ is constant equal to the identity). On the other hand, as mentioned before (see the discussion in the beginning of Section \ref{sec:proof dc1}) we have that $\mu(\phi,\Psi) = 2n+4-\frac{2n+2}{p}-n$. Since the Reeb flow of $\balpha$ generates a circle action and $\phi$ is a generic orbit, it implies that $\mu(\phi^q,\Psi) = q(2n+4-\frac{2n+2}{p})-n$ for every $q$.
\end{proof}

From now on, \emph{if the trivialization is not explicitly stated we use a trivialization given by a section of $\db$ and the trivialization $D$ for closed orbits of Hamiltonians on $U$}. Note that the index with respect to $D$ coincides with the index using the constant trivialization of $T\R^{2n}$. Let $U_0=\{x\in U; -G(x)<r_0\}$. Clearly, $U_0$ is invariant under the Hamiltonian flow of $H$ and if the image of $\ga_H$ is not contained in $U_0$ then $\mu(\ga_H)=-n$ which implies, by Lemma \ref{lemma:triv}, that $\mu(\ga) = -n + q(2n+4 )- q(2n+2)/p \geq n+2$ since $q\geq p\geq 3$.

Thus suppose that the image of $\ga_H$ lies in $U_0$. If $\ga_H(t)=0$ for some $t$ then $\ga_H=(\ga_H^0)^{pk}$ for some $k\in \N$, where $\ga_H^0: [0,1] \to U_0$ is the constant solution $\ga_H^0(t)\equiv 0$. A direct computation shows that in this case, due to our choice of $f$, the linearized Hamiltonian flows of $H$ and $G$ along $\ga_H$ coincide and therefore have the same index. But the Hamiltonian flow of $G$ is given by
\begin{equation}
\label{eq:flow G}
\vr^G_t(z_1,\dots,z_n)=(e^{-(2n+4-\Delta+\ep)\pi \sqrt{-1}t}z_1,e^{-2\ep\pi \sqrt{-1}t}z_2,\dots,e^{-2\ep\pi \sqrt{-1}t}z_n),
\end{equation}
where we are identifying $(q_i,p_i)$ with $z_i=q_i+\sqrt{-1}p_i$. Thus, it is clear that if $k=1$ then, choosing $\ep<1/p$, we have
\[
\mu(\ga_H) = -p(2n+4)+p\Delta-1-(n-1)=-p(2n+4)+p\Delta-n.
\]
This implies, by Lemma \ref{lemma:triv}, that $\mu(\ga)=p\Delta-(2n+2)-n \geq n+2$ (note that $q=pk=p$), where we used in the last inequality that $p\Delta-(2n+2)\geq 2n+2$ by \eqref{eq:Ct}. The case where $k>1$ follows from the fact that if a symplectic path $\Ga$ in $\Sp(2n)$ starting at the identity satisfies $\cz(\Ga) \geq n+2$ then $\cz(\Ga^k) \geq n+2$ for every $k$.

Now, let us consider the remaining case, where $\ga_H$ lies inside $U_0$ and $\ga_H(t)\neq 0$ for every $t$. Let $\ga_G(t)=\ga_H(t/f'(G(\ga_H(0))))$ be the corresponding orbit of $G$. The next three lemmas are taken from \cite{GM}. For the sake of completeness, we will provide their proofs.

\begin{lemma}
\label{lemmma:indexes HS}
We have that $|\mu(\ga_H)-\mu(\ga_G)|\leq 1$.
\end{lemma}

\begin{proof}
Firstly, notice that it is enough to find some symplectic trivialization $\Phi^t: T_{\ga_G(t)}\R^{2n} \to \R^{2n}$ such that
\[
|\mu(\ga_G,\Phi^t)-\mu(\ga_H,\Phi^{f'(e)t})|\leq 1,
\]
where $e=G(\ga_H(0))$ and the indexes above are the indexes of the symplectic paths defined using the corresponding trivializations. Let $T$ and $T_G=f'(G(\ga_H(0)))T$ be the periods of $\ga_H$ and $\ga_G$ respectively. We claim that there exists a symplectic plane $P$ such that $X_G(x)\in P$ and $d\vr^G_{T_G}(x)P=P$, where $x=\ga_H(0)$. Indeed, write $x=(z_1,\dots,z_n)$ and let $v_i=(0,\dots,0,z_i,0,\dots,0)$, $w_i=(0,\dots,0,\sqrt{-1}z_i,0,\dots,0)$ and $P_i=\tspan\{v_i,X_G(x)\}$. Note that if $z_i \neq 0$ then $P_i$ is a symplectic plane because $X_G(x)=-\pi\sqrt{-1}(((2n+4)-\Delta+\ep)z_1,2\ep z_2,\dots,2\ep z_n)$ and $\om(v_i,w_i)\neq 0$. Moreover, $d\vr^G_{T_G}(x)P_i=P_i$ for every $i$ because clearly $d\vr^G_{T_G}(x)X_G(x)=X_G(x)$ and $d\vr^G_{T_G}(x)v_i=v_i$ whenever $v_i\neq 0$ by \eqref{eq:flow G}. Thus, take $P=P_i$ for some $i$ such that $z_i\neq 0$.

Now, let $S=G^{-1}(e)$. Define
\[
D_1(t)=d\vr^G_{t}(x)P\quad\text{and}\quad D_2(t)=d\vr^G_{t}(x)P^\om=D_1(t)^\om,
\]
where $P^\om$ is the symplectic orthogonal to $P$. Since $X_G(\ga_G(t)) \in D_1(t)$ we have that $D_2(t) \subset T_{\ga_G(t)}S$ for every $t$. By construction, $D_1$ and $D_2$ are both invariant under the linearized Hamiltonian flow of $G$. But $\vr^H_t|_S=\vr^G_{f'(e)t}|_S$ which implies that $D_2$ is also invariant under the linearized Hamiltonian flow of $H$ and therefore the same holds for $D_1$ (since $D_1=D_2^\om$). Moreover, $X_H(\ga_G(t))=f'(e)X_G(\ga_G(t)) \in D_1(t)$ for every $t$.

Take symplectic trivializations $\Phi^t_1: D_1(t) \to \R^2$ and $\Phi^t_2: D_2(t) \to \R^{2n-2}$. We can choose $\Phi_1$ such that $\Phi^t_1(X_G(\ga_G(t)))=v$ for every $t$, where $v$ is a fixed vector in $\R^2$. Let $\Ga^G_1(t)=\Phi^t_1 \circ d\vr^G_{t}(x) \circ (\Phi^0_1)^{-1}$, $\Ga^G_2(t)=\Phi^t_2 \circ d\vr^G_{t}(x) \circ (\Phi^0_2)^{-1}$ and $\Ga^G(t)=\Phi^t \circ d\vr^G_{t}(x) \circ (\Phi^0)^{-1}$ be the corresponding symplectic paths, where $\Phi=\Phi_1\oplus\Phi_2$. Similarly, define $\Ga^H_1(t)=\Phi^{f'(e)t}_1 \circ d\vr^H_{t}(x) \circ (\Phi^0_1)^{-1}$, $\Ga^H_2(t)=\Phi^{f'(e)t}_2 \circ d\vr^H_{t}(x) \circ (\Phi^0_2)^{-1}$ and $\Ga^H(t)=\Phi^{f'(e)t} \circ d\vr^H_{t}(x) \circ (\Phi^0)^{-1}$. Since the Hamiltonian flows of $G$ and $H$ restricted to $S$ differ by a constant reparametrization,
\[
\mu(\Ga^G_2)=\mu(\Ga^H_2).
\]
By our choice of $\Phi_1$, we have that the spectra of $\Ga^G_1(t)$ and $\Ga^H_1(t)$ are equal to $\{1\}$ for every $t$. This implies that
\[
\mu(\Ga^G_1) \in \{0,-1\}\quad\text{and}\quad\mu(\Ga^H_1) \in \{0,-1\}
\]
and consequently $|\mu(\Ga^G_1)-\mu(\Ga^H_1)|\leq 1$. Thus,
\[
|\mu(\Ga^G)-\mu(\Ga^H)|\leq 1
\]
as desired.
\end{proof}

Now, let $k: \R \to \R$ be a smooth function such that $k'(r)>0$ for every $r>0$. Define $\delta: \R \to \Z$ as
\[
\delta(x)=
\begin{cases}
2x +1\quad\text{if}\ x \in \Z ,\\
2\lceil x\rceil -1\quad\text{otherwise,}
\end{cases}
\]
where $\lceil x\rceil = \min\{k \in \Z; k \geq x\}$. In what follows, recall that $F: \R^2 \to \R$ is the Hamiltonian given by $F(q,p)=\pi(q^2+p^2)$.

\begin{lemma}
\label{lemma:index K}
Let $K=-k\circ F$ with $k$ as above. Given a periodic orbit $\ga_K$ of $K$ with period $T_K$ then
\[
\mu(\ga_K) \geq -\delta(k'(F(\ga_K(0)))T_K).
\]
\end{lemma}

\begin{proof}
The flow of $K$ is given by $\vr^K_t(z)=e^{-2\pi k'(F(z))\sqrt{-1}t}z$, where, as before, we are identifying $(q,p)$ with $z=q+\sqrt{-1}p$. An easy computation shows that if $\ga_K(t)\equiv 0$ then the linearized Hamiltonian flow on $\ga_K$ is given by $e^{-2\pi k'(0)\sqrt{-1}t}z$ and consequently $\mu(\ga_K) = -\delta(k'(0)T_K)$.

If $\ga_K$ is away from the origin, we proceed similarly as in the proof of the previous lemma. Let $\ga_{-F}(t)=\ga_K(t/k'(F(\ga_K(0))))$ be the corresponding periodic orbit of $-F$ with period $T_{-F}=k'(F(\ga_K(0)))T_K$. Take a symplectic trivialization $\Phi^t: T_{\ga_{-F}(t)}\R^2 \to \R^2$  such that $\Phi^t(X_{-F}(\ga_{-F}(t)))=v$ and  $\Phi^t(-\nabla F(\ga_{-F}(t))/\|\nabla F(\ga_{-F}(t))\|)=w$ (here the gradient and the norm are taken with respect to the Euclidean metric) where $\{v,w\}$ is a fixed symplectic basis in $\R^2$. Then clearly the linearized Hamiltonian flow of $-F$ with respect to this trivialization is constant equal to the identity and therefore
\[
\mu(\ga_{-F},\Phi)=-1.
\]
Since $X_K(\ga_K(t))=k'(F(\ga_K(0)))X_{-F}(\ga_K(t))$ is preserved under the linearized Hamiltonian flow of $K$, we see that  the spectrum of the symplectic path
\[
\Ga^K(t)=\Phi^{k'(F(\ga_K(0)))t} \circ d\vr^K_{t}(x) \circ (\Phi^0)^{-1}
\]
is constant and equal to $\{1\}$. Thus,
\[
\mu(\ga_K,\Phi)=\mu(\Ga^K) \in \{-1,0\}.
\]
In particular, we conclude that $\mu(\ga_K,\Phi) \geq \mu(\ga_{-F},\Phi)$. But this implies that
\[
\mu(\ga_K) \geq \mu(\ga_{-F}),
\]
where the indexes above are computed using the canonical trivialization of $\R^2$. Finally, an easy computation shows that
\[
\mu(\ga_{-F})=-\delta(k'(F(\ga_K(0)))T_K).
\]
\end{proof}

Therefore, it follows from \eqref{eq:flow G} and Lemma \ref{lemma:index K} that
\[
\mu(\ga_G) \geq -\delta(\frac{2n+4-\Delta+\ep}{2}T_G) - (n-1)\delta(\ep T_G).
\]
Let $q$ be the integer number such that $[\ga]=q\mathfrak f$ given by \eqref{eq:Ham action}. It is clear from the previous inequality and Lemmas \ref{lemma:triv} and \ref{lemmma:indexes HS} that
\begin{equation}
\label{eq:index bga}
\mu(\ga) \geq q(2n+4-\frac{2n+2}{p}) - \delta(\frac{2n+4-\Delta+\ep}{2}T_G) - (n-1)\delta(\ep T_G) - 1.
\end{equation}
(Recall that $\mu(\ga)$ is the index of $\ga$ with respect to a trivialization given by a section of $\db$.) Thus, to prove that $\mu(\ga)\geq n+2$ it is enough to show that
\begin{equation}
\label{eq:dynconvex}
q^{-1}(n+2 +\delta(\frac{2n+4-\Delta+\ep}{2}T_G) + (n-1)\delta(\ep T_G)+1) \leq 2n+4-\frac{2n+2}{p}.
\end{equation}
In order to prove this inequality, we need the following result.

\begin{lemma}
\label{lemma:action}
If $\ga_H$ is a non-constant periodic orbit of $H$ then $A_{H}(\ga_H)>T_G$, where $T_G$ is the period of the corresponding orbit of $G$.
\end{lemma}

\begin{proof}
Let $T$ be the period of $\ga_H$. We have that
\begin{align*}
A_{H}(\ga_H) & = \int_0^T f(G(\ga_H(t))) - f'(G(\ga_H(t)))\lambda(X_{G}(\ga_H(t)))\,dt \\
& = Tf(G(\ga_H(0))) - f'(G(\ga_H(0)))\int_0^T \lambda(X_{G}(\ga_H(t)))\,dt.
\end{align*}
An easy computation shows that
\[
\lambda(X_{G}(q_1,p_1,\dots,q_n,p_n)) = G(q_1,p_1,\dots,q_n,p_n).
\]
Thus,
\[
A_H(\ga_H) = T(f(G)-f'(G)G),
\]
where, to simplify the notation, we have omitted the dependence of $G$ on $\ga_H(0)$.

Let $T_G=Tf'(G)$ be the period of the corresponding orbit of $G$. Since $f'(G)> 0$, we arrive at
\begin{align}
\label{eq:action H}
A_H(\ga_H) & = T(f(G)-f'(G)G) \nonumber\\
& = T_G\frac{f(G)-f'(G)G}{f'(G)} \nonumber\\
& = s(G)T_G,
\end{align}
where $s: (-r_0,0] \to \R$ is given by
\[
s(r)=f(r)/f'(r)-r.
\]
We claim that $s(G)>1$. As a matter of fact, since $\ga_H$ is non-constant, $G(\ga_H(0)) \in (-r_0,0)$. But, by our choice of $f$, $s(0)=1$ and
\[
s'(r) = \frac{-f''(r)f(r)}{f'(r)^2} < 0
\]
for every $r \in (-r_0,0)$. Consequently,
\begin{equation}
\label{eq:bound action H}
A_H(\ga_H) > T_G,
\end{equation}
as desired.
\end{proof}

Now notice that
\[
\delta(cT_G) \leq 2c\lceil T_G\rceil +1
\]
for every positive real number $c$. Hence,
\begin{align*}
& q^{-1}(n+2 +\delta(\frac{2n+4-\Delta+\ep}{2}T_G) + (n-1)\delta(\ep T_G)+1) \\
& \leq \frac{2n+3}{q} + \frac{2\lceil T_G\rceil}{q}(\frac{2n+4-\Delta+\ep}{2}+(n-1)\ep).
\end{align*}

In order to estimate the last expression, note that, by Lemma \ref{lemma:action} and the fact that $q=A_H(\ga_H)$ is an integer,
\[
\frac{\lceil T_G\rceil}{q} \leq 1.
\]
Hence,
\begin{align*}
\frac{2n+3}{q} + \frac{2\lceil T_G\rceil}{q}(\frac{2n+4-\Delta+\ep}{2}+(n-1)\ep) & \leq \frac{2n+3}{q} + 2n+4-\Delta+\ep+(2n-2)\ep \nonumber \\
& = 2n+4 - (\Delta - \frac{2n+3}{q}) + (2n-1)\ep \nonumber \\
& \leq 2n+4 - \frac{2n+2}{p},
\end{align*}
choosing $\ep$ such that $(2n-1)\ep<2/p$, where the last inequality follows from the facts that $\Delta - \frac{2n+3}{q} > \Delta - \frac{2n+3}{p}$ because $q\geq p$ and $\Delta - \frac{2n+2}{p} \geq \frac{2n+4}{p}$ by \eqref{eq:Ct}. This proves \eqref{eq:dynconvex} finishing the proof the lemma.

\subsubsection{Proof of Lemma \ref{lemma:not a-dynconvex}}
\label{sec:not a-dynconvex}

As in the proof of Lemma \ref{lemma:dynconvex1}, let $\ga_H^0: [0,1] \to U_0$ be the constant solution $\ga_H^0(t)\equiv 0$ of period one so that $\ga_0$ is the corresponding closed Reeb orbit of $\alpha$. As explained in the previous section, it is clear from \eqref{eq:flow G} that
\[
\cz(\ga_H^0) = -(2n+4) + \Delta - n.
\]
Hence, we conclude from Lemma \ref{lemma:triv} that
\begin{align*}
\cz(\ga) & = -(2n+4) + \Delta - n + 2n+4 - \frac{2n+2}{p} \nonumber \\
& = \Delta - \frac{2n+2}{p} - n \nonumber \\
& < n+2 -  \frac{2n+2}{p},
\end{align*}
where the last inequality follows from \eqref{eq:Ct}. But note that
\[
k_a = \frac{2n+2}{p}(p-1) - n =  n+2 -  \frac{2n+2}{p}.
\]

\subsection{Proof of Assertion 2}
\label{sec:proof dc2}

As in Section \ref{sec:proof dc1}, let $M=L^{5}_{11}(1,1,1)$, with $p$ odd, be the lens space endowed with the induced contact structure $\xi$ and $N=11$ be the minimal number such that $Nc_1(\xi)=0$. Let $a \in \pi_1(M)$ be such that $j_a=5$ so that $k_a=\frac{2n+2}{11}j_a-n=8/11$. Let $m=n-1$ and $\balpha$ be the contact form given by Proposition \ref{prop:circle action} corresponding to $a$ and $m$. Although $j_a$, $n$, $N$ and $p$ are fixed in what follows, in view of Remark \ref{rmk:Chern class}, we will write them as variables unless explicitly stated.

\begin{remark}
\label{rmk:CC1}
For Remark \ref{rmk:Chern class}, let $M=L^{2n+1}_p(1,\dots,1)$, with $p$ odd, be the lens space endowed with the induced contact structure $\xi$ and $a \in \pi_1(M)$ be such that $j_a=\lfloor p/2 \rfloor$ so that $h_a=\frac{2n+2}{p}\lfloor p/2 \rfloor$ (see Example \ref{ex:h_a}). Suppose that $p\geq 5$, $j_a$ and $n-1$ are coprime, $a$ is a generator of $\pi_1(M)$ and $n\geq p+2$. Let $m=n-1$ and $\balpha$ be the contact form given by Proposition \ref{prop:circle action} correponding to $a$ and $m$. Under these hypotheses, we have several examples with vanishing first Chern class. For instance, $L^{29}_5(1,\dots,1)$.
\end{remark}

Let $F_h: \R^2 \to \R$ be the Hamiltonian whose graph and phase portrait are depicted in Figure \ref{fig:Hamiltonian F_h}. The Hamiltonian vector field has precisely one hyperbolic singularity $p_0$ at the origin and two elliptic ones $p_\pm$. The regular orbits are two homoclinic connections and periodic orbits. We ask that
\begin{itemize}
\item $F_h$ attains its global mininum at $p_\pm$ and $F_h(p_\pm)=0$;
\item $F_h$ is strictly convex near the elliptic singularities.
\end{itemize}
Note that $F_h$ can be chosen arbitrarily $C^\infty$-small.

\begin{figure}[ht]
\includegraphics[width=4in, height=2.5in]{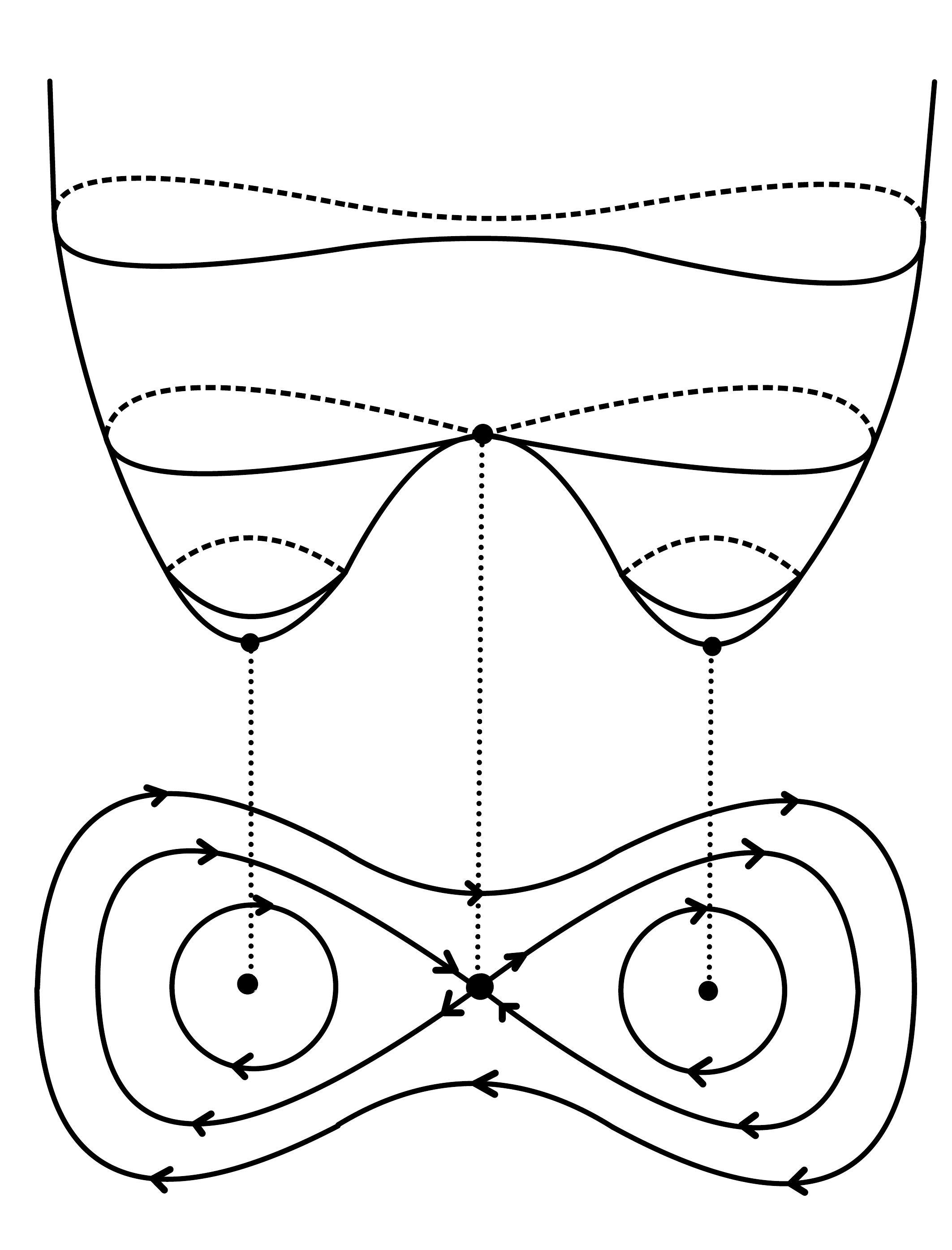}
\centering
\caption{The graph and phase portrait of the Hamiltonian $F_h$.}
\label{fig:Hamiltonian F_h}
\end{figure}

As in the previous section, let $F: \R^2 \to \R$ be defined as $F(q,p)=\pi(q^2+p^2)$. Consider the Hamiltonians $G: \R^{2n} \to \R$ and $G_h: \R^{2n} \to \R$ given by
\[
G(q_1,p_1,\dots,q_n,p_n) = f\bigg(\sum_{i=1}^n F(q_i,p_i)\bigg)\quad\text{and}\quad G_h(q_1,p_1,\dots,q_n,p_n) =  f_h\bigg(\sum_{i=1}^n F_h(q_i,p_i)\bigg)
\]
where $f: \R \to \R$ and $f_h: \R \to \R$ are smooth functions such that
\begin{itemize}
\item $f_h(x)=x$ for every $x\leq \delta/2$ for some $\delta>0$ small;
\item $f_h'(x)>0$ for every $x<\delta$;
\item $f_h(x)=C_h$ for every $x\geq \delta$ and some constant $C_h$;
\item $f(r)=-r$ for every $r\leq r_0$, where $r_0$ is such that $G_h^{-1}([0,C_h))$ is contained in the interior of the ball $B_0:=\{(q_1,p_1,\dots,q_n,p_n) \in \R^{2n};\,\pi\sum_i q_i^2+p_i^2 \leq r_0\}$;
\item $f(r)=C$ for every $r\geq r_1$ and  some constant $C$, where $r_1>r_0$;
\item $-1<f'(r)<0$ for every $r \in (r_0,r_1)$.
\end{itemize}
We also ask that $G_h^{-1}((0,C_h))$ contains the origin.

Define the time-dependent Hamiltonian
\[
H_t(x)=(G \# G_h)_t(x)=G(x)+G_h(\vr_{G}^{-t}(x)),
\]
where $\vr_{G}^{t}$ is the Hamiltonian flow of $G$. The flow of $H$ is given by $\vr_{H}^{t}=\vr_{G}^{t} \circ \vr_{G_h}^{t}$. Note that on $B_0$, the Hamiltonian flow of $G$ generates a loop of period one. This, together with the fact that $G_h$ is constant outside the ball $B_0$, implies that $H$ is 1-periodic in time.

As in Section \ref{sec:proof dc1}, consider the lens space $M$ as the prequantization circle bundle of an orbifold $B$ given by the quotient of the sphere by the circle action generated by the Reeb flow of $\balpha$ and let $\pi: M \to B$ be the quotient projection. Take a point $x_0$ in the smooth part of $B$ and consider the corresponding periodic orbit $\bga_0$ of $\balpha$ over $x_0$.  Consider a neighborhood $W \simeq U \times S^1$ of $\bga_0$ and coordinates $(q_1,p_1,\dots,q_n,p_n,t)$ on $V$, where $t$ is the coordinate along the fiber such that, with respect to these coordinates, $x_0$ is the origin and
\[
\balpha|_W=\lambda+dt,
\]
where $\lambda=\frac 12\sum_{i=1}^n (q_idp_i-p_idq_i)$ is the Liouville form. Take $r_1$ above small enough such that the ball $B_1:=\{(q_1,p_1,\dots,q_n,p_n) \in \R^{2n};\,\pi\sum_i q_i^2+p_i^2 \leq r_1\}$ is contained in $U$. Using these coordinates, define the contact form $\alpha$ on $M$ such that
\[
\alpha|_W = \lambda + \frac{H_t}{C+C_h} dt,
\]
and extending $\alpha$ to the complement of $W$ as $\balpha$. By the construction of $H_t$, one can easily check that $\alpha$ is smooth. 

\begin{remark}
\label{rmk:small2}
Note that $G_h$ can be chosen $C^\infty$-small. On the other hand, $G$ can be chosen $C^1$-small but not $C^2$-small. Hence, $\alpha$ can be chosen $C^1$-close to the convex contact form $\balpha$ but it is not $C^2$-close.
\end{remark}

The next two lemmas will be proved in Section \ref{sec:technical}. 

\begin{lemma}
\label{lemma:dynconvex2}
The contact form $\alpha$ is dynamically convex. 
\end{lemma}

Let $\ga_0$ be the simple closed orbit of $R_\alpha$ over $x_0$. Recall that, by construction, the homotopy class of $\ga_0$ is $a$.

\begin{lemma}
\label{lemma:hyperbolic}
The periodic orbit $\ga_0$ is hyperbolic and satisfies $\cz(\ga_0)=k_a<h_a$.
\end{lemma}

\begin{remark}
\label{rmk:CC2}
Under the context of Remark \ref{rmk:CC1}, we have that the corresponding periodic orbit $\ga_0$ is hyperbolic and satisfies $\cz(\ga_0)<h_a$.
 \end{remark}

\subsection{Proof of Lemmas \ref{lemma:dynconvex2} and \ref{lemma:hyperbolic}}
\label{sec:technical}

\subsubsection{Proof of Lemma \ref{lemma:dynconvex2}}

We need first the following lemmata.

\begin{lemma}
\label{lemma:index F_h}
All the periodic orbits of $F_h$ have non-negative index. Moreover, the hyperbolic singularity has index zero.
\end{lemma}

\begin{proof}
Let $\ga$ be a periodic orbit of $F_h$ of period $T$ and let $\Phi: [0,T] \to \Sp(2)$ be the corresponding symplectic path obtained via the canonical (constant) trivialization of $T\R^2$. Let $z \in \C\setminus\{0\}$ and $\rho(t)$ be a continuous argument of $z(t)=\Phi(t)z$, that is, $\rho: [0,T] \to \R$ is a continuous function such that $e^{2\pi \sqrt{-1}\rho(t)}=z(t)/|z(t)|$. Let $\Delta(z):=\rho(T)-\rho(0)$ and $I(\Phi)=\{\Delta(z);\,z\in\C\setminus\{0\}\}$. The set $I(\Phi)$ is an interval of length less than $1/2$ and therefore it is between two integers or contains an integer. Define
\[
\tmu(\Phi)=
\begin{cases}
2k+1\quad& \text{if}\ I(\Phi)\subset (k-1,k),\\
2k\quad & \text{if}\ k \in I(\Phi).
\end{cases}
\]
It is well known that $\cz(\ga)$ equals $\tmu(\Phi)$ if $\Phi$ is non-degenerate and equals $\tmu(\Phi)$ or $\tmu(\Phi)-1$ in general. It is easy to see that if $\ga$ is hyperbolic then $\cz(\ga)=\tmu(\ga)=0$. If $\ga$ is one of the elliptic singularities then we have that $\cz(\ga)\geq 2$ by the convexity of $H$ near these singularities.

Now, suppose that $\ga$ is regular. The Hamiltonian vector field $X_{F_h}$ satisfies $\Phi(t)X_{F_h}(\ga(0))=X_{F_h}(\ga(t))$ for every $t$ and therefore one can easily see that $\tmu(\Phi)\geq 1$. Hence, $\cz(\ga)\geq 0$.
\end{proof}

\begin{lemma}
\label{lemma:index G_h}
All the periodic orbits of $G_h$ have index bounded from below by $-n$.
\end{lemma}

\begin{proof}
Consider the Hamiltonian $\bG_h(q_1,p_1,\dots,q_n,p_n) = \sum_{i=1}^n F_h(q_i,p_i)$. It follows from Lemma \ref{lemma:index F_h} that every periodic orbit of $\bG_h$ is non-negative. In particular, every such periodic orbits has non-negative mean index $\mi$. From this we can conclude that every periodic orbit of $G_h$ in $G_h^{-1}([0,C_h))$ has index bounded from below by $-n$. Indeed, let $\ga$ be a closed orbit of $G_h$ in $G_h^{-1}([0,C_h))$ and denote by $\bga$ the corresponding periodic orbit of $\bG_h$. It is well known that $\mi(\bga)=\mi(\ga)$; see \cite[Lemma 2.6]{GG09}. Since $\mi(\bga)\geq 0$ and $|\mi(\ga)-\cz(\ga)|\leq n$ we conclude that $\cz(\ga)\geq -n$.

Finally, if $\ga$ is a closed orbit lying outside $G_h^{-1}([0,C_h))$ then, by construction of $f_h$, we have that the linearized Hamiltonian flow along $\ga$ is constant equal to the identity and therefore $\cz(\ga)=-n$.
\end{proof}

\begin{lemma}
\label{lemma:index G}
Let $\ga$ be a closed orbit of $G$ with integral period $T$ and whose image is contained in $B_1$. Then $\cz(\ga) \geq -n(2T+1)$.
\end{lemma}

\begin{proof}
Consider the Hamiltonian $\bG(q_1,p_1,\dots,q_n,p_n) = -\sum_{i=1}^n F(q_i,p_i)$ and let $\bga$ be the corresponding closed orbit of $\bG$. The period of of $\bga$ is $\bT=Tf'(\bG(\ga))$. It is easy to see that the index of $\bga$ is bounded from below by $-n(2\bT+1)$. By our choice of $f$, $\bT \leq T \implies -n(2\bT+1)\geq -n(2T+1)$ with the equality if and only $\ga$ lies in $B_0$. On $B_0$, clearly $\cz(\bga)=\cz(\ga)$ and therefore the result holds if $\ga$ lies in $B_0$. If $\ga$ lies in $B_1\setminus B_0$, we have the strict inequality $-n(2\bT+1)>-n(2T+1)$ and one can easily check that $|\cz(\bga)-\cz(\ga)|\leq 1$. Therefore, $\cz(\ga) \geq -n(2T+1)$, where we are using the fact that $T$ is an integer.
\end{proof}

\vskip .2cm
\begin{proof}[Proof of Lemma \ref{lemma:dynconvex2}]
Let $W_i=\pi^{-1}(B_i)$ ($i=0,1$), where $\pi: W \simeq U \times S^1 \to U$ is the projection. Since $H_t$ is constant outside $B_1$ and $\balpha$ is toric, we have only to check that every contractible periodic orbit $\ga$ of $\alpha$ contained in $W_1$ satisfies $\cz(\ga)\geq n+2$. Note that the projection $\ga_H$ of $\ga$ to $B_1$ corresponds to a periodic orbit of $H_t$ with period $T$ given by a positive multiple of $p$ (since $a$ is a generator of $\pi_1(M)$).

As in Section \ref{sec:dynconvex}, the Darboux coordinates induce an obvious (constant) trivialization $D: TU^N \to U \times \R^{2nN}$. From this we get a trivialization of $\xi^N|_W$ given by
\[
\Phi(v_1,\dots,v_N)=\pi_2(D(\pi_*v_1,\dots,\pi_*v_N)),
\]
where $\pi_2: U \times \R^{2nN} \to \R^{2nN}$ is the projection onto the second factor. It is clear that
\begin{equation}
\mu(\ga,\Phi)=\mu(\ga_H),
\end{equation}
where $\mu(\ga,\Phi)$ stands for the index of $\ga$ with respect to the trivialization $\Phi$ and the index of $\ga_H$ is computed using the trivialization $D$.

Consider a trivialization $\Psi$ of $\ga^*\xi^N$ induced by our choice of a section of $\db$. The relation between the trivializations $\Phi$ and $\Psi$ is given by the following lemma whose proof is analogous to the one of Lemma \ref{lemma:triv}.

\begin{lemma}
\label{lemma:triv2}
We have that
\[
\mu(\ga,\Psi)=\mu(\ga,\Phi)+T(\frac{2n+2}{p}j_a+2(n-1)).
\]
\end{lemma}

From now on, \emph{if the trivialization is not explicitly stated we use a trivialization given by a section of $\db$ and the trivialization $D$ for closed orbits of Hamiltonians on $B_1$}. Note that the index with respect to $D$ coincides with the index using the constant trivialization of $T\R^{2n}$. Clearly, $B_0$ and $B_1$ are both invariant under the Hamiltonian flow of $H$ and if the image of $\ga_H$ is not contained in $B_1$ then $\mu(\ga_H)=-n$. Write $T=kp$ with $k$ being a positive integer (recall that $T$ is a positive multiple of $p$). By Lemma \ref{lemma:triv2}, we have that $\mu(\ga) = (2n+2)kj_a+2kp(n-1)-n \geq n+2$.

Now suppose that the image of $\ga$ lies in $W_0$. Since on $B_0$ the Hamiltonian flow of $G$ generates a loop of period one and Conley-Zehnder index $-3n$ we have, by Lemmas \ref{lemma:index G_h} and \ref{lemma:triv2},
\begin{align}
\label{eq:index ga_h}
\cz(\ga) & \geq T((2n+2)j_a/p+2(n-1)-2n)-n \nonumber \\
& = k(2n+2)j_a-2kp-n \nonumber \\
& = 2k((n+1)j_a-p)-n \nonumber \\
& \geq 2(3\cdot 5-11)-2=6 > n+2=4
\end{align}
where in the second inequality and the last equality we used that $j_a=5$, $k\geq 1$, $n=2$ and $p=11$. (Here we are using the fact that the index of a symplectic path $\Ga$ in $\Sp(2n)$ composed with a loop with Conley-Zehnder index $-3n$ is equal to the index of $\Ga$ minus $2n$.)

\begin{remark}
\label{rmk:CC3}
Regarding Remark \ref{rmk:CC1}, we have the relation
\begin{align}
\label{eq:index ga_h2}
\cz(\ga) & \geq T(\frac{2n+2}{p}j_a+2(n-1)-2n)-n \nonumber \\
& = 2k((n+1)j_a-p)-n \nonumber \\
& \geq 2((n+1)\lfloor p/2 \rfloor-p)-n \nonumber \\
& = 2((n+1)(p/2-1/2)-p)-n \nonumber \\
& = n(p-2)-p-1 \geq 2n-p \geq n+2
\end{align}
where in the second equality we used the fact that $p$ is odd, the third inequality follows from the assumption that $p\geq 5$ (and the fact that $n\geq 1$) and the last one holds because $n\geq p+2 \implies 2n-p\geq n+p+2-p=n+2$.
\end{remark}

Finally, suppose that the image of $\ga$ lies in $W_1\setminus W_0$. Since $G_h$ restricted to $B_1\setminus B_0$ is constant, $\ga_H$ must be a closed orbit of $G$ with integral period. By Lemmas \ref{lemma:index G} and \ref{lemma:triv2} we have, as in \eqref{eq:index ga_h},
\[
\cz(\ga) \geq T((2n+2)j_a/p+2(n-1)-2n)-n > n+2
\]
using the fact that $j_a=5$, $k\geq 1$, $n=2$ and $p=11$.

\begin{remark}
\label{rmk:CC4}
Concerning Remark \ref{rmk:CC1}, inequality \eqref{eq:index ga_h2} also holds under the assumptions that $p\geq 5$ and $n\geq p+2$ as explained in Remark \ref{rmk:CC3}.
\end{remark}
\end{proof}

\subsubsection{Proof of Lemma \ref{lemma:hyperbolic}}

Since the hyperbolic periodic orbit of $G_h$ has index zero and on $B_0$ the flow of $G$ generates a loop of period one and Conley-Zehnder index $-3n$, we have from Lemma \ref{lemma:triv2} that
\begin{equation}
\label{eq:hyp}
\cz(\ga) = \frac{2n+2}{p}j_a+2(n-1)-2n = \frac{2n+2}{p}j_a - 2 = k_a
\end{equation}
where we are using the fact that $n=2$. (Here, as before, we are using the fact that the index of a symplectic path $\Ga$ in $\Sp(2n)$ composed with a loop with Conley-Zehnder index $-3n$ is equal to the index of $\Ga$ minus $2n$.)

\begin{remark}
\label{rmk:CC5}
Under the assumptions of Remark \ref{rmk:CC1}, note that by equality \eqref{eq:hyp} we have that
\[
\cz(\ga) = \frac{2n+2}{p}j_a+2(n-1)-2n = \frac{2n+2}{p}j_a - 2 = h_a-2
\]
where we are using the fact that $j_a<p/2$ (because $p$ is odd).
\end{remark}

\section{Proof of Theorem \ref{thm:sharp}}
\label{sec:sharp}

\subsection{Proof of Assertion 1}

First, note that if $\alpha$ is non-degenerate, then, as mentioned before, $\HC^a_*(\L)$ can be obtained as the homology of a chain complex generated by the good periodic orbits of $\alpha$ with homotopy class $a$. Therefore, if every closed orbit $\ga$ of $\alpha$ with homotopy class $a$ satisfies $\cz(\ga)>k_a$ we would conclude that $\HC^a_{k_a}(\L)=0$, a contradiction. Hence, we must have a periodic orbit $\ga$  with homotopy class $a$ such that $\cz(\ga)\leq k_a$ which implies, by the first assertion of Theorem \ref{thm:main}, that $\cz(\ga)=k_a$.

If $\alpha$ is degenerate, we proceed as follows. Suppose that every closed orbit of $\alpha$ with homotopy class $a$ satisfies $\cz(\ga)>k_a$. Since $\HC^a_{k_a}(\L) \neq 0$ we have that $\HC^{a,T}_{k_a}(\alpha) \neq 0$ for some $T>0$ sufficiently large. Let $\balpha$ be a non-degenerate perturbation of $\alpha$ such that every periodic orbit of $\balpha$ with action less than $T$ is close to some periodic orbit of $\alpha$. Choosing $T$ away from the action spectrum of $\alpha$, we have that $\HC^{a,T}_{k_a}(\balpha) \cong \HC^{a,T}_{k_a}(\alpha) \neq 0$. Since the index is lower semicontinuous, we have that every periodic orbit $\bga$ of $\balpha$ with action less than $T$ satisfies $\cz(\bga)>k_a$ which implies that $\HC^{a,T}_{k_a}(\balpha)=0$, a contradiction. This implies that $\alpha$ has a periodic orbit $\ga$  with homotopy class $a$ such that $\cz(\ga)\leq k_a$ which implies, by the convexity of $\alpha$ as before, that $\cz(\ga)=k_a$.

\subsection{Proof of Assertion 2}

The proof of this assertion follows the argument of the proof the second assertion of Theorem \ref{thm:dc} but it is actually much simpler. We consider the contact form $\balpha$ on $M=L^{2n+1}_p(1,\dots,1)$ whose Reeb flow generates a free circle action ($\balpha$ is induced from a constant multiple of the Liouville form restricted to the unit sphere in $\R^{2n+2}$). In this way, we consider the lens space $M$ as the prequantization circle bundle of the complex projective space $\CP^n$ given by the quotient of the sphere by the circle action generated by the Reeb flow of $\balpha$. Let $\pi: M \to \CP^n$ be the quotient projection.

Take a point $x_0$ in $\CP^n$ and a neighborhood $U$ of $x_0$ with Darboux coordinates $(q_1,p_1,\dots,q_n,p_n)$ identifying $x_0$ with the origin. Consider the Hamiltonian $G_h$ defined in Section \ref{sec:proof dc2} as an Hamiltonian on $U$ (taking $\delta$ sufficiently small) and extend it to $\CP^n$ setting $G_h|_{\CP^n\setminus U}\equiv C_h$. Define the contact form
\[
\alpha=\balpha/(1+\hat G_h),
\]
where $\hat G_h=G_h\circ\pi$. The Reeb vector field of $\alpha$ is given by
\[
R_\alpha=(1+\hat G_h)R_{\balpha}+ \hat X_{G_h},
\]
where $R_{\balpha}$ is the Reeb vector field of $\balpha$ (assume that the simple orbits of $\balpha$ have minimal period one) and $\hat X_{G_h}$ is the horizontal lift of the Hamiltonian vector field of $G_h$. By Remark \ref{rmk:small2}, $\alpha$ can be chosen $C^\infty$-close to $\balpha$. Therefore, it is strictly convex.

Let $\ga_0$ be the periodic orbit of $\alpha$ over $x_0$. By construction, it is hyperbolic. Let us compute its index. By Lemma \ref{lemma:index F_h} and the construction of $G_h$, we have from \eqref{eq:index Phi} and an inspection of the proof of Lemma \ref{lemma:triv} that
\[
\cz(\ga_0) = \frac{2n+2}{p}
\]
since the right hand side is the Conley-Zehnder index of the (simple) orbits of $\balpha$ plus $n$ and the hyperbolic orbit of $G_h$ has index zero. Now, note that the homotopy class $a$ of $\ga_0$ satisfies $j_a=1$ and therefore $h_a=\th_a=(2n+2)/p$. Consequently, $\ga_0$ is hyperbolic and satisfies $\cz(\ga_0)=h_a=\th_a$, as desired.

\subsection{Proof of Assertion 3}

Consider the contact form $\alpha$ constructed in the previous section with $n=1$ and $p=4$. It is a strictly convex contact form carrying a hyperbolic closed orbit $\ga_0$ with homotopy class $a$ satisfying $\cz(\ga_0)=h_a=\th_a$ (since $j_a=1<p/2=2$). But we have that $k_a=\th_a-n=\th_a-1 \implies \cz(\ga_0)=k_a+1$.

\section{Proof of Theorem \ref{thm:dc open}}
\label{sec:dc open}

Let $\alpha$ be one of the contact forms given by Theorem \ref{thm:dc} and $\beta$ its lift to $S^{2n+1}$. Given a contactomorphism $\bar\vr: S^{2n+1} \hookleftarrow$ that commutes with $\psi$, we have to show that there exists a $C^1$-neighborhood $U$ of $\bar\vr$ such that $\vr^*\beta$ cannot be convex for any $\vr \in U$. Let $\bar\phi: \Lone \hookleftarrow$ be the contactomorphism induced by $\bar\vr$. As discussed in the introduction, the action $\bar\phi_*$ on $\pi_1(\Lone)$ induced by $\bar\phi$ is trivial (because $k_a\neq k_b$ for distinct homotopy classes $a$ and $b$). We have that $\vr^*\beta$ is invariant under the conjugated $\Z_p$-action generated by $\vr^{-1}\psi\vr$.

Let $\Lpone$ be the quotient of $S^{2n+1}$ by the action generated by $\vr^{-1}\psi\vr$. Denote by $\phi:\Lpone \to \Lone$ the contactomorphism induced by $\vr$. Since $\phi=\bar\phi(\bar\phi^{-1}\phi)$ and $\bar\phi_*$ is the identity, we have that $\phi_*=(\bar\phi^{-1}\phi)_*$. Therefore, it is enough to show that if $\phi^*\alpha$ is convex (the definition of convex contact forms on $\Lpone$ is analogous to the one for contact forms on $\Lone$) and $\ga$ is a periodic orbit of $\phi^*\alpha$ with homotopy class $(\bar\phi^{-1}\phi)^{-1}_*a$ then
\begin{enumerate}
\item $\cz(\gamma) \geq k_a$;
\item if $\cz(\gamma) < h_a$ (resp. $\cz(\gamma) < \th_a$) then $\ga$ is non-hyperbolic;
\item if $\ell^a_i>0$ and $\ell^a_i\neq p/2$ (resp.  $\ell^a_i>0$) for every $i$ and $\cz(\gamma) = k_a$ then $\ga$ is elliptic.
\end{enumerate}

The proof of Theorem \ref{thm:main} carries out word by word to prove this except that, instead of $G_a$, we have to take into account the Hamiltonian $G^{\vr}_a:=G_a\circ(\bar\vr^{-1}\vr,\dots,\bar\vr^{-1}\vr)$ whose flow might not be linear; see Remarks \ref{rmk:linearity1} and \ref{rmk:linearity2}. Taking $U$ sufficiently small, we have that $\bar\vr^{-1}\circ\vr$ is $C^1$-close to identity and therefore the symplectic path $D\vr^{G^\vr_a}_t(\hga(0))^{-1}$ is $C^0$-close to the path $\vr^{G_a}_{-t}=D\vr^{G_a}_{-t}$.

By the lower semicontinuity of the Bott's function with respect to the symplectic path (in the $C^0$-topology) we infer that $\Bott_{G^\vr_a}(z)\geq \Bott_{G_a}(z)$ for every $z \in S^1$, where $\Bott_{G^\vr_a}$ is the Bott's function of the symplectic path $D\vr^{G^\vr_a}_t(\hga(0))^{-1}$. Since the proof of Theorem \ref{thm:main} follows from a lower bound of $\Bott_{G_a}$ ($\Bott_{G_a}(1) \geq N(k_a-1)$ for the first assertion, $\Bott_{G_a}(z) \geq Nh_a$ and $\Bott_{G_a}(z) \geq N\th_a$ for some $z\in S^1\setminus\{1\}$ for the second assertion and $\Bott_{G_a}(z) \geq N(k_a+n)$ for some $z\in S^1\setminus\{1\}$ for the third assertion) we have the same bound for $\Bott_{G^\vr_a}$, concluding the desired result.

\section{Proof of Theorem \ref{thm:multiplicity non-hyp}}
\label{sec:proof multiplicity non-hyp}

Let $M=L^{2n+1}_p(1,\dots,1)$ and let $a \in \pi_1(M)$ be the homotopy class of a simple orbit $\ga$ of the obvious contact form $\alpha_0$ on $M$ whose Reeb flow generates a free circle action (induced from the contact form on $S^{2n+1}$ given by the Liouville form restricted to the unit sphere). Clearly $M$ and $\alpha_0$ satisfy the hypotheses of Section \ref{sec:LS} used to conclude \eqref{eq:ESH prequantization}. An easy computation shows that
\[
\cz(\ga^k)=(2n+2)k/p-n
\]
for every $k \in \N$. Therefore, we have from \eqref{eq:ESH prequantization} that
\[
\HC^a_*(M) \cong \oplus_{k\in \N}\H_{*-(2n+2)((k-1)p+1)/p+n}(\CP^n;\Q).
\]
From Example \ref{ex:h_a}, we have that $h_a=(2n+2)/p$. Thus, we have precisely $\lfloor (n+1)/2 \rfloor$ non-trivial elements in $\HC^a_*(M)$ with degrees less than $h_a$.

Now, let $\alpha$ be any contact form on $M$ satisfying the hypothesis of the theorem (in particular, $\alpha$ is convex). From the discussion in Section \ref{sec:LS}, we have an injective map
\[
\psi: \{0,\dots,\lfloor (n-1)/2 \rfloor\} \to \P^a(\alpha)
\]
such that if $\ga_i=\psi(i)$ then $\HC_{(2n+2)/p-n+2i}(\ga_i)\neq 0$. Since $\HC_*(\ga_i)$ is supported in $[\cz(\ga_i),\cz(\ga_i)+\nu(\ga_i)]$, we conclude that $\cz(\ga_i)\leq (2n+2)/p-n + 2\lfloor (n-1)/2 \rfloor \leq (2n+2)/p-1 < h_a$. Therefore, by Theorem \ref{thm:main}, we conclude the existence of $\lfloor (n-1)/2 \rfloor+1=\lfloor (n+1)/2 \rfloor$ non-hyperbolic periodic orbits. Now, we will use our pinching condition to ensure that these closed orbits are simple:

\begin{lemma}
\label{lemma:simple}
Under our pinching conditions the periodic orbits $\ga_i=\psi(i)$ are simple in the sense that they are not an iterate of another periodic orbit with homotopy class $a$. In particular, the orbits $\ga_i$ are geometrically distinct.
\end{lemma}

\begin{proof}
Let $\ga$ be a periodic orbit of $\alpha$ with homotopy class $a$ and period $T$. The (second or more) iterates of $\ga$ with homotopy class $a$ are given by $\ga^{kp+1}$ with $k \in \N$. Since $\cz(\ga_i)<h_a$, it is enough to show that $\cz(\ga^{kp+1})\geq n+2$ because $h_a=(2n+2)/p<n+2$ since $p\geq 2$. Let $\beta$ be the lift of $\alpha$ to $S^{2n+1}$ and $H_\beta: \R^{2n+2} \to \R$ the homogeneous of degree two Hamiltonian such that $H_\beta^{-1}(1)=\Sigma_\beta$. 

Let $\Ga_\beta: [0,(kp+1)T] \to \Sp((2n+2)N)$ be the $N$ copies of the linearized Hamiltonian flow of $H_\beta$ along a lift $\hga$ of $\ga^{kp+1}$ as discussed in Section \ref{sec:computations} and defined in \eqref{eq:Ga_beta}. By Proposition \ref{prop:index} we have that
\begin{equation}
\label{eq:index iterate}
\cz(\ga^{kp+1})=\frac{\cz(\vr^{G_a}_{-t/(kp+1)T}\circ\Ga_\beta)}{N}+1.
\end{equation}
We claim that the symplectic path $\Ga(t)=(\vr^{G_a}_{-t/(kp+1)T}\circ\Ga_\beta)(t)$ is positive. As explained in Section \ref{sec:proof main}, $\Ga$ satisfies the differential equation
\[
\frac{d}{dt}\Ga(t)=JA(t)\Ga(t)
\]
with
\[
A(t)=-\frac{\Hess G_a}{(kp+1)T}+(\vr^{G_a}_{t/(kp+1)T})^*\big(\oplus_1^N\Hess H_\beta(\hga(t))\big)\vr^{G_a}_{t/(kp+1)T}.
\]
From our pinching condition, $\Hess H_\beta \geq \frac{\Id_{2n}}{R^2}$ and therefore 
\[
A(t)\geq-\frac{\Hess G_a}{(kp+1)T}+\frac{\Id_{2nN}}{R^2},
\]
where in the last equation we used that $\vr^{G_a}_t$ is unitary. The right hand side of the last inequality is positive if and only if
\begin{equation}
\label{eq:positive0}
\frac{\|w\|^2}{R^2} > \frac{\lg\Hess G_aw,w\rg}{(kp+1)T}
\end{equation}
for every $w\in \R^{(2n+2)N}$. From \eqref{eq:G_a} and the fact that $\ell_i^a=1$ for every $i$ we conclude that the eigenvalues of $\Hess G_a$ are $2\pi/p$ (with multiplicity $2N(n+1)-2$) and $-2\pi N(n+1)$ (with multiplicity 2). Hence,
\begin{equation}
\label{eq:positive1}
\frac{\lg\Hess G_aw,w\rg}{(kp+1)T} \leq \frac{2\pi\|w\|^2}{(kp+1)Tp}.
\end{equation}
Now, we claim that
\begin{equation}
\label{eq:positive2}
T\geq 2\pi r^2/p.
\end{equation}
Indeed, let $H_r: \R^{2n+2} \to \R$ be the Hamiltonian given by
\[
H_r(x)=\frac{1}{2r^2}\|x\|^2.
\]
The inequality $\Hess H_\beta(x)(v,v) \leq r^{-2}\|v\|^2$ applied to $v=x$ implies that
\[
H_\beta(x) \leq H_r(x).
\]
(Indeed, note that, by homogeneity, $H_\beta(x)= \frac 12\lg \Hess H(x)x,x \rg$.) A theorem due to Croke and Weinstein \cite[Theorem A]{CW} establishes that if $H: \R^{2n+2} \to \R$ is a strictly convex Hamiltonian homogeneous of degree two such that $H(x) \leq H_r(x)$ then every non-constant periodic solution of $H$ has period at least $2\pi r^2$. Thus, since the lift of $\ga^p$ (with period $pT$) is a periodic orbit of $H_\beta$ (note that $\ga^p$ is contractible), we conclude that $pT\geq 2\pi r^2$, as desired.

Consequently, from \eqref{eq:positive1} and  \eqref{eq:positive2} we arrive at
\begin{equation}
\label{eq:positive3}
\frac{\lg\Hess G_aw,w\rg}{(kp+1)T} \leq \frac{\|w\|^2}{r^2(kp+1)}.
\end{equation}
On the other hand,
\begin{equation}
\label{eq:positive4}
\frac{\|w\|^2}{R^2} > \frac{\|w\|^2}{r^2(kp+1)}
\end{equation}
because $\frac{r^2}{R^2} > \frac{1}{kp+1} \iff \frac{R^2}{r^2} < kp+1$ which is a consequence of the inequality $\frac{R}{r} < \sqrt{p+1}$ since $k\geq 1$. From \eqref{eq:positive3} and  \eqref{eq:positive4} we conclude \eqref{eq:positive0}, showing that $\Ga$ is positive.

From the positivity of $\Ga$, we have that $\cz(\Ga) \geq (n+1)N$. Therefore, from \eqref{eq:index iterate},
\[
\cz(\ga^{kp+1}) = \cz(\Ga)/N+1 \geq n+2
\]
as desired.
 \end{proof}

\section{Proof of Theorem \ref{thm:multiplicity}}
\label{sec:proof multiplicity}

Let $\alpha_1$ be the contact form on $L^{2n+1}_p(1,\dots,1)$ induced by the Liouville form $\lambda$ restricted to the unit sphere $S^{2n+1} \subset \R^{2n+2}$. Let $\alpha_t=t^2\alpha_1$ for $t\in[r,R]$. Note that our pinching condition means that $\alpha_r < \alpha < \alpha_R$. Indeed, write $\alpha=f\lambda|_{S^{2n+1}}$ for a positive function $f: S^{2n+1}\to\R$, and note that our pinching condition means that $r\leq \|x\|\leq R$ for every $x\in\Sigma_\beta$. But $x\in\Sigma_\beta$ if and only if $x=\sqrt{f(x/\|x\|)}x/\|x\|$ and therefore our pinching condition is equivalent to the inequality $r^2 \leq f(v) \leq R^2$ for every $v \in S^{2n+1}$.

It is easy to see that the Reeb flow of $\alpha_t$ generates a free circle action with period $\pi t^2/p$. Let $a$ be the homotopy class of the (simple) orbits of $\alpha_t$ and note that
\[
\A^a(\alpha_t)=\{\pi ((k-1)p+1)t^2/p;\,k \in \N\}.
\]
Choose $\ep>0$ sufficiently small such that $\pi R^2/p+\ep \notin \A(\alpha_R)$ and, by \eqref{eq:ESH prequantization action},
\begin{equation}
\label{eq:ESH alpha_R}
\HC^{a,\pi R^2/p+\ep}_*(\alpha_R) \cong \H_{*-\cz(\ga)}(\CP^n;\Q),
\end{equation}
where $\ga$ is a simple orbit of $\alpha_R$ (an easy computation shows that $\cz(\ga)=(2n+2)/p-n$). From the discussion in Section \ref{sec:LS}, we conclude that this isomorphism is equivariant with respect to the shift operators $D$ and $\Delta$ and therefore there exist non-zero elements $w_i \in \HC^{a,\pi R^2/p+\ep}_{(2n+2)/p-n+2i}(\alpha_0)$, $i=0,\dots,n$, such that $Dw_{i+1}=w_i$ ($\Delta: H_*(\CP^n;\Q) \to \H_{*-2}(\CP^n;\Q)$ is an isomorphism).

We claim that $\ep$ can be chosen such that $\pi R^2/p+\ep \notin \A^a(\alpha_t)$ for every $t\in[r,R]$. As a matter of fact, note that if $k\geq 2$ then
\[
\frac{\pi ((k-1)p+1)t^2}{p} \geq \frac{\pi(p+1)r^2}{p} > \frac{\pi R^2}{p}
\]
for every $t\in[r,R]$, where the last inequality holds because $(p+1)r^2 > R^2 \iff \frac{R^2}{r^2}<p+1$. Thus, it is enough to take $\ep>0$ such that 
\[
\frac{\pi ((k-1)p+1)t^2}{p} > \frac{\pi R^2}{p}+\ep
\]
for every $t\in[r,R]$ and $k\geq 2$.

Thus, the continuation map
\[
\phi_{\alpha_R,\alpha_r}: \HC^{a,\pi R^2/p+\ep}_*(\alpha_R) \to \HC^{a,\pi R^2/p+\ep}_*(\alpha_r)
\]
is an isomorphism. Take $\ep$ such that $\pi R^2/p+\ep \notin \A^a(\alpha)$. From the commutative diagram (see \eqref{eq:triangle})
\begin{equation*}
\xymatrix{\HC^{a,\pi R^2/p+\ep}_*(\alpha_R) \ar[rr]^{\phi_{\alpha_R,\alpha_r}}
\ar[dr]_{\phi_{\alpha_R,\alpha}} && HC^{a,\pi R^2/p+\ep}_*(\alpha_r).\\
& \HC^{a,\pi R^2/p+\ep}_*(\alpha) \ar[ur]_{\phi_{\alpha,\alpha_r}}&}
\end{equation*}
we infer that $\phi_{\alpha_R,\alpha}$ is injective. Hence from the discussion in Section \ref{sec:LS} and the equivariant isomorphism \eqref{eq:ESH alpha_R} we conclude that there exists an injective map (see \eqref{eq:filtered carrier map})
\[
\psi^T: \{0,\dots,n\} \to \P^{a,\pi R^2/p+\ep}(\alpha),
\]
such that if $\ga_i=\psi(i)$ then $A(\ga_i)=c_{\phi_{\alpha_R,\alpha}(w_i)}(\alpha)$ and $\HC_{(2n+2)/p-n+2i}(\ga_i)\neq 0$.

We claim that every $\ga_i$ is simple in its homotopy class in the sense that it cannot be an iterate of a periodic orbit with homotopy class $a$. Indeed, let $\beta$ be the lift of $\alpha$ to $S^{2n+1}$. From the inequality $H_\beta(x) \leq r^{-2}\|x\|^2$ and the strict convexity of $H_\beta$ we conclude from  \cite[Theorem A]{CW} that every periodic orbit $\alpha$ has period bigger than or equal to $\pi r^2/p$ (c.f. the proof of Lemma \ref{lemma:simple}). Arguing by contradiction, assume that $\ga_i=\ga^k$ for some closed orbit $\ga$ with homotopy class $a$ and $k\geq 2$. Then $k=jp+1$ for some $j\geq 1$. Since $A(\ga)\geq \pi r^2/p$, we arrive at
\[
A(\ga_i) = (jp+1)A(\ga) \geq (p+1)\frac{\pi r^2}{p} > \pi R^2/p+\ep, 
\]
contradicting the fact that $\ga_i \in  \P^{a,\pi R^2/p+\ep}(\alpha)$. Consequently, the periodic orbits $\ga_i$ are geometrically distinct.


\begin{thebibliography}{aaa}

\bibitem{AM1} M. Abreu and L. Macarini, {\em Contact homology of good toric contact manifolds.} Compositio Mathematica {\bf 148} (2012), 304--334.

\bibitem{AM2} M. Abreu, L. Macarini, \emph{Dynamical convexity and elliptic periodic orbits for Reeb flows.}  Math. Ann. \textbf{369} (2017), 331--386.

\bibitem{AMM} M. Abreu, L. Macarini, M. Moreira, {\em On contact invariants of non-simply connected Gorenstein toric contact manifolds.} Preprint  arXiv:1812.10361, 2018. To appear in Mathematical Research Letters.

\bibitem{AHG} P. Albers, D. Hein, J. Gutt, {\em Periodic Reeb orbits on prequantization bundles.} J. Mod. Dyn. {\bf 12} (2018), 123--150.

\bibitem{Ar} M. Arnaud, {\it Existence d'orbites p\'eriodiques compl\`etement elliptiques des Hamiltoniens convexes pr\'esentant certaines sym\'etries.} C. R. Acad. Sci. Paris. {\bf 328} (1999), 1035--1038.

\bibitem{BO10} F. Bourgeois, A. Oancea, {\em Fredholm theory and transversality for the parametrized and for the $S^1$-invariant symplectic action.} J. Eur. Math. Soc. (JEMS) {\bf 12} (2010), no. 5, 1181--1229.

\bibitem{BO13a} F. Bourgeois, A. Oancea, {\em The index of Floer moduli problems for parametrized action functionals.} Geom. Dedicata {\bf 165} (2013), 5--24.

\bibitem{BO13b} F. Bourgeois, A. Oancea, {\em The Gysin exact sequence for $S^1$-equivariant symplectic homology.} J. Topol. Anal. {\bf 5} (2013), no. 4, 361--407.

\bibitem{BO17} F. Bourgeois, A. Oancea, {\em $S^1$-equivariant symplectic homology and linearized contact homology.} Int. Math. Res. Not. IMRN 2017, no. 13, 3849--3937. 

\bibitem{CE} J. Chaidez, O. Edtmair, {\em 3d convex contact forms and the Ruelle invariant.} Preprint arXiv:2012.12869, 2020.

\bibitem{CGH}  D. Cristofaro-Gardiner, M. Hutchings, {\em From one Reeb orbit to two.} J. Differential Geom. {\bf 102} (2016), no. 1, 25--36.

\bibitem{CW} C. Croke, A. Weinstein, {\it Closed curves on convex hypersurfaces and periods of nonlinear oscillations.} Invent. Math. {\bf 64} (1981), no. 2, 199--202. 

\bibitem{DDE}  G. Dell'Antonio, B. D'Onofrio, I. Ekeland, {\it Periodic solutions of elliptic type for strongly nonlinear Hamiltonian systems.} The Floer memorial volume, 327--333, Progr. Math., 133, Birkh\"auser, Basel, 1995.

\bibitem{DL} H. Duan, H. Liu, {\em Multiplicity and ellipticity of closed characteristics on compact star-shaped hypersurfaces in $\R^{2n}$.}  Calc. Var. Partial Differential Equations 56 (2017), no. 3, Art. 65, 30 pp.

\bibitem{Eke} I. Ekeland, {Convexity methods in Hamiltonian mechanics.} Ergebnisse der Mathematik und ihrer Grenzgebiete (3), 19. Springer-Verlag, Berlin, 1990.

\bibitem{EL} I. Ekeland, J. Lasry, {\em On the number of periodic trajectories for a Hamiltonian flow on a convex energy surface.} Ann. of Math. {\bf 112} (1980), no. 2, 283--319. 

\bibitem{GHHM} V. Ginzburg, D. Hein, U. Hryniewicz, L. Macarini, {\em Closed Reeb orbits on the sphere and symplectically degenerate maxima.} Acta Math. Vietnam. {\bf 38} (2013), no. 1, 55-78.

\bibitem{GG09} V. Ginzburg, B. G\"urel, {\em Periodic orbits of twisted geodesic flows and the Weinstein--Moser theorem.} Comment. Math. Helv. {\bf 84} (2009), 865--907.

\bibitem{GG} V. Ginzburg, B. G\"urel, {\em Lusternik--Schnirelmann theory and closed Reeb orbits.}  Math. Z. {\bf 295} (2020), 515--582. 
  
\bibitem{GM} V. Ginzburg, L. Macarini,  \emph{Dynamical convexity and closed orbits on symmetric spheres.} Duke Math. J. {\bf 170} (2021), 1201--1250.
 
\bibitem{HWZ} H. Hofer, K. Wysocki and E. Zehnder. \textit{The dynamics of strictly convex energy surfaces in $\R^4$.} Ann. of Math. {\bf 148} (1998), 197--289.
  
\bibitem{HM} U. Hryniewicz, L. Macarini, {\em Local contact homology and applications.} J. Topol.\ Anal.\ {\bf 7} (2015), 167--238.

\bibitem{Ke} E. Kerman, {\em Rigid constellations of closed Reeb orbits.} Compos. Math. {\bf 153} (2017), no. 11, 2394--2444.

\bibitem{Le} E. Lerman, {\em Contact toric manifolds}, J. Symplectic Geom. {\bf 1} (2003), 785--828.

\bibitem{LL} H. Liu, Y. Long, {\em The existence of two closed characteristics on every compact star-shaped hypersurface in $\R^4$.} Acta Math. Sin. (Engl. Ser.) {\bf 32} (2016), no. 1, 40--53.

\bibitem{LLZ} C. Liu, Y. Long, C. Zhu, {\em Multiplicity of closed characteristics on symmetric convex hypersurfaces in $\R^{2n}$} Math. Ann. {\bf 323} (2002), 201--215.

\bibitem{LWZ} H. Liu, C. Wang, D. Zhang, {\em Elliptic and non-hyperbolic closed characteristics on compact convex P-cyclic symmetric hypersurfaces in $\R^{2n}$.}  Calc. Var. Partial Differential Equations 59 (2020), no. 1, Paper No. 24, 20 pp. 
 
\bibitem{Lon99} Y. Long. \textit{Bott formula of the {M}aslov-type index theory.} Pacific J. Math. {\bf 187} (1999), 113--149.

\bibitem{Lon02} Y. Long, {Index Theory for Symplectic Paths with Applications.} Birkh\"auser,  Basel, 2002.

\bibitem{LZ} Y. Long, C. Zhu, {\em Closed characteristics on compact convex hypersurfaces in $\R^{2n}$.} Ann. of Math. {\bf 155} (2002), no. 2, 317--368.

\bibitem{McL} M. McLean, {\em Reeb orbits and the minimal discrepancy of an isolated singularity.} Invent. Math. {\bf 204} (2016), no. 2, 505--594.

\bibitem{Sei} P. Seidel, {\em A biased view of symplectic cohomology.} In Current developments in mathematics, 2006, 211--253. Int. Press, Somerville, MA, 2008.

\bibitem{Vit}  C. Viterbo, {\em Functors and computations in Floer homology with applications. I.} Geom. Funct. Anal. {\bf 9} (1999), no. 5, 985--1033.

\bibitem{Wa} W. Wang, {\em Existence of closed characteristics on compact convex hypersurfaces in  $\R^{2n}$.} Calc. Var. Partial Differential Equations, {\bf 55} (2016), 1--25.

\end{thebibliography}
\end{document}